\documentclass[a4paper,12pt]{report}
\usepackage{amsmath}
\usepackage{amsfonts}
\usepackage[dvips]{epsfig}
\usepackage{subfigure}
\usepackage{latexsym}
\usepackage{amssymb}
\usepackage[dvips]{graphicx}
\usepackage[icelandic,english]{babel}
\usepackage{cmap}
\usepackage[T1]{fontenc}
\usepackage[utf8]{inputenc}

\usepackage{color}
\usepackage{amsthm,amssymb}
\usepackage{enumerate}
\usepackage{hyperref}
\usepackage{fancyhdr}
\usepackage{eucal}
\usepackage{amscd}
\usepackage{xcolor}
\usepackage{textcomp}
\usepackage{booktabs}
\usepackage{psfrag} 

\usepackage{xspace,float}

\newcommand{\C}{\ensuremath{\mathcal{C}}\xspace}
\newcommand{\M}{\ensuremath{\mathcal{M}}\xspace}
\newcommand{\B}{\ensuremath{\mathcal{B}}\xspace}
\newcommand{\T}{\ensuremath{\mathcal{T}}\xspace}
\newcommand{\A}{\ensuremath{\mathcal{A}}\xspace}
\newcommand{\V}{\ensuremath{\mathcal{V}}\xspace}
\newcommand{\R}{\ensuremath{\mathcal{R}}\xspace}
\newcommand{\F}{\ensuremath{\mathcal{F}}\xspace}
\newcommand{\G}{\ensuremath{\mathcal{G}}\xspace}
\renewcommand{\H}{\ensuremath{\mathcal{H}}\xspace}
\newcommand{\K}{\ensuremath{\mathcal{K}}\xspace}
\newcommand{\J}{\ensuremath{\mathcal{J}}\xspace}

\newcommand{\sym}{\ensuremath{\mathfrak{S}}\xspace}
\newcommand{\sympattern}{\ensuremath{\preccurlyeq_{\sym}}\xspace}
\newcommand{\poly}{\ensuremath{\mathfrak{P}}\xspace}
\newcommand{\polypattern}{\ensuremath{\preccurlyeq_{\poly}}\xspace}
\newcommand{\matr}{\ensuremath{\mathfrak{M}}\xspace}

\newcommand{\Av}{Av}
\newcommand{\Avperm}{Av_{\mathfrak{S}}}
\newcommand{\Avp}{Av_{\mathfrak{P}}}
\newcommand{\Avm}{Av_{\mathfrak{M}}}

\newtheorem{theorem}{Theorem}
\newtheorem{definition}[theorem]{Definition}
\newtheorem{remark}[theorem]{Remark}
\newtheorem{proposition}[theorem]{Proposition}
\newtheorem{example}[theorem]{Example}
\newtheorem{corollary}[theorem]{Corollary}
\newtheorem{lemma}[theorem]{Lemma}

\def\ie{{\em i.e.}\xspace}

\graphicspath{{./imagespdf/}}

\begin{document}

\titlepage
\begin{titlepage}

\begin{center}


\begin{center}
\begin{minipage}[c]{.40\textwidth}
\centering
\includegraphics[width=.50\textwidth]{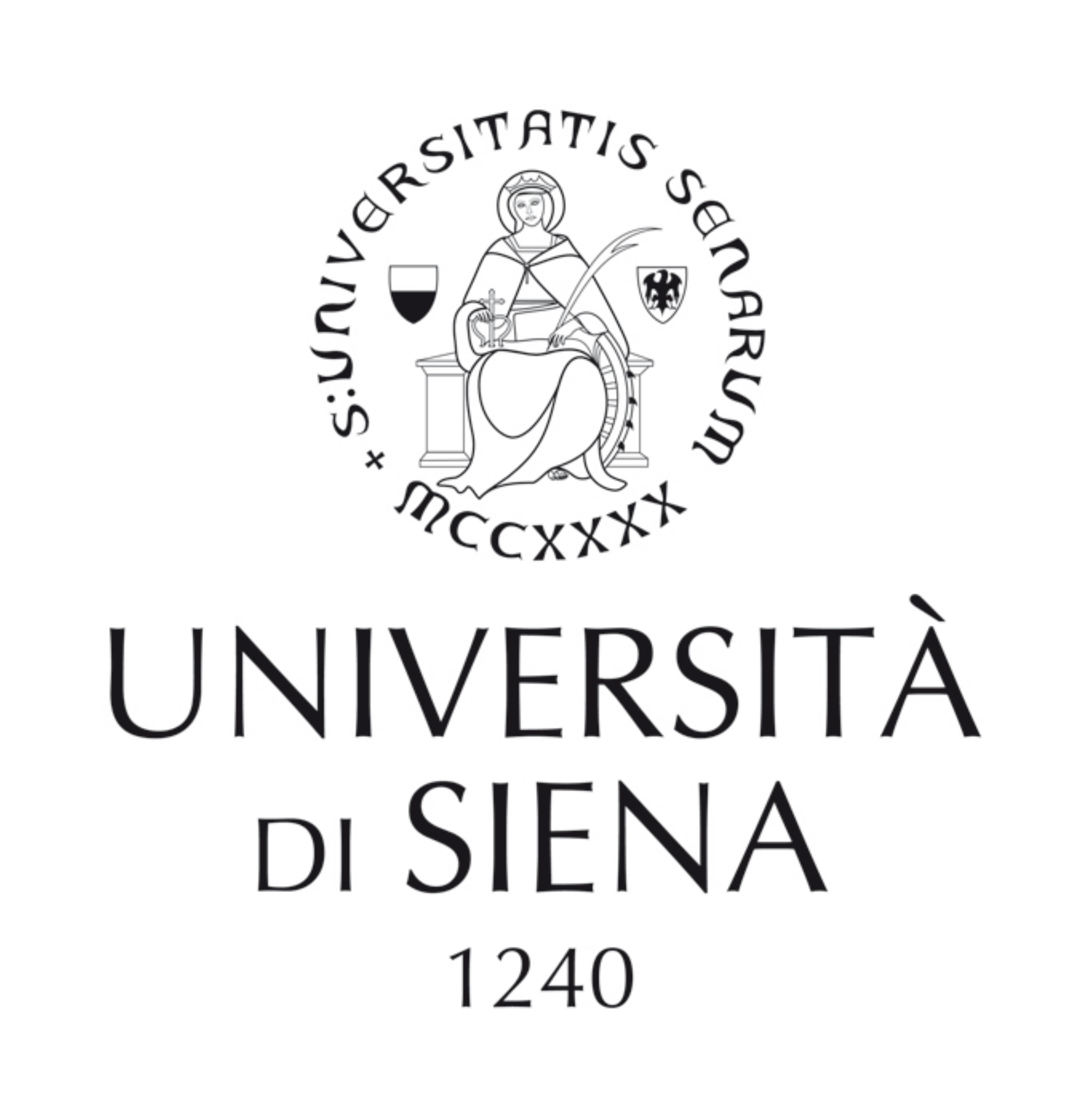}
\end{minipage}
\hspace{3mm}
\begin{minipage}[c]{.40\textwidth}
\centering
\includegraphics[width=.45\textwidth]{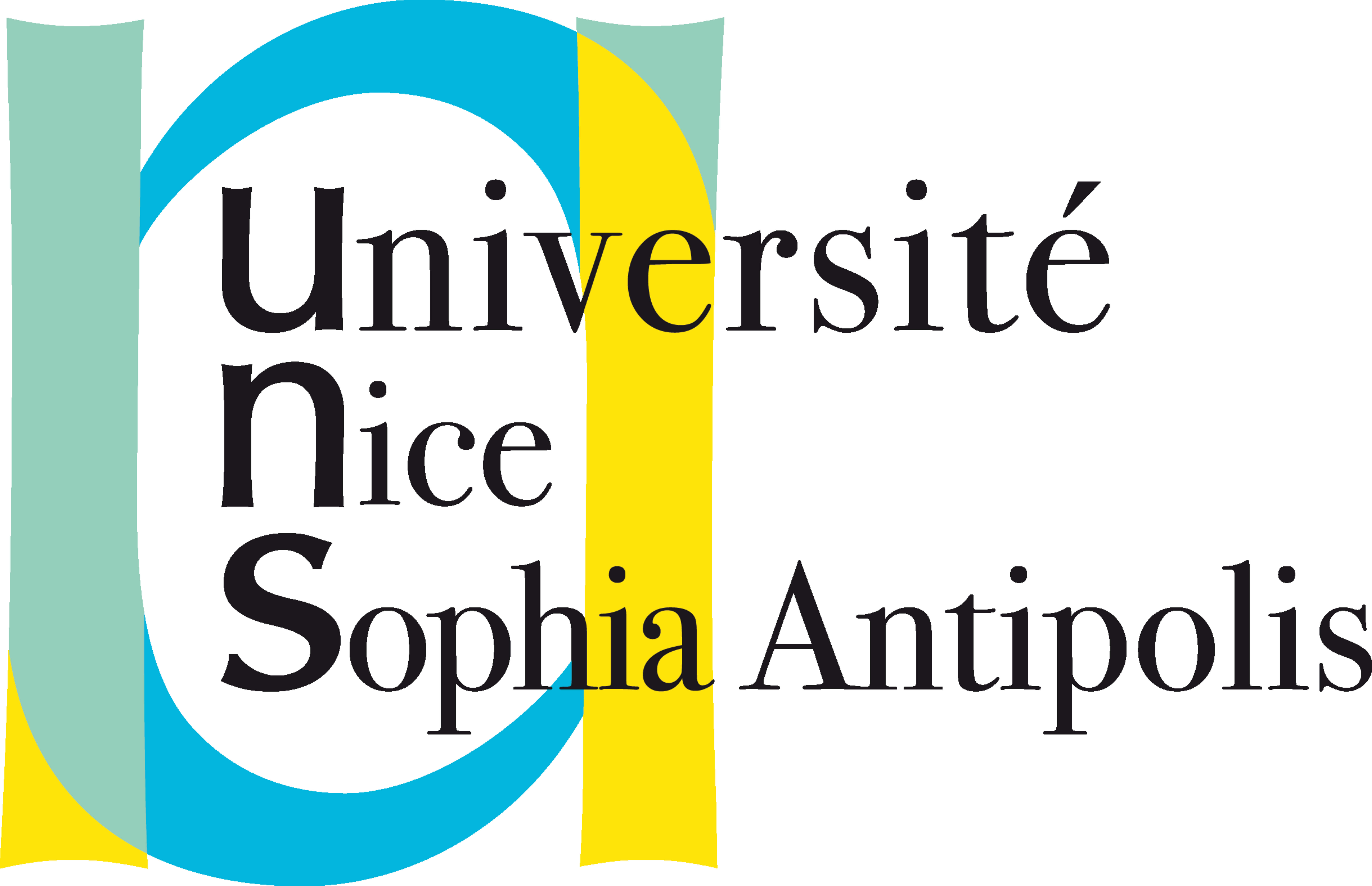}
\end{minipage}
\end{center}

\vspace{5pt}

%
%
%
\begin{LARGE}
\begin{center}
\textbf{Enumeration of polyominoes defined in terms of pattern avoidance or convexity constraints}
\end{center}
\end{LARGE}

\vspace{5pt}
\begin{center}
Thesis of the University of Siena and the University of Nice Sophia Antipolis\\
\vspace{1pt}
Advisors: Prof. Simone Rinaldi and Prof. Jean Marc Fédou
\end{center}

\vspace{10pt}
\begin{center}
 to obtain the
\end{center}

\vspace{5pt}

\begin{center}
 Ph.D. in Mathematical Logic, Informatics and Bioinformatics of the University of Siena\\
 Ph.D. in Information and Communication Sciences of the University of Nice Sophia Antipolis
\end{center}

\vspace{10pt}
Candidate: Daniela Battaglino
\vspace{15pt}





\begin{center}
\begin{minipage}[c]{.40\textwidth}
\textit{Jury composed by:
\vspace{4pt}
\\ Prof. Elena Barcucci
\\ Prof. Marilena Barnabei
\\ Prof. Srecko Brlek
\\ Prof. Enrica Duchi
\\ Prof. Jean Marc Fédou
\\ Prof. Rinaldi Simone}
\end{minipage}
\end{center}


\end{center}
\end{titlepage}

\pagenumbering{roman}

\tableofcontents

\cleardoublepage
\pagenumbering{arabic}

\phantomsection
\addcontentsline{toc}{chapter}{Introduction}

\pagestyle{plain}

\chapter*{Introduction}

This dissertation discusses some topics and applications in combinatorics.

Combinatorics is a branch of mathematics which concerns the study of classes of discrete objects which are often designed as models
for real objects. The motivations for studying these objects may arise from Informatics (models
for data structures, analysis of algorithms, ...), but even from biology - in
particular molecular and evolutive biology \cite{denise} - from physics as in \cite{baxt} or from chemistry \cite{bouvGreb}.

Combinatorialists are particularly interested in several aspects of a class
of objects: its different characterizations, the description of its properties,
the enumeration of its elements, and their generation both
randomly or exhaustively, by the use of algorithms, the definition of some
relations (as for example order relations) between the elements belonging to
the same class. 
We have taken into exam two remarkable subfields of combinatorics, which have often been considered in the literature. These two aspects are strictly related, and they permit us to give a deep insight on the nature of the combinatorial structures which are being studied: enumerative combinatorics and the study of patterns into combinatorial structures.

{\bfseries Enumerative Combinatorics.} An unavoidable step for a profound comprehension of the structure of an object is certainly the capability of counting its elements. Counting can not have an exhaustive definition since is something that flows back to a philosophical difficulty of language and understanding.
To our aim, the main concern of enumerative combinatorics is counting the number of elements of a finite class in an exact or approximate way.  
Various problems arising from different fields can be solved by analysing them from a combinatorial point of view. Usually, these problems have the common feature to be represented by simple
objects suitable to enumerative techniques of combinatorics. Given a class ${\cal O}$ of objects and
a parameter $p$ on this class, called the size, we focus on the set ${\cal O}_n$ of objects for which the value of the
parameter, is equal to $n$, where $n$ is a non negative integer. The parameter $p$ is discriminating if, for each non negative integer $n$, the number of objects of ${\cal O}_n$ is finite. Then, we ask for the cardinality $a_n$ of the set ${\cal O}_n$ for each possible $n$. Enumerative combinatorics answers to this question.
Only in rare cases the answer will be a completely explicit closed formula for $a_n$, involving
only well known functions, and free from summation symbols. However, a recurrence for $a_n$
may be given in terms of previously calculated values $a_k$, thereby giving a simple procedure
for calculating $a_n$ for any $n\in \mathbb{N}$. Another approach is based on generating functions: whether
we do not have a simple formula for $a_n$, we can hope to get one for the formal power series
$f(x) =\varSigma_n a_nx^n$, which is called the generating function of the class ${\cal O}$ according to the
parameter $p$. Notice that the $n$-th coefficient of the Taylor series of $f(x)$ is just the term $a_n$.
In some cases, once that the generating function is known, we can apply standard techniques
in order to obtain the required coefficients $a_n$ (see for instance \cite{GouldJack,GrahamKnuth}). Otherwise we can obtain an asymptotic value of the coefficients through the analysis of the singularities in the generating function (see \cite{flajolet}).

Several methods for the enumeration, using algebraic or analytical tools, have been developed in the last forty years. A first general and empirical approach consists in calculating the first terms of $a_n$ and then try to deduce the sequence. For instance, one can use the book from Sloane and Plouffe \cite{Sl, SlPl} in order to
compare the first numbers of the sequence with some known sequences and try to identify
$a_n$. More advanced techniques (Brak and Guttmann \cite{BrGut}) start from
the first terms of the sequence and find an algebraic or differential equation satisfied by the
generating function of the sequence itself. A more common approach consists in looking for
a construction of the studied class of objects and successively translating it into a recursive
relation or an equation, usually called functional equation, satisfied by the generating function
$f(x)$. The approach to enumeration of combinatorial objects by means of generating functions
has been widely used (see for instance Goulden and Jackson \cite{GouldJack} and Wilf \cite{wilf}).
Another technique which has often been applied to solve combinatorial problems is the Sch\"utzenberger methodology, also called DSV \cite{Schutz}, which can be decomposed into three steps. First construct a bijection between the objects and the words of an algebraic language in such a way that for every object the parameter to the length of the words of the language. At the next step, if the language is generated by an unambiguous
context-free grammar, then it is possible to translate the productions of the grammar into
a system of functional equations. Finally one deduces an equation for which the generating function of the sequence $a_n$ is the unique and algebraic solution (Sch\"utzenberger and Chomsky \cite{sch}). A variant of the DSV methodology are the operator grammars (Cori and Richard \cite{CoRich}). These grammars take in account some cases in which the language encoding the objects is not algebraic.
The theory of decomposable structure (Flajolet, Salvy, and Zimmermann \cite{FlajSalv1,FlajSalv2}), describes recursively the objects in terms of basic operations between them. These operations are directly translated into operations between the corresponding generating functions, cutting off the passage to words. A nice presentation of this theory appears in the book of
Flajolet and Sedgewick \cite{flajolet}. A variant is the theory of species, introduced by Bergeron, Labelle and Leroux \cite{Berg}, which also follows the philosophy of decomposable structures. Basing on the idea of Joyal \cite{Joy}, they define an algebra on species of structures, where the operations between the species immediately reflect on the generating functions.

Finally, a very convenient formalization of the approach of decomposable structures was
introduced by Dutour and Fedou \cite{Dut}. This method is based on the notion of object grammars, and describe objects
using very general kinds of operations.

A significantly different way of recursively describing objects appears in the ECO methodology, introduced by Barcucci, Del Lungo, Pergola, and Pinzani \cite{ECO}. In the ECO method each object is obtained from a smaller object by making some local expansions. Usually these local expansions are very regular and can be described in a simple way by a succession rule. Then a succession rule can be translated into a functional equation for the generating function. It has been shown that this method is very effective on large number of combinatorial structures.
%
Another approach is to find a bijection between the studied class of objects and another one, simpler to count. In order to have consistent enumerative results, the bijection has to preserve the size of the objects. Moreover, a bijective approach also permits a better comprehension of some properties of the studied class and to relate them to the class in bijection with it.

{\bfseries Patterns into combinatorial structures.} A possible strategy to understand more about the nature of some combinatorial structures and which provides a different way to look at a combinatorial object, is to describe it by the containment or avoidance of some given substructures, which are commonly known as patterns.
The concept of pattern within a combinatorial structure is undoubtedly one of the most investigated notions in combinatorics. It has been deeply studied for permutations, starting first with~\cite{Kn}.
More in details, given a permutation $\sigma$ we can say that $\sigma$ contains a certain pattern $\pi$ if such a pattern can be seen as a sort of ``subpermutation'' of $\sigma$. If $\sigma$ does not contain $\pi$ we say that $\sigma$ avoids $\pi$.

In particular, the concept of pattern containment on the set of all permutations can be seen as a partial order relation, and it was used to define permutation classes, \ie families of permutations downward closed under such pattern containment relation. So, every permutation class can be defined in terms of a set of avoided patterns, and the minimal of this sets is called the basis of the class.

These permutation classes can then be regarded as objects to be counted. 
We can find many results concerning this research guideline in the literature. For instance, we quote two works that collect  a large part of the obtained results. The first is the thesis of Guibert\cite{guibert} and the second is the work of Kitaev and Mansour \cite{KitMans}. In the latest, in addition to the list of the obtained results regarding the enumeration of set of permutations that avoid a set of patterns, the author also take into account the study of the number of objects which contains a fixed number of occurrences of a certain pattern and make an interesting parallel between the concept of pattern on the set of permutations and the concept of pattern on the set of words.

As regards the results obtained on the enumeration of classes that avoid patterns of small size, we mention the work of Simion and Schmidt \cite{SimionSchmidt}, in which we can find  an exhaustive study of all cases with patterns of length less than or equal to three. However, for results concerning patterns of size four we refer the reader to the work of Bona \cite{bona}.
One of the most important recent contributions is the one by Marcus and Tardos \cite{marcTard}, consisting in the proof of the so-called Stanley-Wilf conjecture, thus defining an exponential upper bound to the number of permutations avoiding any given pattern.
Later, given the enormous interest in this area, were taken into analysis not only patterns by the classical definition, but also patterns defined under the imposition of some constraints.

Babson and Steingr\'imsson \cite{babStein} introduced the notion of {\em generalized patterns}, which requires that two adjacent letters in a pattern must be adjacent in the permutation. The authors introduced such patterns to classify the family of {\em Mahonian permutation statistics}, which are uniformly distributed with the number of inversions.
Several results on the enumeration of permutations classes avoiding generalized patterns have been achieved. Claesson obtained the enumeration of permutations avoiding a generalized pattern of length three \cite{claesson} and the enumeration of permutations avoiding two generalized patterns of length three \cite{CM}. Another interesting result in terms of permutations avoiding a set of generalized patterns of length three was obtained by Bernini \textit{et al.} in \cite{BBF,BFP}, where one can find the enumeration of permutation avoiding set of generalized patterns as a function of its length and another parameter.

Another kind of patterns, called {\em bivincular patterns}, was introduced in \cite{BMCDK} with the aim to increase the symmetries of the classical patterns. In \cite{BMCDK}, the bijection between permutations avoiding a particular bivincular pattern was derived, as well as several other classes of combinatorial objects. Finally, we mention the {\em mesh patterns}, which were introduced in \cite{mesh} to generalize multiple varieties of permutation patterns.


Otherwise, from the algorithmic point of view, an interesting problem is to find an efficient way to establish if an element belongs to a permutation class \C. More in detail, if we know the elements of the basis of \C, and especially if the basis is finite, this problem consists in verifying if a permutation contains an element of the basis. Generally the complexity of the algorithms is high, but there are some special cases in which linear algorithms have been found, for instance in \cite{Kn}.

Another remarkable problem which has been considered is to calculate the basis of a given class of permutations. A very useful result in this direction was obtained by Albert and Atkinson in \cite{AA}, in which the authors provide a necessary and sufficient condition to ensure that a permutation class has a finite basis. 

As we have previously mentioned, some definitions analogous to those given for permutations were provided in the context of many other combinatorial structures, such as set partitions~\cite{Go,Kl,Sa}, words~\cite{Bj,Bu}, trees~\cite{DPTW,R}, and paths~\cite{ferrari}.

\vspace{1cm}

In the present thesis we examine the two previously quoted general issues,
on a rather remarkable class of combinatorial objects, i.e. the polyominoes. These objects arise in many scientific areas of research, for instance in combinatorics, physics, chemistry,... (more explicit details are given in Chapter \ref{chap:chapter1}).  In particular, in this thesis, we consider under a combinatorial and an enumerative point of view families of polyominoes defined by imposing several types of constraints.

The first type of constraint, which extends the well-known convexity constraint \cite{DV}, is the $k$-convexity constraint, introduced by Castiglione and Restivo \cite{lconv2}. A convex polyomino is said to be $k$-convex if every pair of its cells can be connected by a monotone path with at most $k$ changes of direction. The problem of enumerating $k$-convex polyominoes was solved only for the cases $k=1,2$, while
the case $k>2$ is yet open and seems difficult to solve. To get rid of this problem, we have taken into exam a particular subclass of $k$-convex polyominoes, the $k$-parallelogram polyominoes, \ie the $k$-convex polyominoes that are also parallelogram.



The second type of constraint we are going to consider, extends, in a natural way, the concept of pattern avoidance on the set of polyominoes. Since a polyomino can be represented in terms of binary matrix, we can say that a polyomino $P$ is a pattern of a polyomino $Q$ when the binary matrix representing $P$ is a submatrix of that representing $Q$. We have then attempted at reconsidering the same problems treated within polyomino classes even on the case of patterns avoiding polyominoes. 

Basing on this idea, we have defined a polyomino class to be a set of polyominoes which are downward closed w.r.t. the containment order. Then we have given a characterization to some known families of polyominoes, using this new notion of pattern avoidance.

This new approach also allowed us to study a new definition of permutations that avoid submatrices, and to compare it with the classical notion of pattern avoidance.

In details, the thesis is organized as follows.

Chapter \ref{chap:chapter1} provides the basic definitions of the most important combinatorial structures, which will be considered in the thesis and contains a brief state of the art. There are three main classes of objects we have studied in this work. The first one is the class of $k$-parallelogram polyominoes, which will be studied by an enumerative point of view. The second one is the class of permutations: in particular we will present the concept of patterns avoidance. The third and last class we have focused on is the one of partially ordered sets (p.o.sets or simply posets).

In Chapter \ref{chap:chapter2}, we deal with the problem of enumerating a subclass of $k$-convex polyominoes, the $k$-convex polyominoes which are also {\em parallelogram polyominoes}. More precisely we provide an unambiguous decomposition for the class of the $k$-parallelogram polyominoes, for any $k\geq 1$. Then, we also translate this decomposition in a functional equation for the generating function of $k$-parallelogram polyominoes, for any $k$. We are then able to express such a generating function in terms of the Fibonacci polynomials and thanks to this new expression we find a bijection between the class of $k$-parallelogram polyominoes and the class of planted planar trees having height less than or equal to $k+2$.

In Chapter \ref{chap:cap4} borrowing a known concept already used for several structures, the concept of pattern avoidance, we have found a new characterization of the set of permutations and polyominoes both seen as matrices. In particular, this approach allows us to define these classes of objects as the sets of elements that are downward closed under the pattern relation, that is a partial order relation. We then study the poset of polyominoes, by an algebraic and a combinatorial point of view. Moreover, we introduce several notions of bases, and we study the relations among these. We investigate families of polyominoes which can be described by the avoidance of matrices, and families which are not. In this case, we consider some possible extension of the concept of submatrix avoidance to be able to represent also these families.   


\pagestyle{plain}
\chapter{Polyominoes, permutations and posets}
\label{chap:chapter1}

This thesis studies the combinatorial and enumerative properties of some families of polyominoes, defined in terms of particular constraints of convexity and connectivity. Before we discuss these concepts in-depth, we need to summarise the principal definitions and classifications of polyominoes. More specifically, we introduce the notions of {\em polyomino, permutation} and {\em posets} (partially ordered set). The chapter is organised as follows. In Section~\ref{sec:cap1_pol} we briefly introduce the history of polyominoes; in Section~\ref{sec:polyFamilies} we discuss some of the most important families of polyominoes; in Section~\ref{sec:permutations} we focus on permutations; Section~\ref{sec:posets} concludes the chapter by discussing posets.

\section{Polyominoes}
\label{sec:cap1_pol}

The enumeration of {\em polyominoes} on a regular lattice is without any doubt one of the most studied topics in Combinatorics. The term polyomino was introduced by Golomb in $1953$ during a talk at the Harvard Mathematics Club (which was published one year later \cite{gol}) and popularized by Gardner in $1957$ \cite{gardner}. A polyomino is defined as follows.
\begin{definition}
In the plane $\mathbb{Z} \times\mathbb{Z}$ a cell is a unit square and a polyomino is a finite connected union of cells having no cut point. 
\end{definition}
Polyominoes are defined up to translations. Polyominoes can be similarly defined in other two-dimensional lattices (e.g. triangular or honeycomb); however, in this work we will focus exclusively on the square lattice.

A {\em column (resp. row)} of a polyomino is the intersection between the polyomino and an infinite strip of cells whose centers lie on a vertical (resp. horizontal) line. A polyomino is characterised by four parameters: {\em area}, {\em width}, {\em height} and {\em perimeter}. The area is the number of elementary cells of the polyomino; the width and height are respectively the number of columns and rows; the perimeter is the length of the polyomino's boundary.

As we already observed, polyominoes have been studied for a long time in Combinatorics, but they have also drawn the attention of physicists and chemists. The former in particular established a relationships with polyominoes by defining equivalent objects named {\em animals} \cite{dhar,haknad}, obtained by taking the center of the cells of a polyomino as shown in Figure~\ref{fig:animal}. These models allowed to simplify the description of phenomena like phase transitions (Temperley, $1956$ \cite{temperley}) or percolation (Hammersely, \cite{hamm}).

\begin{figure}[htd]
\centering
\includegraphics[width=8cm]{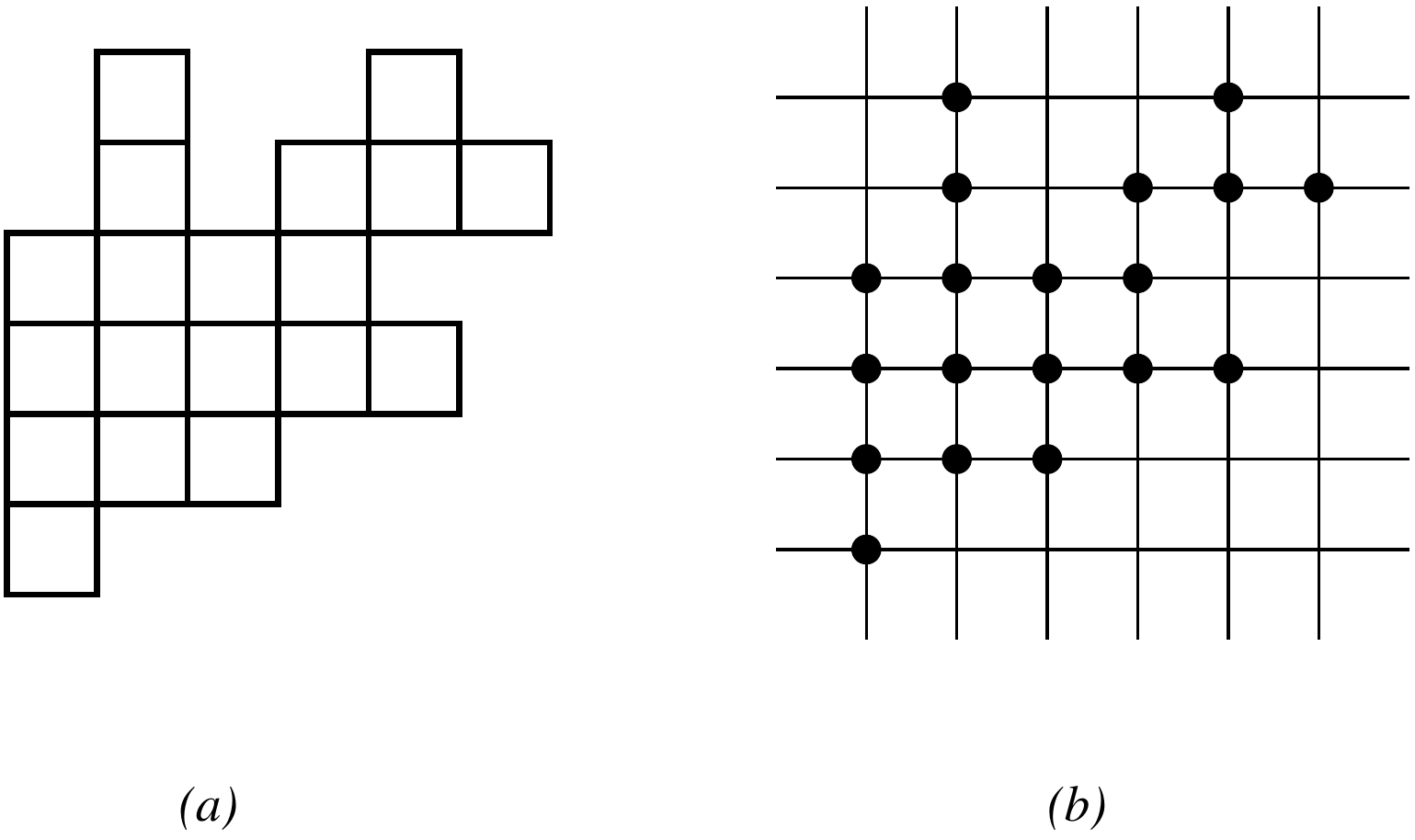}
\caption{A polyomino in $(a)$ and the corresponding animal in $(b)$.}
\label{fig:animal}
\end{figure}

Other important problems concerned with polyominoes are the problem
of covering a polyomino with rectangles \cite{chai} or problems of tiling regions by polyominoes
\cite{beniv,conway}.

In this work we are mostly interested in the problem enumerating polyominoes with respect to the area or perimeter. Several important results were obtained in the past in this field. For example, in \cite{klarner} Klarner proved that, given $a_n$ polyominoes of area $n$, the limit $${\lim_{n \to \infty}} a_n^{\frac{1}{n}}$$ tends to a growth constant $\mu$ such that: $$3.72 < \mu < 4.64\,.$$
Moreover, in $1995$ Conway and Guttmann \cite{congut} adapted a method previously used for polygons to calculate $a_n$ for $n\leq 25$. Further refinements by Jensen and Guttman \cite{jensengut} and Jensen \cite{jens} allowed to reach respectively $n=46$ and $n=56$. Despite these important results, the enumeration of general polyominoes still represents an open problem whose solution is not trivial but can be simplified, at least for certain families of polyominoes, by introducing some constraints such as convexity and directedness.

\section{Some families of polyominoes}
\label{sec:polyFamilies}

In this section we briefly summarize the basic definitions concerning some families of convex polyominoes. More specifically, we focus on
the enumeration with respect to the number of columns and/or rows, to the semi-perimeter and to the area. Given a polyomino $P$ we denote with:
\begin{enumerate}
 \item $A(P)$ the area of $P$ and with $q$ the corresponding variable;
 \item $p(P)$ the semi-perimeter of $P$ and with $t$ the corresponding variable;
 \item $w(P)$ the number of columns (width) of $P$ and with $x$ the corresponding variable;
 \item $h(P)$ the number of rows (height) of $P$ and with $y$ the corresponding variable.
\end{enumerate}

\begin{definition}
A polyomino is said to be {\em column-convex (row-convex)}
when its intersection with any vertical (horizontal) line is convex. 
\end{definition}
An example of column-convex and row-convex polyominoes are provided in Figure~\ref{fig:RCconvex} $(a)$ and $(b)$.

\begin{figure}[htd]
  \begin{center}
  \includegraphics[width=10cm]{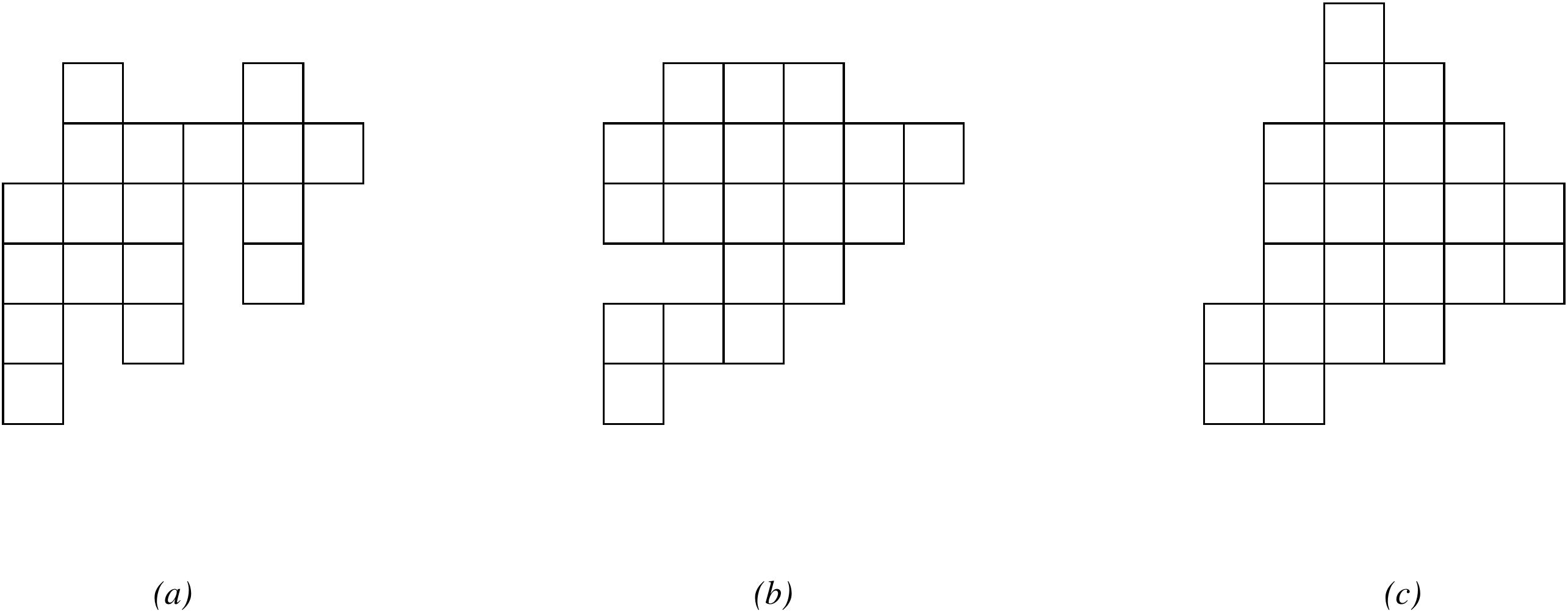}
  \caption{$(a)$: A column-convex polyomino; $(b)$: A row-convex polyomino; $(c)$: A convex polyomino.}
  \label{fig:RCconvex}
  \end{center}
  \end{figure}

In \cite{temperley}, Temperley proved that the generating function of column-convex polyominoes with respect to the perimeter is algebraic and found the following generating function according to the number of columns and to the area: 
\begin{equation}
f(x,q)=\frac{xq(1-q)^3}{(1-q)^4-xq(1-q)^2(1+q)-x^2q^3}\,.
\end{equation}

Inspired by this work, similar results were obtained in $1964$ by Klarner \cite{klarner} and in $1988$ by Delest \cite{delest}. The former was able to define the generating function of column-convex polyominoes according to the area, by means of a combinatorial interpretation of a Fredholm integral; the latter derived the expression for the generating function of column-convex polyominoes as a function of the area and the number of columns, by means of the Sch\"utzemberger methodology \cite{sch}.

In the same years, Delest \cite{delest}  derived the generating function for column-convex polyominoes according to the semi-perimeter by means of context-free languages and the computer software for algebra MACSYMA\footnote{Macsyma (Project MAC's SYmbolic MAnipulator) is a computer algebra system that was originally developed from 1968 to 1982.}. Such function is defined as follows:
\begin{equation}
f(t)=(1-t)\left( 1-\frac{2\sqrt{2}}{3\sqrt{2}-\sqrt{1+t+\sqrt{\frac{(t^2-6t+1)(1+t)^2}{(1-t)^2}}}}\right).
\label{eq:genFunVsSempiPer}
\end{equation}

In Equation~(\ref{eq:genFunVsSempiPer}), the number of column-convex polyominoes with semi-perimeter $n+2$ is the coefficient of $t^n$ in $f(t)$; it is worth noting that such coefficients are an instance of sequence $A005435$ \cite{Sl}, whose first few terms are:
$$1, 2, 7, 28, 122, 558, 2641, 12822, \cdots\,\,,$$
and they count, for example, the number of permutations avoiding $13-2$ that contain the pattern $23-1$ exactly twice, but there is no a combinatorial explanation of this fact.

Several studies were carried out in the attempt to improve the above formulation or to obtain a closed expression not relying on software, including: a generalization by Lin and Chang \cite{changLing}; an alternative proof by Fereti\'c \cite{feretic}; an equivalent result obtained by means of Temperley's methodology and the Mathematica software\footnote{See \url{www.wolfram.com/mathematica}.} by Brak \textit{et al.} \cite{BEG}.

\begin{definition}
A polyomino is  \em{convex} if it is both column and row convex (see Figure~\ref{fig:RCconvex} $(c)$). 
\end{definition}

It is worth noting that the semi-perimeter of a convex polyomino is equivalent to the sum of its rows and columns. \\

Bousquet-M\'elou derived several expressions for the generating function of convex polyominoes according to the area, the number of rows and columns, among which we mention the one obtained in collaboration with Fedou \cite{BMF} and the one in \cite{mbm}.

The generating function for convex polyominoes indexed by semi-perimeter obtained by Delest and Viennot in $1984$ \cite{DV} is the following:
\begin{equation}
f(t)=\frac{t^2(1-8t+21t^2-19t^3+4t^4)}{(1-2t)(1-4t)^2}-\frac{2t^4}{(1-4t)\sqrt{1-4t}}\,.
\end{equation}

The above expression is obtained by differencing two series with positive terms, whose combinatorial interpretation was given by Bousquet-M\'elou and Guttmann in \cite{BMG}.
The closed formula for the convex polyominoes is:
\begin{equation}
f_{n+2} = (2n + 11)4^n - 4(2n + 1)\binom{2n}{n},
\end{equation}
with $n \geq 0$, $f_0 = 1$ and $f_1 = 2$. Note that this is an instance of sequence $A005436$ \cite{Sl}, whose first few terms are:
$$1, 2, 7, 28, 120, 528, 2344, 10416, \cdots .$$

In \cite{changLing}, Lin and Chang derived the generating function for the number of convex polyominoes with $k + 1$ columns and $j + 1$ rows, where $k, j \geq 0$. Starting from their work, Gessel \cite{gessel} was able to infer that the number of such polyominoes is:
\begin{equation}
\frac{k+j+kj}{k+j}\binom{2k+2j}{2k}-2(k+j)\binom{k+j-1}{k}\binom{k+j-1}{j}\,.
\end{equation}

Finally, in \cite{DlDFR} the authors defined the generating function of convex polyominoes according to the semi-perimeter using the ECO method \cite{ECO}.

\begin{definition}
A polyomino $P$ is said to be {\em directed convex} when every cell of $P$ can be reached
from a distinguished cell, called {\em source} (usually the leftmost cell at the lowest
ordinate), by a path which is contained in $P$ and uses only north and east unit
steps. 
\end{definition}

An example of a directed convex polyomino is depicted in Figure~\ref{fig:FSPD} $(d)$.


The number of directed convex polyominoes with semi-perimeter $n + 2$ is equal to $b_{n-2}$, where $b_n$ are the central binomial coefficients:
$$b_n =\binom{2n}{n},$$
giving an instance of sequence $A000984$ \cite{Sl}.

The enumeration with respect to the semi-perimeter of this set was first obtained by Lin and Chang in $1988$ \cite{changLing} as follows:
\begin{equation}
f(t)=\frac{t^2}{\sqrt{1-4t}}\,.
\end{equation}

Furthermore, the generating function of directed convex polyominoes according to the area and the number of columns and rows, was derived by M. Bousquet-M\'elou and X. G. Viennot \cite{BMV}:
\begin{equation}
f(x,y,q)=y\frac{M_1}{J_0}
\end{equation}
where
\begin{equation}
M_1=\sum_{n\geq1}\frac{x^nq^n}{(yq)_n}\sum_{m=0}^{n-1}\frac{(-1)^mq^{\binom{m}{2}}}{(q)_m(yq^{m+1})_{n-m-1}}
\end{equation}
and
\begin{equation}
 J_0=\sum_{n\geq 0}\frac{(-1)^nx^nq^{\binom{n+1}{2}}}{(q)_n(yq)_n}\,,
\label{eq:J0}
\end{equation}
with $(a)_n=(a;q)_n=\prod_{i=0}^{n-1}(1-aq^i)$.

\subsubsection*{Definitions of polyominoes according to cells}

It is also possible to discriminate between different families of polyominoes by looking at the sets of cells $A, B, C$ and $D$ individuated by a convex polyomino and its {\em minimal bounding rectangle}, \ie the minimum rectangle that contains the polyomino itself (see Figure~\ref{fig:directed}). For instance, a polyomino $P$ is {\itshape directed convex} when $C$ is empty \ie, the lowest leftmost vertex belongs to $P$.

\begin{figure}[h!]
\centering
\includegraphics[width=3.5cm]{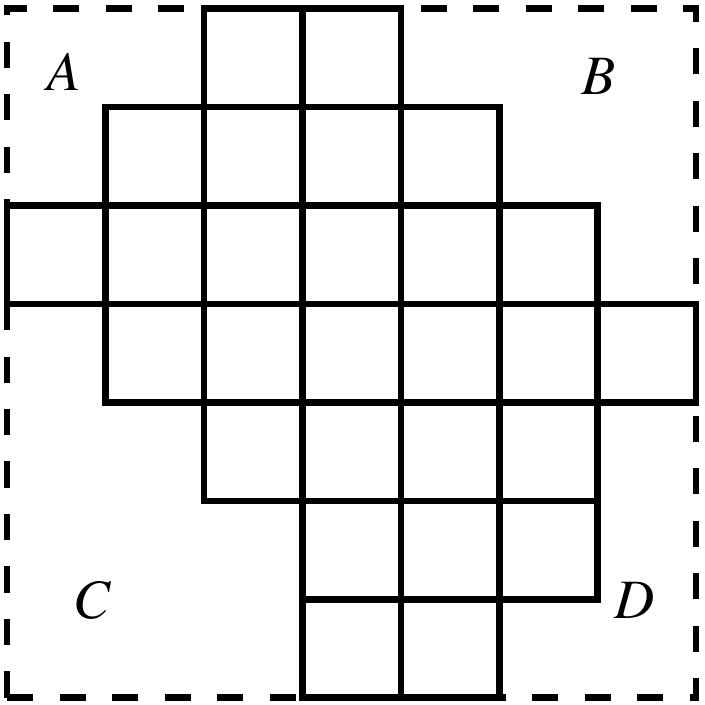}
\caption[Sets of cells $A, B, C$ and $D$ of convex polyominoes]{A convex polyomino and the $4$ sets of cells identified by its intersection with the minimal bounding rectangle.}
\label{fig:directed}
\end{figure}

Specifically, in this thesis we will consider the following families of polyominoes:
\begin{description}
 \item {\bfseries (a)} {\em Ferrer diagram}, \ie $A$, $C$ and $D$ empty;
 \item {\bfseries (b)} {\em Stack} polyomino,\ie $C$ and $D$ empty;
 \item {\bfseries (c)} {\em Parallelogram} polyomino, \ie $C$ and $B$ empty. 
\end{description}
We now review the most important results concerning the enumeration of the aforementioned sets of polyominoes.\\

\noindent
{\bfseries (a)} {\em Ferrer diagrams} (Figure~\ref{fig:FSPD} $(a)$) provide a graphical representation of integers partitions and have the same characteristics of the other families of convex polyominoes. The generating function with respect to the area, that was already known by Euler \cite{eulero}, is:
\begin{equation}
f(q)=\frac{1}{(q)_\infty}\,,
\end{equation}
while the generating function according to the number of columns and rows is:
\begin{equation}
f(x,y)=\frac{xy}{1-x-y}\,.
\label{eq:ferrerGenerating}
\end{equation}
The generating function of the Ferrer diagrams with respect to the semi-perimeter can be easily derived by setting all the variables of Equation~(\ref{eq:ferrerGenerating}) equal to $t$.\\

\noindent
{\bfseries (b)} {\em Stack} polyominoes (Figure~\ref{fig:FSPD} $(b)$) can be seen as a composition of two Ferrer diagrams. Their generating function according to the number of columns, rows and area, is \cite{wright}:
\begin{equation}
f(x,y,q)=\sum_{n\geq1}\frac{xy^nq^n}{(xq)_{n-1}(xq)_n}\,.
\end{equation}
The generating function with respect to semi-perimeter is rational \cite{DV}:
\begin{equation}
f(t)=\frac{t^2(1-t)}{1-3t+t^2}=\sum_{n\geq2}F_{2n-4}t^n\,,
\end{equation}
where $F_n$ denotes the $n$-th number of Fibonacci. For more details on the sequence of Fibonacci $A000045$ the reader is referred to \cite{Sl}. By definition, the first two numbers of the Fibonacci sequence are $F_0=0$ and $F_1=1$, and each subsequent number is the sum of the previous two. Consequently, their recurrence relation can be expressed as follows:
\begin{equation}
F_n=F_{n-1}+F_{n-2}\, \quad\mbox{with}\quad n\geq2\,.
\end{equation}

\noindent
{\bfseries (c)} {\em Parallelogram} polyominoes (Figure~\ref{fig:FSPD} $(c)$) are a particular class of  convex polyominoes uniquely identified by a pair of paths consisting only of north and east steps, such that the paths are disjoint except at their common ending points. The path beginning with a north (respectively east) step is called {\em upper} (respectively {\em lower}) path. 

It is known from  \cite{S} that the number of parallelogram polyominoes with semi-perimeter $n \geq 2$ is equal to the $(n -1)$-th {\em Catalan number}. The sequence of Catalan numbers is  widely used in several combinatorial problems across diverse scientific areas, including Mathematical Physics, Computational Biology and Computer Science. This sequence of integers was introduced in the $18th$ Century by Leonhard Euler in the attempt to determine the different ways to divide a polygon into triangles. The sequence is named after the Belgian mathematician Eug\`ene Charles Catalan, who discovered the connection to parenthesized expression of the {\em Towers of Hanoi} puzzle. Each number of the sequence is obtained as follows\footnote{More in-depth information on the Catalan sequence $A000108$ is provided in \cite{Sl}. The reader may also refer to the book by R. P. Stanley \cite{S}, where over $100$ different interpretations of Catalan numbers tackling with different counting problems of combinatorics are provided.}:
$$C_n = \frac{1}{n+1}\binom{2n}{n}.$$ 

The generating function of parallelogram polyominoes with respect to the number of columns and rows is:
\begin{equation}
f(x,y)=\frac{1-x-y-\sqrt{x^2+y^2-2x-2y-2xy+1}}{2}\,.
\label{eq:parallelogramColsRows}
\end{equation}
The corresponding function depending on the semi-perimeter is straightforwardly derived by setting all the variables equal to $t$. It is also worth noting that the function in Equation~(\ref{eq:parallelogramColsRows}) is algebraic.

Delest and Fedou \cite{DelFed} enumerated this set of polyominoes according to the area by generalizing the results by Klarner and Rivest \cite{KR} as follows:
\begin{equation}
f(q)=\frac{J_1}{J_0}\,,
\end{equation}
where:
\begin{equation}
J_1=\displaystyle\sum_{n\geq 1}\frac{(-1)^{n-1}x^nq^{\binom{n+1}{2}}}{(q)_{n-1}(yq)_n}
\end{equation}
and $J_0$ is the same of Equation~(\ref{eq:J0}).

\begin{figure}[h!]
  \begin{center}
  \includegraphics[width=12cm]{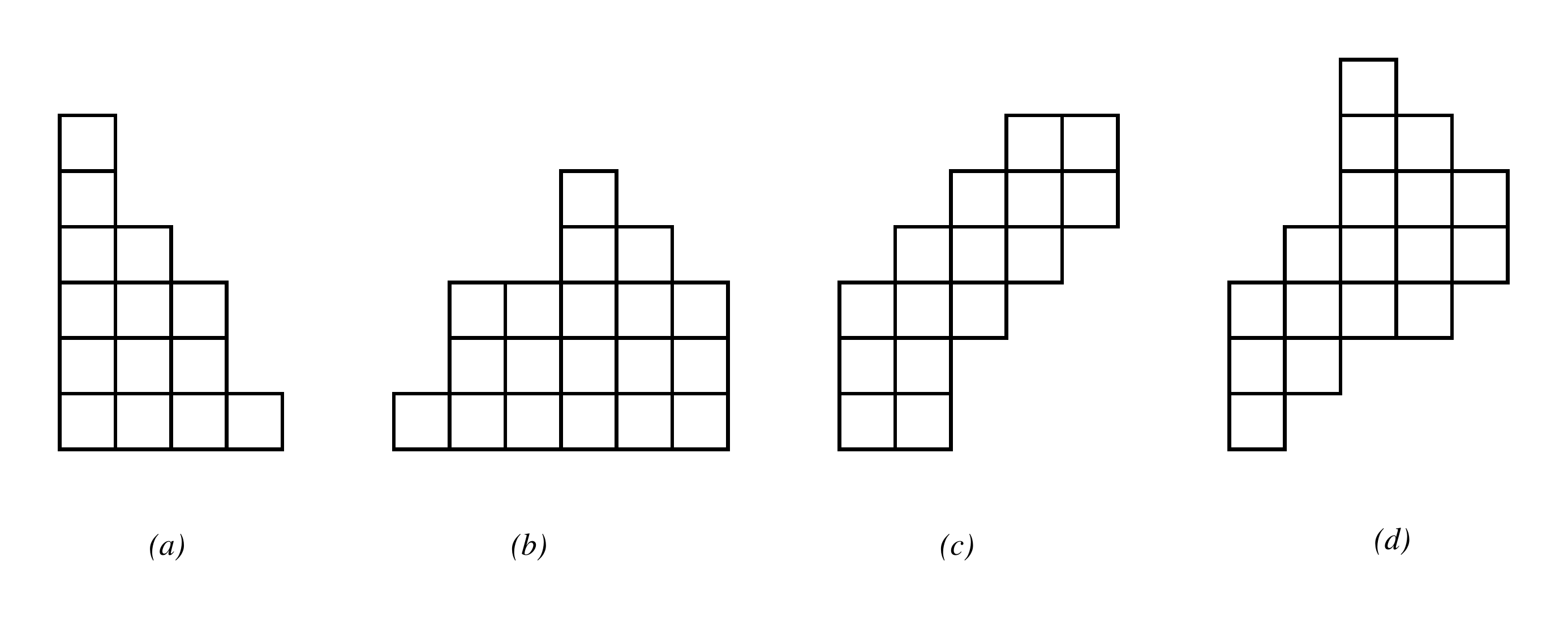}
  \caption{$(a)$: A Ferrer diagram; $(b)$: A stack polyomino; $(c)$: A parallelogram polyomino; $(d)$: A directed convex polyomino which is neither a parallelogram nor a stack one.}
  \label{fig:FSPD}
  \end{center}
  \end{figure}

\subsection{$k$-convex polyominoes}\label{subsec:kconv}
The studies of Castiglione and Restivo \cite{lconv2} pushed the interest of the research community towards the characterization of the convex polyominoes whose internal paths satisfy specific constraints.\\
We recall the following definition of internal path of a polyomino.
\begin{definition}\label{def:Internal-path}
 A path in a polyomino is a self-avoiding sequence of unit steps of four types: north $n=(0, 1)$,
south $s=(0, -1)$, east $e=(1, 0)$, and west $w=(-1, 0)$, entirely contained in the polyomino.
\end{definition}

A path connecting two distinct cells $A$ and $B$ of the polyomino starts from the center of $A$, and ends at the center of $B$ as shown in  Figure~\ref{monopath} $(a)$. We say that a path is {\em monotone} if it consists only of two types of steps, as in Figure~\ref{fig:monopath1} $(c)$. Given a path $w = u_1 \ldots u_k$, each pair of steps $u_i u_{i+1}$ such that $u_i \neq u_{i+1}$ , $0 < i < k$, is called a \textit{change of direction}.

\begin{figure}[h!]
\centering
\includegraphics[width=8cm]{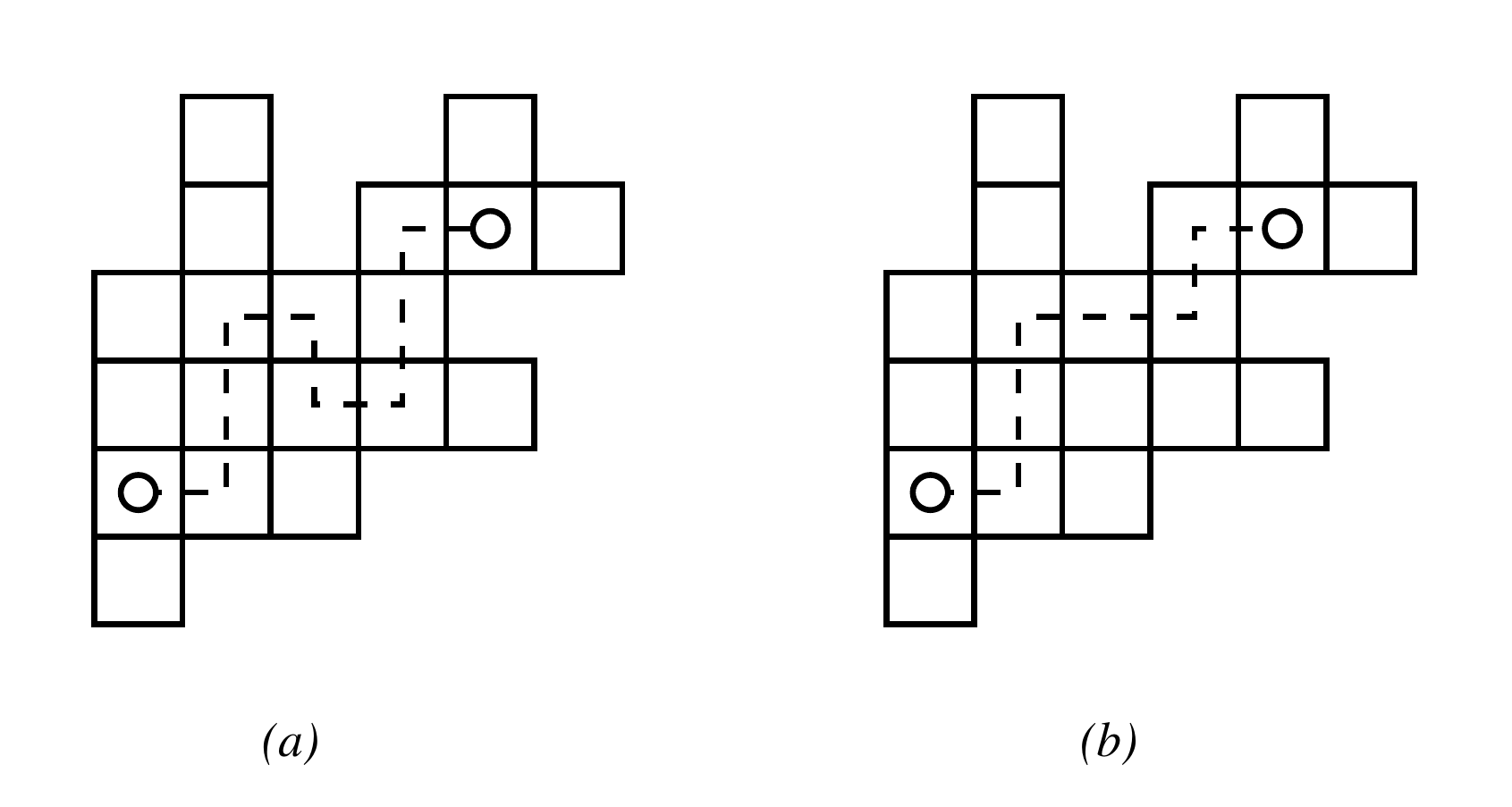}
\caption{$(a)$: A path between two cells of the polyomino ; $(b)$: A monotone path between two cells of the polyomino with four changes of direction.}
\label{fig:monopath1}
\end{figure}

In \cite{lconv2}, it has been observed that in convex polyominoes each pair of cells is connected by a monotone path; therefore, a classification of convex polyominoes based on the number of changes of direction in the paths connecting any two cells of the polyomino was proposed. 
\begin{definition}
 A convex polyomino is said to be $k$-convex if every pair of its cells can be connected by a monotone path with at most $k$ changes of direction. The parameter $k$ is referred to as the {\em convexity degree} of the polyomino.
\end{definition}

For $k=1$, we have the {\em $L$-convex polyominoes}, where any two cells can be connected by a path with at most one change of direction. Such objects have several interesting properties and can also characterized through their maximal rectangle. As a consequence, a convex polyomino $P$ is $L$-convex if and only if any two of its maximal rectangles have a non-void crossing intersection. Some examples of rectangles having non-crossing and crossing intersections are shown in Figure~\ref{fig:crossing}.

\begin{figure}[htd]
\centering
\includegraphics[width=10cm]{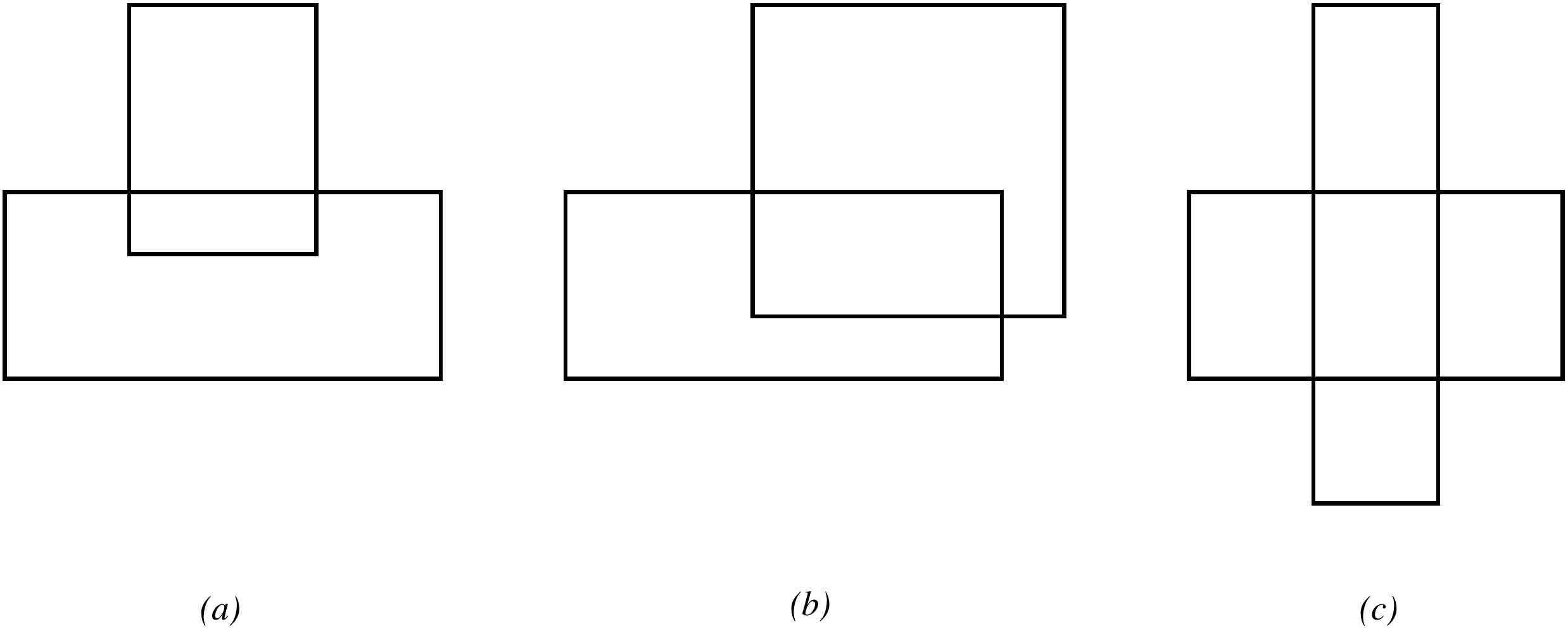}
\caption[Examples of crossing and non-crossing interesections]{(a), (b) Two rectangles having a non-crossing intersection; (c) Two rectangles having crossing intersection.} 
\label{fig:crossing}
\end{figure}
  
In recent literature, several aspects of the $L$-convex polyominoes have been studied: in \cite{CR}, it is shown that they are a well-ordering according to the sub-picture order; in \cite{CFRR}, it has been shown that $L$-convex polyominoes are uniquely determined by their horizontal and vertical projections; finally, it has been proved in \cite{lconv,CFRR2} that the number $f_n$ of $L-$convex polyominoes having semi-perimeter equal to $n + 2$ satisfies the recurrence relation:
\begin{equation}
f_{n+2} = 4 f_{n+1} - 2 f_{n},
\end{equation}
with  $n\geqslant 1$, $f_0 = 1$, $f_1 = 2$ and $f_2 = 7$.

For $k=2$, we have {\em $2$-convex} (or {\em$Z$-convex}) polyominoes, where each pair of cells can be connected by a path with at most two changes of direction. Unfortunately, $Z$-convex polyominoes do not inherit most of the combinatorial properties of $L$-convex polyominoes. In particular, standard enumeration techniques can not be applied to the enumeration of $Z$-convex polyominoes, even though this problem has been tackled with in \cite{DRS} by means of the so-called {\em inflation method}. The authors were able to demonstrate that the generating function is algebraic and the sequence asymptotically grows as $n4^n$, that is the same growth of the whole family of the convex polyominoes.\\ 

Because the solution found for the $Z$-convex polyominoes can not be directly extended to a generic $k$, the problem of enumerating $k$-convex polyominoes for $k>2$ is yet open and difficult to solve. Some recent results of the asymptotic behavior of $k$-convex polyominoes have been achieved by Micheli and Rossin in \cite{MR}. In this thesis we contribute to this topic by  enumerating a remarkable subset of $k$-convex polyominoes, \ie the $k$-convex polyominoes which are also {\em parallelogram polyominoes}, called for brevity {\em $k$-parallelogram polyominoes}.

\section{Permutations}
\label{sec:permutations}

In this section we introduce the family of permutations, which have an important role in several areas of Mathematics such as Computer Science (\cite{Kn,tarjan,west}) and Algebraic Geometry (\cite{sandhya}). Even though the existing literature on permutations is indeed vast, we are particularly interested on the topic of {\em pattern avoidance} (mainly of permutations  but also of other families of objects). Therefore, here we provide basic definitions concerning permutations that will help us extend the concept of permutation to the set of polyominoes in Chapter~\ref{chap:cap4}.

The topic of pattern-avoiding permutations (also known as restricted permutations) has raised a remarkable interest in the last twenty years and led to remarkable results including enumerations and new bijections. One of the most important recent contributions is the one by Marcus and Tardos \cite{marcTard}, consisting in the proof of the so-called Stanley-Wilf conjecture, thus defining an exponential upper bound to the number of permutations avoiding any given pattern. However, the study of statistics on restricted permutations started growing very recently, in particular towards the introduction of new kinds of patterns. 


\subsection{Basic definitions}
In the sequel we will indicate with $[n]$ the set $\{1, 2, . . . , n\}$ and with $\sym_n$ the symmetric group on $[n]$. Moreover, we will use a one-line notation for a permutation $\pi \in \sym_n$, that will be then written as $\pi = \pi_1\pi_2\cdots\pi_n$.

According to literature, there are two common interpretations of the notion of permutation, which can be regarded as a word $\pi$ or as as a bijection $\pi : [n] \mapsto [n]$. The concept of pattern avoidance stems from the first interpretation.

A permutation $\pi$ of length $n$ can be represented in three different ways:
\begin{enumerate}
 \item {\em Two-lines notation}: this is perhaps the most widely used method to represent a permutation and consists in organizing in the top row the numbers from $1$ to $n$ in ascending order and their image in the bottom row, as exemplified in Figure~\ref{fig:reprPerm} $(a)$.
 \item {\em One-line notation}: in this case only the second row of the corresponding two-lines notation is used.
 \item {\em Graphical representation}: it corresponds to the graph $$G(\pi)=\{(i,\pi_i):1\leq i\leq n\}\subseteq [1,n]\times[1,n]\,.$$ An example of $G(\pi)$ is displayed in Figure~\ref{fig:reprPerm} $(b)$. 
\end{enumerate}

\begin{figure}[!h]
\centering
\includegraphics[width=10cm]{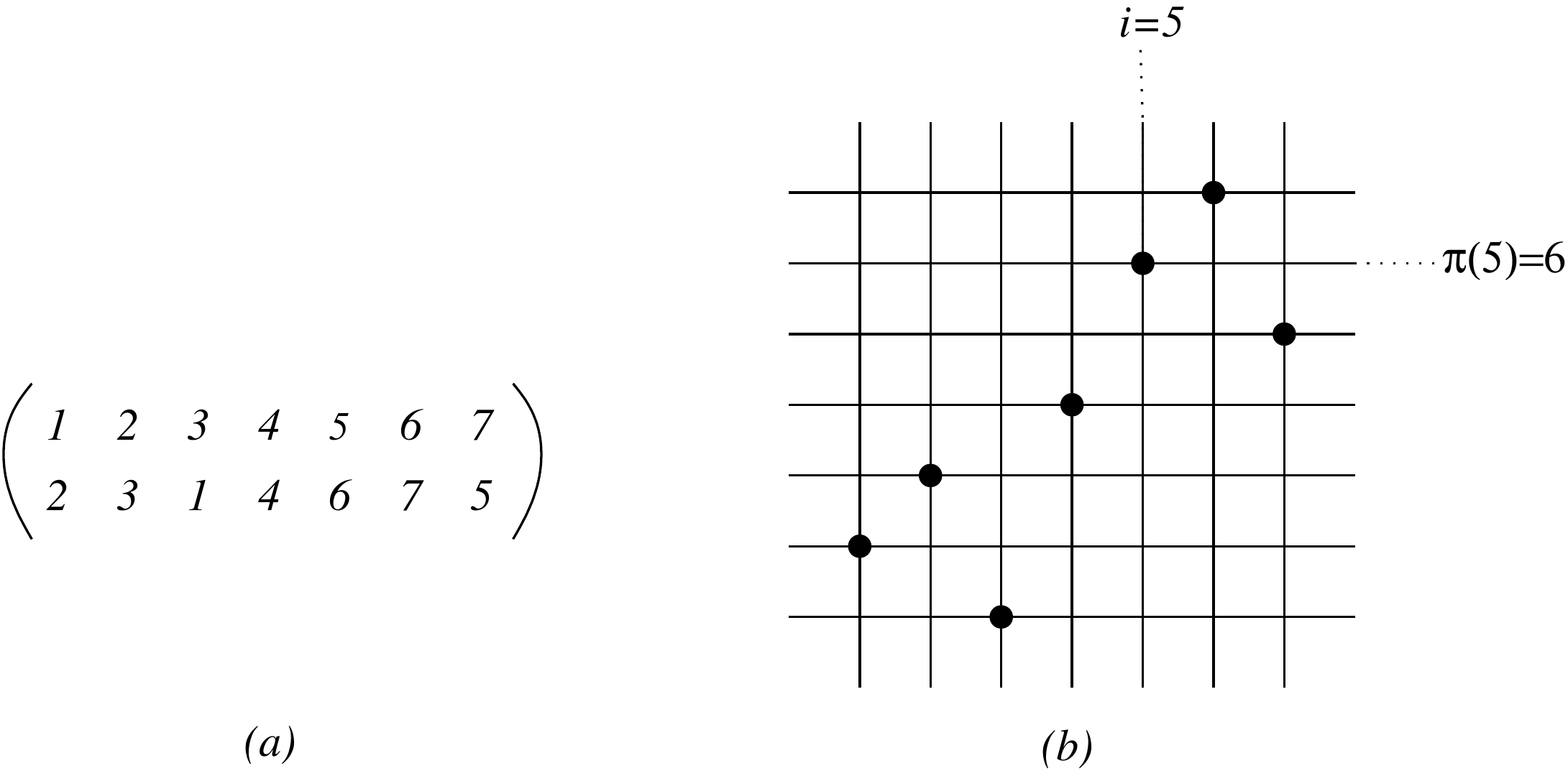}
\caption[Two-lines representation of a permutation]{Two-lines representation of the permutation $\pi=2314675$ in $(a)$ and in $(b)$ the graphical representation of $\pi$.}
\label{fig:reprPerm}
\end{figure}

Let $\pi$ be a permutation; $i$ is a {\em fixed point} of $\pi$ if $\pi_i = i$ and an {\em exceedance} of $\pi$ if $\pi_i > i$. The number of fixed points and exceedances of $\pi$ are indicated with $fp(\pi)$ and $exc(\pi)$ respectively. 

An element of a permutation that is neither a fixed point nor an exceedance, i.e. an $i$ for which $\pi_i < i$, is called {\em deficiency}. Permutations with no fixed points are often referred to as {\em derangements}. 

We say that $i \leq n - 1$ is a {\em descent} of $\pi \in \sym_n$ if $\pi_i > \pi_{i+1}$. Similarly, $i \leq n - 1$ is an {\em ascent} of $\pi \in \sym_n$ if $\pi_i < \pi_{i+1}$. The number of descents and ascents of $\pi$ are indicated with $des(\pi)$ and $asc(\pi)$ respectively. 

Given a permutation $\pi$, we can define the following subsets of points \cite{bona}:
\begin{enumerate}
 \item the set of {\em right-to-left minima} as the set of points: $$\{(i,\pi_i) : \pi_i<\pi_j \quad\forall j , 1\leq j<i\}\,;$$
 \item the set of {\em right-to-left maxima} as the set of points: $$\{(i,\pi_i) : \pi_i>\pi_j \quad\forall j , 1\leq j<i\}\,;$$
 \item the set of {\em left-to-right minima} as the set of points: $$\{(i,\pi_i) : \pi_i<\pi_j \quad\forall j , i<j\leq n\}\,;$$
 \item the set of {\em left-to-right maxima} as the set of points: $$\{(i,\pi_i) : \pi_i>\pi_j \quad\forall j , 1<j\leq n\}\,.$$
\end{enumerate}
An example of each of the sets defined above is provided in Figure~\ref{fig:rtl_ltr}.\\

\begin{figure}[htbp]
\centering
\includegraphics[width=14cm]{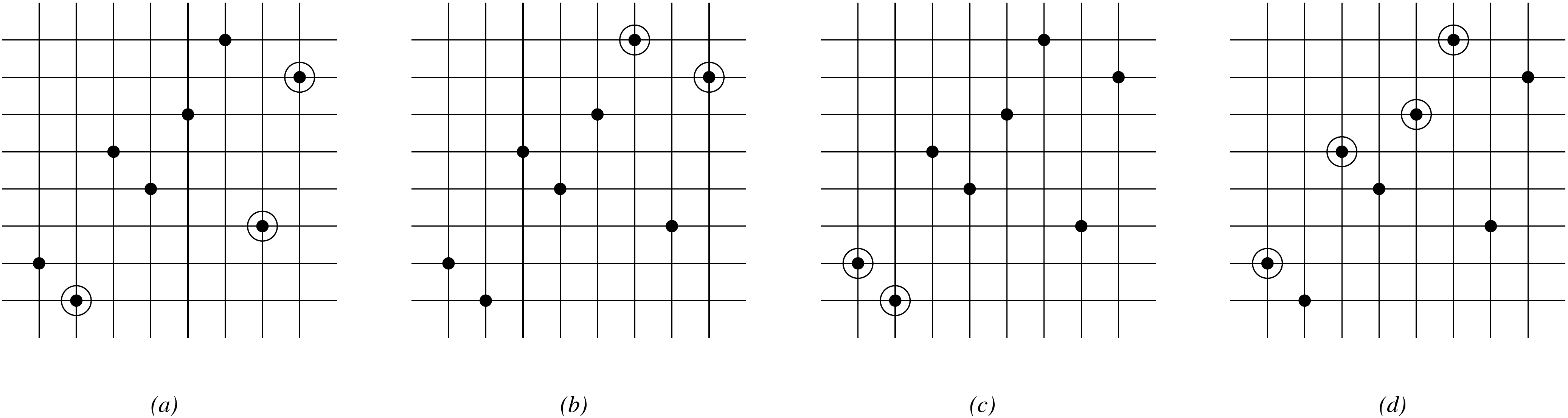}
\caption[Examples of the four subsets of points of $\pi$]{$(a)$ Set of right-to-left minima for $\pi=21546837$; $(b)$ set of right-to-left maxima; $(c)$ set of left-to-right minima; and $(d)$ set of left-to-right maxima.}
\label{fig:rtl_ltr}
\end{figure}

Let $lis(\pi)$ denote the length of the longest increasing subsequence of $\pi$, i.e., the largest $m$ for which there exist indexes $i_1 < i_2 < \cdots < i_m$ such that $\pi_{i_1} < \pi_{i_2} < \cdots< \pi_{i_m}$.

Define the {\em rank} of $\pi$, denoted $rank(\pi)$, to be the largest $k$ such that $\pi_i > k$ for all $i \leq k$. For example, if $\pi = 63528174$, then $fp(\pi) = 1$, $exc(\sigma) = 4$, $des(\pi) = 4$ and $rank(\pi) = 2$.

We say that a permutation $\pi\in \sym_n$ is an {\em involution} if $\pi = \pi^{-1}$. The set of involutions of length $n$ is indicated with ${\cal I}_n$.

\subsection{Pattern avoiding permutations}\label{sec:patternAv}


The concept of permutation patterns is well-known to many branches of Mathematics literature, as proved by the several works that have been proposed in the last decades. A comprehensive overview,  ``Patterns in Permutations'', has been proposed by Kitaev in \cite{kit}.

\begin{definition}\label{def:permutation_pattern}
Let $n, m$ be two positive integers with $m \leq n$, and let $\pi \in \sym_n$ and
$\sigma \in \sym_m$ be two permutations. We say that $\pi$ contains $\sigma$ if there exist indexes $i_1 < i_2 < \cdots < i_m$ such that $\pi_{i_1}\pi_{i_2}\cdots\pi_{i_m}$ is in the same relative order as $\sigma_1\sigma_2\cdots\sigma_m$ (that is, for all indexes $a$ and $b$, $\pi_{i_a} < \pi_{i_b}$ if and only if $\sigma_a < \sigma_b$). In that case, $\pi_{i_1}\pi_{i_2}\cdots\pi_{i_m}$ is called an occurrence of $\sigma$ in $\pi$ and we write $\sigma \prec\pi$. In this context, $\sigma$ is also called a pattern.
\end{definition}

If $\pi$ does not contain $\sigma$, we say that $\pi$ {\em avoids} $\sigma$, or that $\pi$ is {\em $\sigma$-avoiding}. For example, if $\sigma = 231$, then $\pi = 24531$ contains $231$, because the subsequence $\pi_2\pi_3\pi_5 = 451$ has the same relative order as $231$. However, $\pi = 51423$ is $231$-avoiding. We indicate with $\Av_n(\sigma)$ the set of $\sigma$-avoiding permutations in $\sym_n$. 

\begin{definition}
A class $\cal C$ of permutations is {\em stable} or {\em downward closed} for $\preceq$ if, for any $\pi \in \cal C$ and for any pattern $\sigma \prec \pi$, $\sigma \in \cal C$.
\end{definition}

It is a natural generalization to consider permutations that avoid several patterns at the same time. If ${\cal B} \subseteq \sym_k$, $k\geq1$, is any finite set of patterns, we denote by $\Av_n({\cal B})$, also called ${\cal B}$-avoiding permutation, the set of permutations in $\sym_n$ that avoid simultaneously all the patterns in ${\cal B}$. For example, if ${\cal B} = \{123, 231\}$, then $\Av_4({\cal B}) = \{1432, 2143, 3214, 4132, 4213, 4312, 4321\}$. We remark that for every set ${\cal B}$, $\Av_n({\cal B})$ is a class of permutations.

The sets of permutations pairwise-incomparable with respect to the order relation ($\preceq$) are called {\em antichains}.
\begin{definition}
 If ${\cal B}$ is an antichain, then ${\cal B}$ is unique and is called basis of the class of permutations $\Av_n({\cal B})$. In this case, it also true that $${\cal B}=\{\pi\notin \Av_n({\cal B}):\forall \sigma\prec\pi,\sigma\in \Av_n({\cal B})\}\,.$$
\end{definition}

\begin{proposition}\label{unicityBase}
 Let be ${\cal C}=\Av_n({\cal B}_1)=\Av_n({\cal B}_2)$. If ${\cal B}_1$ and ${\cal B}_2$ are two antichains then ${\cal B}_1={\cal B}_2$.
\end{proposition}

It is quite simple to demonstrate the following proposition.
\begin{proposition}
A class $\cal C$ of permutations that is stable for $\preceq$ is a class of pattern-avoiding permutations and so it can be characterized by its basis.
\end{proposition}

%
%

Even though the majority of permutation classes analyzed in literature are characterized by finite bases, there exist classes of permutations with infinite basis (e.g. the {\em pin} permutations in \cite{bouvel}). Understanding whether a certain class of permutations is characterized by a finite or infinite basis is not an entirely solved problem; some, but suggestions on the decision criteria can be found in \cite{AA,AS}. 


\subsubsection*{Results on pattern avoidance}

\begin{definition}
Two patterns are Wilf equivalent and belong to the same Wilf class if, for each $n$, the same number of permutations of length $n$ avoids the same pattern.
\end{definition}
Wilf equivalence is a very important topic in the study of patterns. The smallest example of non-trivial Wilf equivalence is for the classical patterns of length $3$; in fact, the patterns $123$ and $321$ are Wilf equivalent, and the same is true for the remaining four patterns of length $3$, namely $132$, $213$, $231$, and $312$. All six of these patterns are Wilf equivalent, which is easy but non-trivial to demonstrate; each pattern is avoided by $C_n$ permutations of length $n$, where $C_n$ is the Catalan number $\frac{1}{n+1}{2n\choose n}$.\\

By extension, we can define the {\em strongly Wilf equivalence} as follows.
\begin{definition}
Two patterns $\pi$ and $\sigma$ are strongly Wilf equivalent if they have the same distribution on the set of permutations of length $n$ for each $n$, that is, if for each nonnegative integer $k$ the number of permutations of length $n$ with exactly $k$ occurrences of $\pi$ is the same as that for $\sigma$.
\end{definition}
For example, $\pi = 132$ is strongly Wilf equivalent to $\sigma = 231$, since the bijection defined by reversing a permutation turns an occurrence of $\pi$ into an occurrence of $\sigma$ and conversely. On the other hand, $132$ and $123$ are not strongly Wilf equivalent, although they are Wilf equivalent. Furthermore, the permutation $1234$ has four occurrences of $123$, but there is no permutation of length $4$ with four occurrences of $132$.

One of the most investigated problems is the enumeration of the elements of a given class $\cal C$ of permutations for any integer $n$. Interesting recent results in this direction can be found in \cite{guibert,kit} (respectively, $1995$ and $2003$). However, such enumeration problem was already known since $1973$ thanks to the work of Knuth \cite{Kn}, where permutations avoiding the pattern $231$ were considered.

As for the case of patterns of length three even for the patterns of length four we can reduce the problem to take into consideration the seven symmetrical classes and it is sufficient to study three of them to obtain the sequences of enumeration. We can find a few results relatives to this patterns in \cite{bona}. Only the problem of enumeration of the permutations that avoid $4231$ (or $1324$) remains unsolved. 

In $1990$, Stanley and Wilf conjectured that, for all classes $\cal C$, there exists a constant value $c$ such that for all integer $n$ the number of elements in ${\cal C}_n={\cal C}\cap \sym_n$ is less than or equal to $c^n$. In $2004$, Marcus and Tardos \cite{marcTard} proved the Stanley-Wilf conjecture. Before such result, Arratia \cite{arratia} showed that, being ${\cal C}_n = \Av_n(\sigma)$, the conjecture was equivalent to the existence of the limit:
$$SW(\sigma) = \lim_{n\to \infty}\Av_n(\sigma)\,,$$
which is called the {\em Stanley-Wilf limit} for $\sigma$.

The Stanley-Wilf limit is $4$ for all patterns of length three, which follows from the fact
that the number of avoiders of any one of these is the $n$-th Catalan number $C_n$, as
mentioned above. This limit is known to be $8$ for the pattern $1342$ (see \cite{bona2}). For the pattern $1234$, the limit is $9$; such limit was obtained as a special case of a result of Regev \cite{regev2,regev1}, who provided a formula for the asymptotic growth of the number of standard Young tableaux with at most $k$ rows. The same limit can also be derived from Gessel's general result \cite{gessel2} for the number of avoiders of an increasing pattern of any length.
The only Wilf class of patterns of length four for which the Stanley-Wilf limit is unknown
is represented by $1324$, for which a lower bound of $9.47$ was established by Albert et al. \cite{AERWZ}. Later, Bona \cite{bona0} was able to refine this bound by resorting to the method in \cite{CJS}; finally, Madras and Liu \cite{ML} estimated that the limit for the pattern$1324$ lies, with high likelihood, in the interval $[10.71, 11.83]$\footnote{This result was obtained by using Markov chain Monte Carlo methods to generate random $1324$-avoiders.}.

Finally, considering a permutation as a bijection we can take in exam some concepts such as fixed points and exceedances. This new way to see a permutation makes it interesting to study some of statistics together with the notion of pattern avoidance. There is a lot of mathematical literature devoted to permutation statistics (see for example \cite{ehrStein, foata, garsiaGessel, gesselReut}).
 
\subsection{Generalized patterns and other new patterns}\label{sec:genPatt}

Babson and Steingr\'imsson \cite{babStein} introduced the notion of {\em generalized patterns}, which requires that two adjacent letters in a pattern must be adjacent in the permutation, as shown in Figure~\ref{vincMesh} $(b)$. The authors introduced such patterns to classify the family of {\em Mahonian permutation statistics}, which are uniformly distributed with the number of inversions.

A generalized pattern can be written as a sequence wherein two adjacent elements may or may not be separated by a dash. With this notation, we indicate a classical pattern with dashes between any two adjacent letters of the pattern (for example, $1423$ as $1-4-2-3$). If we omit the dash between two letters, we mean that for it to be an occurrence in a permutation $\pi$, the corresponding elements of $\pi$ have to be adjacent. For example, in an occurrence of the pattern $12-3-4$ in a permutation $\pi$, the entries in $\pi$ that correspond to $1$ and $2$ are adjacent. The permutation $\pi = 3542617$ has only one occurrence of the pattern $12-3-4$, namely the subsequence $3567$, whereas $\pi$ has two occurrences of the pattern $1-2-3-4$, namely the subsequences $3567$ and $3467$.

If $\sigma$ is a generalized pattern, $\Av_n(\sigma)$ denotes the set of permutations in $\sym_n$ that have no occurrences of $\sigma$ in the sense described above. Throughout this chapter, a pattern represented with no dashes will always denote a classical pattern, i.e. one with no requirement about elements being consecutive, unless otherwise specified.

Several results on the enumeration of permutations classes avoiding generalized patterns have been achieved. Claesson obtained the enumeration of permutations avoiding a generalized pattern of length three \cite{claesson} and the enumeration of permutations avoiding two generalized patterns of length three \cite{CM}. Another interesting result in terms of permutations avoiding a  set of generalized patterns of length three was obtained by Bernini \textit{et al.} in \cite{BBF,BFP}, where one can find the enumeration of permutation avoiding set of generalized patterns as a function of its length and another parameter.

Another kind of patterns, called {\em bivincular patterns}, was introduced in \cite{BMCDK} with the aim to increase the symmetries of the classical patterns. In \cite{BMCDK}, the bijection between permutations avoiding a particular bivincular pattern was derived, as well as several other classes of combinatorial objects.

\begin{definition}
Let $p=(\sigma,X,Y)$ be a triple where $\sigma$ is a permutation of $\sym_n$ and $X$ and $Y$ are subsets of $\{0\}\cup[n]$. An occurrence of $p$ in $\pi$ is a subsequence $q=(\pi_{i_1},\cdots,\pi_{i_k})$ such that $q$ is an occurrence of $\sigma$ in $\pi$ and, with $(j_1<j_2<\cdots<j_k)$ being the set $\{\pi_{i_1},\cdots,\pi_{i_k}\}$ ordered (so $j_1=min_m\pi_{i_m}$ etc.), and $i_0=j_0=0$ and $i_{k+1}=j_{k+1}=n+1$,
 $$i_{x+1}=i_x+1\quad \forall x\in X\qquad \mbox{and}\qquad j_{y+1}=j_y+1\quad \forall y \in Y\,.$$
\end{definition}
Bivincular patterns are graphically represented by graying out the corresponding columns and rows in the Cartesian plane as exemplified in Figure~\ref{vincMesh} $(c)$. Clearly, bivincular patterns $(\sigma,\emptyset,\emptyset)$ coincide with the classical patterns, while bivincular patterns $(\sigma,X,\emptyset)$ coincide with the generalized patterns (hence, we will refer to them as {\em vincular} in the sequel).

\begin{figure}[!h]
\centering
\includegraphics[width=12cm]{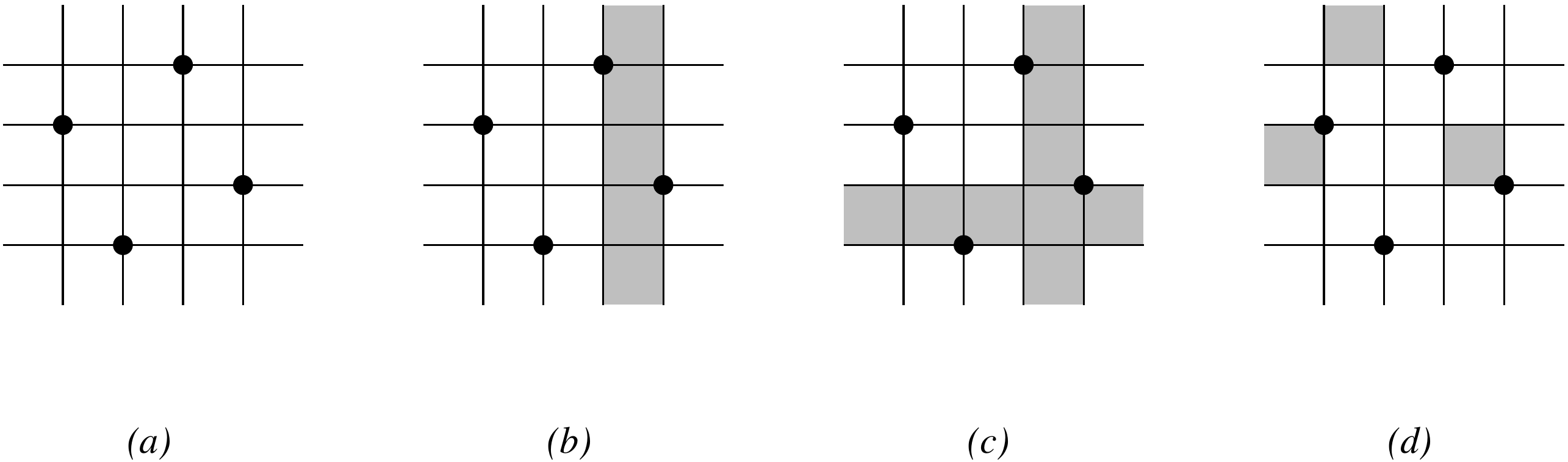}
\caption[Classical, generalized and bivincular patterns]{$(a)$ Classical pattern $3-1-4-2$; $(b)$ Generalized (or vincular) pattern $3-1-42$: $(c)$ Bivincular pattern $(3142,\{1\},\{3\})$; and $(d)$ Mesh pattern $(3142, R)$.}
\label{vincMesh}
\end{figure}

We now give the definition of {\em Mesh patterns}, which were introduced in \cite{mesh} to generalize multiple varieties of permutation patterns. To do so, we extend the above prohibitions determined by grayed out columns and rows to graying out an arbitrary subset of squares in the diagram. 

\begin{definition}
A mesh pattern is an ordered pair $(\sigma, R)$, where $\sigma$ is a permutation of $\sym_k$ and $R$ is a subset of the $(k+1)^2$ unit squares in $[0,k+1]\times [0,k+1]$, indexed by they lower-left corners.
\end{definition}

Thus, in an occurrence, in a permutation $\pi$, of the pattern $(3142,R)$, where $R= \{(0,2),(1,4),(4,2)\})$
in Figure \ref{vincMesh} $(d)$, there must, for example, be no letter in $\pi$ that precedes all letters in the occurrence and lies between the values of those corresponding to the $1$ and the $3$.
This is required by the shaded square in the leftmost column. For example, in the
permutation $425163$, $5163$ is not an occurrence of $(3142, R)$, since $4$ precedes $5$ and
lies between $5$ and $1$ in value, whereas the subsequence $4263$ is an occurrence of this mesh pattern.

The reader can find an extension of mesh patterns in \cite{tenner}, in which the author characterizes all mesh patterns in which the mesh is superfluous.

Both with regard to the bivincular patterns that mesh patterns is interesting to extend the results obtained in the case of classical patterns, in particular the analysis of classes Wilf equivalent. For example in \cite{P} we can find the classification of all bivincular patterns of length two and three according to the number of permutations avoiding them, and a partial classification of mesh patterns of small length in \cite{HJSVU}. 

\section{Partially ordered sets}
\label{sec:posets}

In this section we provide the basic notions and the most important definition on partially ordered sets (posets). For a more in-depth analysis, the interested reader can refer to \cite{S1}.

\begin{definition}
A {\em partially ordered set} or poset is a pair $P = (X; \leq)$ where $X$ is a set and
$\leq$ is a reflexive, antisymmetric, and transitive binary relation on $X$.
\end{definition}
$X$ is referred to as the {\em ground set}, while $P$ is a partial order on $X$. Elements of the ground set $X$ are also called {\em points}. A poset is finite if the ground set is finite.

In our work, we will consider only finite posets. Of course, the notation $x < y$ in $P$ means $x \leq y$ in $P$ and $x\neq y$. When the poset does not change throughout our analysis, we find convenient to abbreviate $x \leq y$ in $P$ with $x \leq_P y$. If $x, y \in X$ and either $x \leq y$ or $y \leq x$, we say that $x$ and $y$ are {\em comparable} in $P$; otherwise, we say that $x$ and $y$ are {\em incomparable} in $P$.

\begin{definition}
A partial order $P = (X; \leq)$ is called {\em total order} (or {\em linear order}) if for all $x, y \in X$, either $x \leq y$ in $P$ or $y \leq x$ in $P$.
\end{definition}

\begin{definition}
Let $x,y$ be two generic elements in $X$. A partial order $P = (X; \leq)$ is called {\em lattice} when there exist two elements, usually denoted by $x \vee y$ and by $x \wedge y$, such that:
\begin{itemize}
 \item $x \vee y$ is the {\em supremum} of the set $\{x, y\}$ in $P$
 \item$x \wedge y$ is the {\em infimum} of the set $\{x, y\}$ in $P$,
\end{itemize}
i.e. for all $z$ in $X$
$$z\geq x\vee y \Longleftrightarrow z\geq x \quad \mbox{and} \quad z\geq y$$
$$z\leq x\wedge y \Longleftrightarrow z\leq x \quad \mbox{and} \quad z\leq y\,\,.$$
\end{definition}

\begin{definition}
Given $x,y$ in a poset $P$, the interval $[x,y]$ is the poset $\{z \in P : x \leq z \leq y\}$ with the same order as $P$.
\end{definition}

\begin{definition}
Let $P = (X, \leq)$ be a poset and let $x$ and $y$ be distinct points from $X$. We say that ``$x$ is covered by $y$'' in $P$ when $x < y$ in $P$, and there is no point $z \in X$ for which $x < z$ in $P$ and $z < y$ in $P$.
\end{definition}

In some cases, it may be convenient to represent a poset with a diagram of the cover graph in the Euclidean plane. To do so, we choose a standard horizontal/vertical coordinate system in the plane and require that the vertical coordinate of the point corresponding to $y$ be larger than the vertical coordinate of the point corresponding to $x$ whenever $y$ covers $x$ in $P$. Each edge in the cover graph is represented by a straight line segment which contains no point corresponding to any element in the poset other than those associated with its two end points. Such diagrams, called {\em Hasse diagrams}, are defined as follows.

\begin{definition}
The Hasse diagram of a partially ordered set $P$ is the (directed) graph whose vertices are the elements of $P$ and whose edges are the pairs $(x, y)$ for which $y$ covers $x$. It is usually drawn so that elements are placed higher than the elements they cover.
\end{definition}

The Boolean algebra $B_n$ is the set of subsets of $[n]$, ordered by inclusion ($S \leq T$ means $S \subseteq T$). Generalizing $B_n$, any collection $P$ of subsets of a fixed set $X$ is a partially ordered set ordered by inclusion. Figure~\ref{hasse} displays the diagram obtained with $n=3$.
\begin{figure}[!h]
\centering
\includegraphics[width=4cm]{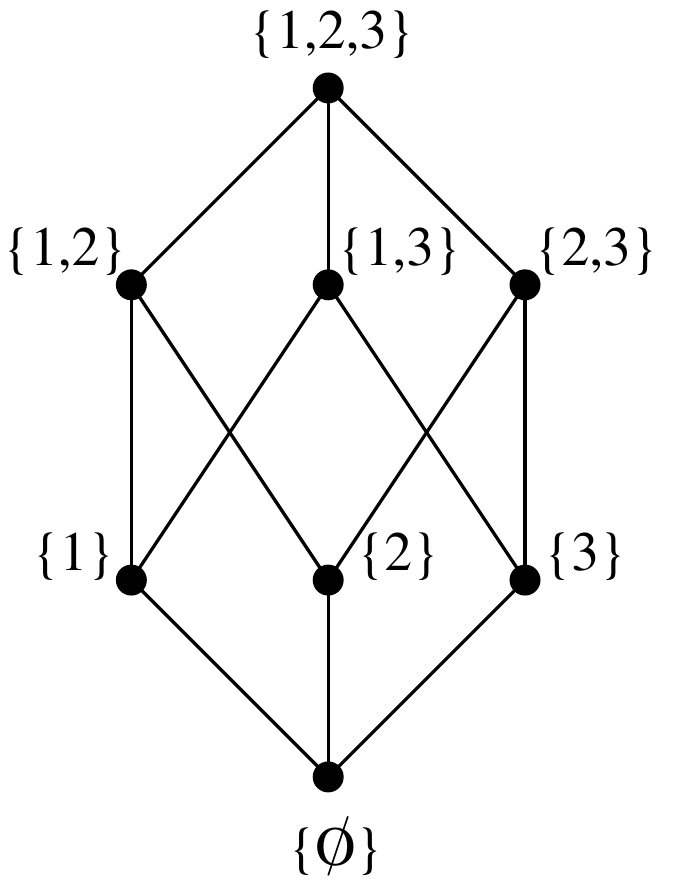}
\caption{The Hasse diagram of $B_3$.}
\label{hasse}
\end{figure}

In particular, Hasse diagrams are useful to visualize various properties of posets.

\begin{definition} 
A linear extension of a poset $P=(X,\leq)$, where $X$ has cardinality $|X|$, is a bijection $\lambda :X\rightarrow \{1,2,\cdots,|X|\}$ such that $x<y$ in $P$ implies $\lambda(x)<\lambda(y)$.
\end{definition}

\begin{definition}
If $P=(X,\leq)$ is a poset and $Y\subseteq X$, then $F_P(Y)=\{x\in X:\forall y\in Y, x>y\}$ (respectively $I_P(Y)=\{x\in X:\forall y\in Y, x<y\}$) is called the {\em filter} (respectively the {\em ideal}) of $P$ generated by $Y$.

If $P=(X,\leq)$ is a poset let ${\cal D}_P=\{I_P(\{x\}): x\in X\}$ and ${\cal U}_P=\{F_P(\{x\}): x\in X\}$ be respectively the set of principal ideals of $P$ and the set of principal filters of $P$.
\end{definition}

\begin{definition}
Given a poset $P=(X,\leq)$ an equivalence relation on $X$ is trivially defined by saying that two elements $x$ and $y$ are order equivalent in $P$ if and only if $I_P(\{x\})=I_P(\{y\})$ and $F_P(\{x\})=F_P(\{y\})$.
\end{definition}

\subsection{Operations on partially ordered sets}
Given two partially ordered sets P and Q, we can define the following new partially ordered sets:
\begin{enumerate}
 \item {\bf Disjoint union.} $P + Q$ is the disjoint union set $P\cup Q$, where $x\leq_{P + Q} y$ if and only if one of the following conditions holds:
\begin{itemize}
 \item $x, y \in P$ and $x \leq_P y$
 \item $x, y \in Q$ and $x \leq_Q y$
\end{itemize}
The Hasse diagram of $P + Q$ consists of the Hasse diagrams of $P$ and $Q$ drawn together.

 \item {\bf Ordinal sum.} $P \oplus Q$ is the set $P\cup Q$, where $x \leq_{P\oplus Q} y$ if and only if one of the following conditions holds:
\begin{itemize}
 \item $x \leq_{P + Q} y$
 \item $x \in P$ and $y \in Q$
\end{itemize}
Note that the ordinal sum operation is not commutative: in $P \oplus Q$, everything in $P$ is less than everything in $Q$.

The posets that can be described by using the operations $\oplus$ and $+$ starting from the single element poset (usually denoted by $1$) are called {\em series parallel orders} \cite{S1}. This set of posets has a nice characterization in terms of avoiding subposet. 

\item {\bf Cartesian product.} $P \times Q$ is the Cartesian product set $\{(x, y): x \in P, y \in Q\}$, where $(x, y) \leq_{P\times Q} (x', y')$ if and only if both $x \leq_P x$ and $y \leq_Q y$.
The Hasse diagram of $P \times Q$ is the Cartesian product of the Hasse diagrams of $P$ and $Q$.
\end{enumerate}
\begin{definition}
A chain of a partially ordered set $P$ is a totally ordered subset $C \subseteq P$, with $C = \{x_0,\cdots, x \}$ with $x_0 \leq \cdots \leq x_l$. The quantity $l= |C| - 1$ is the length of the chain and is equal to the number of edges in its Hasse diagram.
\end{definition}

If $1$ denotes the single element poset, then a chain composed by $n$ elements is the poset obtained by performing the ordinal sum exactly $n$ times: $1\oplus 1\oplus \cdots \oplus 1$.

\begin{definition}
A chain is {\em maximal} if there exist no other chain strictly containing it.
\end{definition}

\begin{definition}
The {\em rank} of $P$ is the length of the longest chain in $P$.
\end{definition}

The set of all permutations forms a poset $P$ with respect to classical pattern containment. That is, a permutation $\sigma$ is smaller than $\pi$ (i.e. $\sigma\leq \pi$) if $\sigma$ occurs as a pattern in $\pi$. This poset is the underlying object of all studies of pattern avoidance and containment.


\chapter[$k$-parallelogram polyominoes]{$K$-parallelogram polyominoes: characterization and enumeration}\label{chap:chapter2}
In this chapter we consider the problem of enumerating a subclass of $k$-convex polyominoes.
We recall (see Section \ref{sec:cap1_pol} for more details) that a convex polyomino is {\em $k$-convex} if every pair of its cells can be connected by means of a {\em monotone path}, internal to the polyomino (see Figure \ref{monopath} $(b)$ and $(c)$), and having at most $k$ changes of direction. In the literature we find some results regarding the enumeration of $k$-convex polyominoes of given semi-perimeter, but only for small values of $k$, precisely $k=1,2$, see again Chapter \ref{chap:chapter1} for more details.

Since the problem of counting $k$-convex polyominoes is difficult, we tackle the problem of enumerating a remarkable subclass of $k$-convex polyominoes, precisely the $k$-convex polyominoes which are also {\em parallelogram polyominoes}, called for brevity {\em $k$-parallelogram polyominoes} and denoted by $\poly_k$.\\
Figure \ref{monopath} $(a)$ shows an example of convex polyomino that is not parallelogram, while Figure \ref{monopath} $(c)$ depicts a $4$-parallelogram (non $3$-parallelogram) polyomino.
%
\begin{figure}[htbp]
  \begin{center}
  \includegraphics[width=10cm]{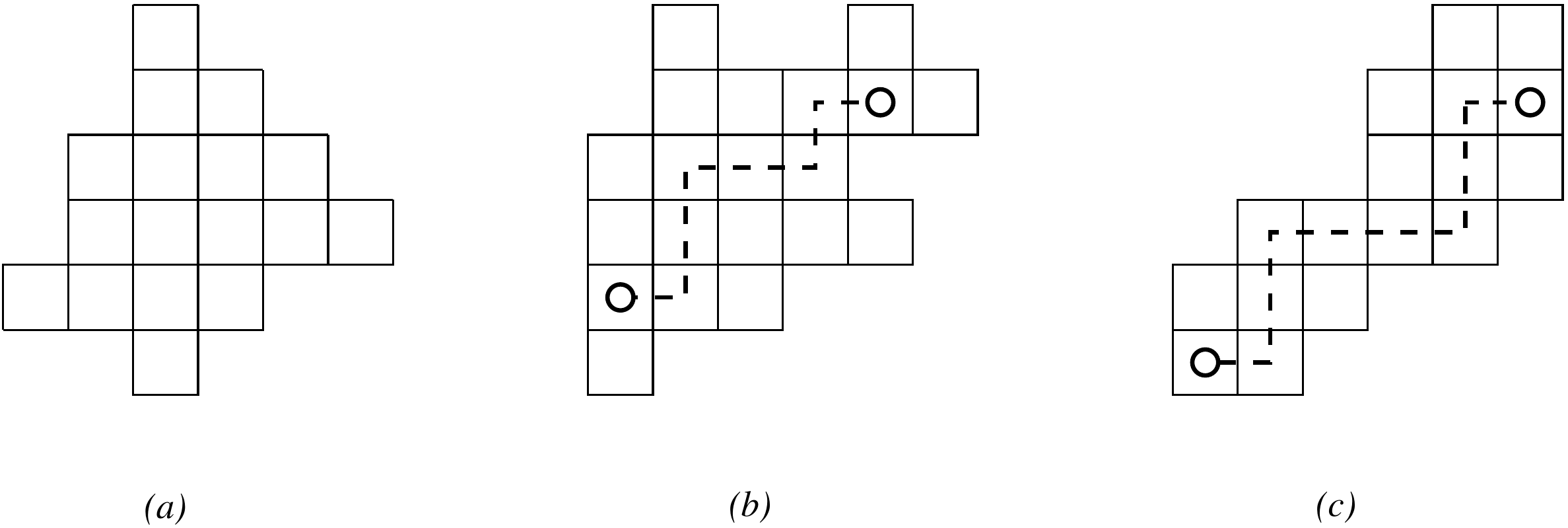}
  \caption{$(a)$ A convex polyomino; $(b)$ a monotone path between two cells of the polyomino with four changes of direction; $(c)$ a $4$-parallelogram.}
  \label{monopath}
  \end{center}
  \end{figure}

The class $\poly$ of $k$-parallelogram polyominoes can be treated in a simpler way than $k$-convex polyominoes, since we can use the simple fact that a parallelogram polyomino is $k$-convex if and only if there exists at least one monotone path having at most $k$-changes of direction running from the lower leftmost cell to the upper rightmost cell of the polyomino.

More precisely, using such a property in the next sections we will partition the class $\poly_k$ into three subclasses, namely the {\em flat}, {\em right}, and {\em up}  $k$-parallelogram polyominoes. We will
provide an unambiguous decomposition for each of the three classes, so we will use these decompositions in order to obtain the generating functions of the three classes and then of $k$-parallelogram polyominoes.
An interesting fact is that, while the generating function of parallelogram polyominoes is algebraic, for every $k$ the generating function of $k$-parallelogram polyominoes is rational.
Moreover, we will be able to express such generating function as continued fractions, and then in terms of the known {\em Fibonacci polynomials}. The final version of the generating function of $\poly_k$ in terms of Fibonacci polynomials suggests us to search some bijection with other combinatorial objects, in particular in \cite{knuth} it is proved that the generating function of plane trees having height less than or equal to a fixed value can be expressed using Fibonacci polynomials and so we found a nice bijection between these two objects. 

To our opinion, this work is a first step towards the enumeration of $k$-convex polyominoes, since it is possible to apply our decomposition strategy to some larger classes of $k$-convex polyominoes (such as, for instance, directed $k$-convex polyominoes).

\section{Classification and decomposition of the class $\poly_k$}

Let us start by providing some basic definitions which will be useful in the rest of the section.

As a Consequence of Definition \ref{def:Internal-path} in Section \ref{sec:cap1_pol} we can represent an internal path as a sequence of cells. 
\begin{definition}
 Let be $A$ and $A'$ two distinct cells of a polyomino; an internal path from $A$ to $A'$, denoted $\pi_{AA'}$, is a sequence of distinct cells $(B_1,\cdots, B_n)$ such that $B_1=A$, $B_n=A'$ and every two consecutive cells in this sequence are edge-connected.
\end{definition}

 Henceforth, since polyominoes are defined up to translation, we assume that the center of each cell of a polyomino corresponds to a point of the plane $\mathbb{Z} \times\mathbb{Z}$, and that the center of the lower leftmost cell of the minimal bounding rectangle (denoted by m.b.r.) of a polyomino corresponds to the origin of the axes. In our case, since we deal of parallelogram polyominoes, we have that the lower leftmost cell of the m.b.r belongs to the polyomino. So, according to the respective position of the cells $B_i$ and $B_{i+1}$, we say that the pair $(B_i,B_{i+1})$ forms:
\begin{enumerate}
 \item a {\em north} step $n$ in the path if $(x_{i+1},y_{i+1})=(x_i,y_{i}+1)$;
 \item an {\em east} step $e$ in the path if $(x_{i+1},y_{i+1})=(x_{i}+1,y_i)$;
 \item a {\em west} step $w$ in the path if $(x_{i+1},y_{i+1})=(x_{i}-1,y_i)$;
 \item a {\em south} step $s$ in the path if $(x_{i+1},y_{i+1})=(x_i,y_{i}-1)$.
\end{enumerate}
Moreover, since we will be working with parallelogram polyominoes which are convex polyominoes, for obvious reasons of symmetry, we will deal only with monotone paths using steps $n$ or $e$.
\begin{definition}
 Let $P$ be a parallelogram polyomino and $\pi$ a path internal to $P$. We call side every maximal sequence of steps of the same type into $\pi$.
\end{definition}
 
 \begin{definition}
  Let $P$ be a parallelogram polyomino. We denote by $S$ and $E$ the lower leftmost cell and the upper rightmost cell of $P$, respectively.  
 \end{definition}

\begin{definition}\label{def:hEv}
 The {\itshape vertical (horizontal)} path $v(P)$ (respectively $h(P)$) is the path - if it exists - internal to $P$, running from $S$ to $E$, and starting  with a north step $n$ (respectively $e$), where every side has maximal length (see Figure \ref{pOpUpR}).
\end{definition}

From now on, in order to make the decomposition more understandable, in the graphical representation the path will be represented using lines rather that cells.
In practice, to represent the path we use a line joining the centers of the cells, more precisely a dashed line to represent $v(P)$, and a solid line to represent $h(P)$.
We remark that our definition does not work if the first column (resp. row) of $P$ is made of one cell, and then in this case we set by definition that $v(P)$ and $h(P)$ coincide (Figure \ref{pOpUpR} $(d)$). Henceforth, if there are no ambiguities we will write $v$ (resp. $h$) in place of $v(P)$ (resp. $h(P)$).
So, by definition, a cell $V_i$ of $v$ (or $H_i$ of $h$) could correspond to one of two possible types of changes of direction, more precisely
\begin{description}
 \item[-] to a change {\em e-n} if $(x_{V_i}+1,y_{V_i})\notin P$ (resp. $(x_{H_i}+1,y_{H_i})\notin P$);
 \item[-] to a change {\em n-e} if $(x_{V_i},y_{V_i}+1)\notin P$ (resp. $(x_{H_i},y_{H_i}+1)\notin P$).
\end{description}

These two paths individuate two distinct types of cells into the polyomino at every change of direction. So, considering $h(P)=(H_1=S,\cdots, H_n=E)$ (respectively $v(P)=(V_1=S,\cdots, V_n=E)$), we can characterize each cell of $P$ that is not in $h(P)$ (or $v(P)$) as follows:
For every cell $B\in P$, $B\notin h(P)$ (resp. $B\notin v(P)$), there exists an index $i$, $1<i\leq n$, such that
\begin{center}
 $x_B<x_{H_i}\,\,$ and $\,\,y_B=y_{H_i}\,\,\,\,\,\,$ or $\,\,\,\,\,\,x_B>x_{H_i}\,\,$ and $\,\,y_B=y_{H_i}$
\end{center}
\begin{center}
 (resp. $\,\,x_B<x_{V_i}\,\,$ and $\,\,y_B=y_{V_i}\,\,\,\,\,\,$ or $\,\,\,\,\,\,x_B>x_{V_i}\,\,$ and $\,\,y_B=y_{V_i}\,$).
\end{center}
We say that in the first case $B$ is a cell of type {\em left-top}, denoted with $B^{\lrcorner}$ and in the second case that $B$ is a cell of type {\em right-bottom}, denoted with $B^\ulcorner$. The reader can observe in Figure \ref{typeOfCells} that the cell $B$ is an example of cell $B^\lrcorner$, in fact $x_B<x_{H_5}\,\,$ and $\,\,y_B=y_{H_5}$ and that $B'$ is an example of cell ${B'}^\ulcorner$, in fact $x_B<x_{H_7}\,\,$ and $\,\,y_B=y_{H_7}$.

\begin{figure}[htbp]
\begin{center}
\includegraphics[width=9cm]{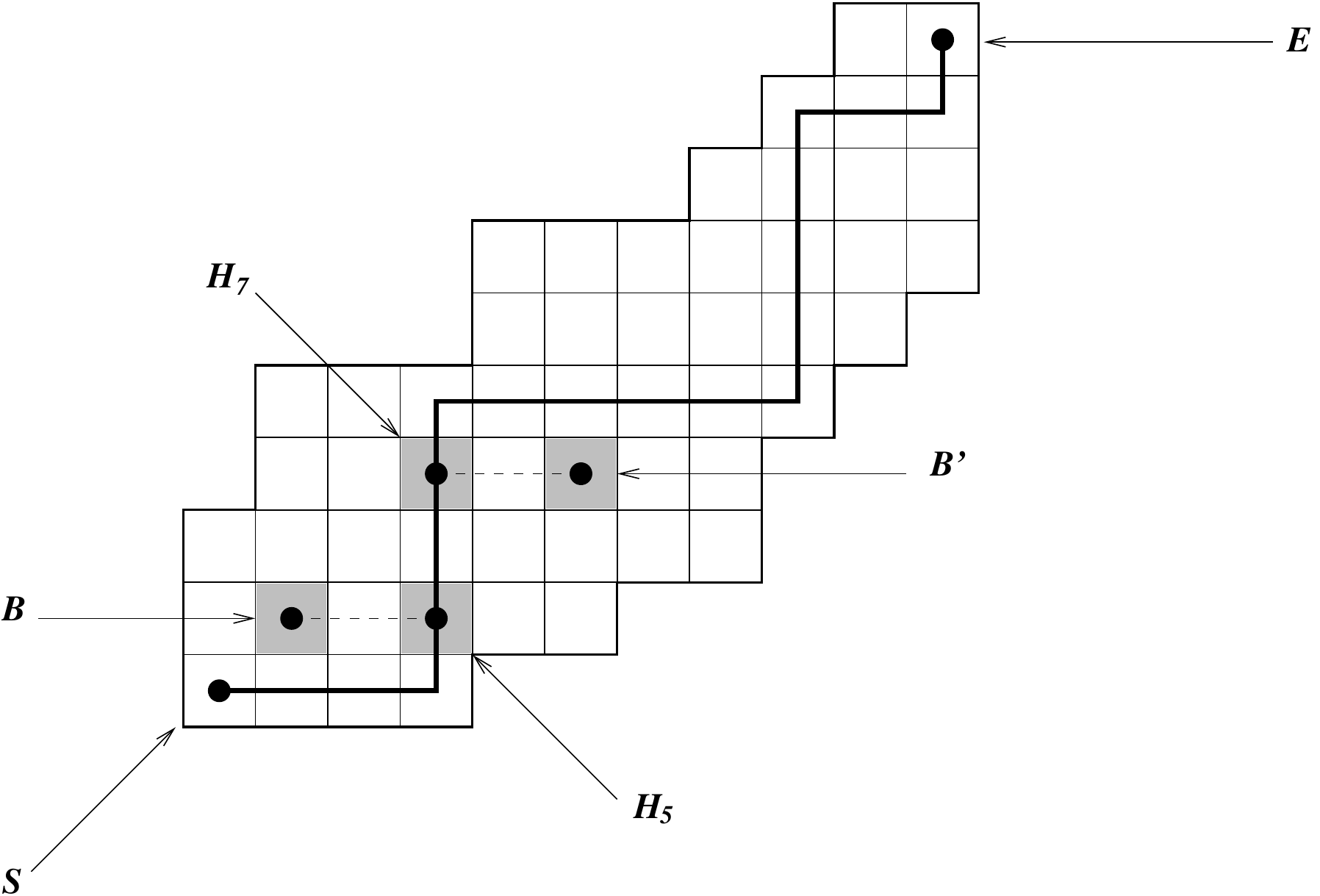}
\caption{An example of $4$-parallelogram polyomino and the internal path  $h={\tiny(H_1=S,H_2,\cdots,H_{20}=E)}$.}
\label{typeOfCells}
\end{center}
\end{figure}
  
Now, we are ready to prove the following important proposition.
\begin{proposition}\label{convexitydegree}
The convexity degree of a parallelogram polyomino $P$ is equal to the minimal number of changes of direction required to any path running from $S$ to $E$.
\end{proposition}

\begin{proof}
 Let $P$ be a polyomino and let be $k$ the minimal number of changes of direction among $h$ and $v$. We want to prove that for every two cells of $P$, $A$ and $A'$ different from $S$ and $E$, exists a path $\pi_{AA'}$ having at most $k$ changes of direction.
 We have to take into consideration three different cases.
 \begin{enumerate}
  \item {\em Both $A$ and $A'$ belong to $v$ (or $h$).}\\
  This case is trivial because the path running from $A$ to $A'$, $\pi_{AA'}$, is a subpath of $v$ (or $h$), so the number of changes of direction is less than or equal to $k$.
  \item {\em Only one between $A$ and $A'$ belongs to $v$ (or $h$).}\\
  We can assume without loss of generality that $A'$ is the cell that belongs to $v$ (or $h$) and that $A$ is a cell of type $A^\lrcorner$. Then, there exists an index $i$ such that $x_B<x_{V_i}\,\,$ and $\,\,y_B=y_{V_i}$ (or $x_B<x_{H_i}\,\,$ and $\,\,y_B=y_{H_i}$), so with $x_{V_i}-x_B$ (or $x_{H_i}-x_B$) steps $e$ we can reach the path $v$ (or $h$) with only one change of direction but after $v$ (or $h$) have changed its direction at least once. At this point it is easy to see that the path $\pi_{AA'}$ has at most $k$ changes of direction, the first to reach the path $v$ (or $h$) and the subsequent ones are those made by the subpath of $v$ (or of $h$) to the cell $A'$.
  \item {\em Neither $A$ nor $A'$ belong to $v$ (or $h$).}\\
  The proof is similar to that one of the previous case.
 \end{enumerate}
\end{proof}

We are now ready to establish an important property of the paths $v$ and~$h$.
\begin{proposition}
 The numbers of changes of direction that $h$ and $v$ require to run from $S$ to $E$ may differ at most by one.
\end{proposition}

The proof is analogous of that one of the previous property and it is left to the reader. Also the following property is straightforward.

\begin{proposition}
 A polyomino $P$ is $k$-parallelogram if and only if at least one among $v(P)$ and $h(P)$ has at most $k$ changes of direction.
\end{proposition}

We will begin our study with the class $\mathbb{P}_k$ of $k$-parallelogram polyominoes where the convexity degree is exactly equal to $k\geq 0$. Then the enumeration of $\poly_k$ will readily follow. According to our definition, $\mathbb{P}_0$ is made of horizontal and vertical bars of any length.
We further notice that, in the given parallelogram polyomino $P$,  there may exist a cell starting from which the two paths $h$ and $v$ are superimposed (see Figure \ref{pOpUpR} $(b)$, $(c)$). In this case, we denote such a cell by $C(P)$ (briefly, $C$)\label{cellC}.  Clearly $C$ may even coincide with $S$ (see Figure \ref{pOpUpR} $(d)$). If such cell does not exist, we assume that $C$ coincides with  $E$. Figure \ref{pOpUpR} depicts the various position of the cell $C$ into a parallelogram polyomino.

\begin{figure}[htbp]
  \begin{center}
     \includegraphics[width=13cm]{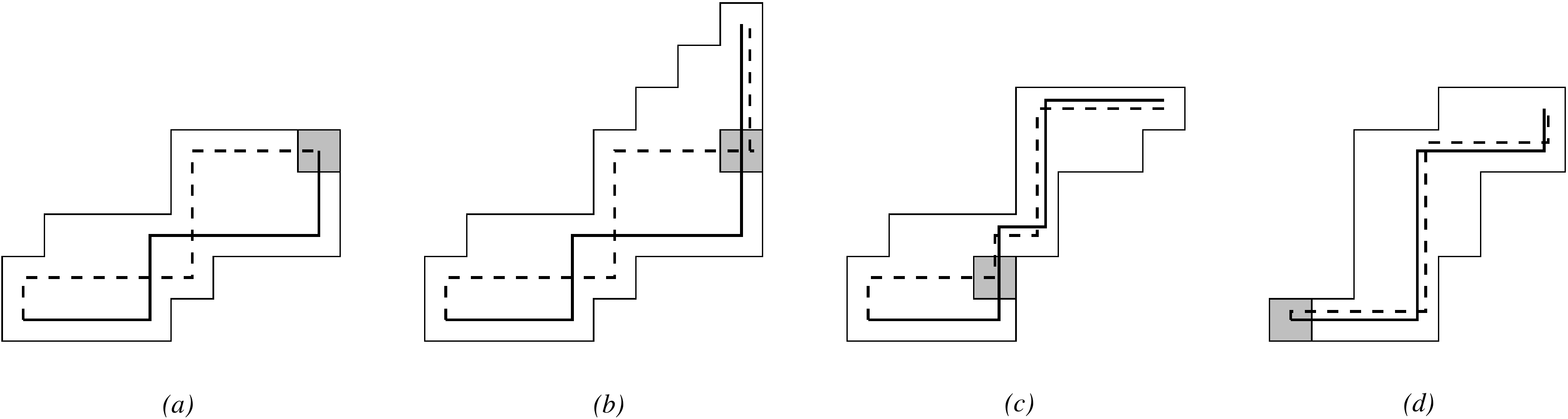}
    \caption{The paths $h$ (solid line) and $v$ (dashed line) in a parallelogram polyomino, where the cell $C$ has been highlighted; $(a)$ a polyomino in $\overline{\mathbb{P}}_3$; $(b)$ a polyomino in $\mathbb{P}_3^U$
    ; $(c)$ a polyomino in $\mathbb{P}_4^R$; $(d)$ a polyomino in $\mathbb{P}_3^U$ where $C$ coincides with $S$.}
    \label{pOpUpR}
  \end{center}
\end{figure}

From now on, unless otherwise specified, we will always assume that $k\geq 1$. Let us give a classification of the polyominoes in $\mathbb{P}_k$, based on the position of the cell $C$ inside the polyomino.
\begin{definition}
 A polyomino $P$ in $\mathbb{P}_k$ is said to be
\begin{enumerate}
 \item  a flat $k$-parallelogram polyomino if $C(P)$ coincides with $E$. The class of these polyominoes will be denoted by ${\overline{\mathbb{P}}_k}$.
 
 \item an up (resp. right) $k$-parallelogram polyomino, if the cell $C(P)$ is distinct from $E$ and $h$ and $v$ end with a north (resp. east) step. The class of up (resp. right) $k$-parallelogram  polyominoes will be denoted by  $\mathbb{P}_k^U$ (resp. $\mathbb{P}_k^R$).
 
 \end{enumerate}
 
\end{definition}
Figure \ref{pOpUpR} $(a)$ depicts a polyomino in ${\overline{\mathbb{P}}_k}$, while Figures \ref{pOpUpR} $(b)$, $(c)$, and $(d)$ depict polyominoes in $\mathbb{P}_k^U$ and $\mathbb{P}_k^R$.
According to this definition all rectangles having width and height greater than one belong to $\overline{\mathbb{P}}_1$.
 
The reader can easily check that up (resp. right) $k$-parallelogram polyominoes where the cell $C(P)$ is distinct from $S$ can be characterized as those parallelogram polyominoes where $h$ (resp. $v$) has $k$ changes of direction and $v$ (resp. $h$) has $k+1$ changes of direction.

\medskip

Now we present a unique decomposition of polyominoes in $\mathbb{P}_k$, based on the following idea: given a polyomino $P$, we are able to detect -- using the paths $h$ and $v$ -- a set of paths on the boundary of $P$, that uniquely identify the polyomino itself.

More precisely\label{dec}, let $P$ be a polyomino of $\mathbb{P}_k$; the cells of the path $h$ (resp. $v$) that coincide with a change of direction have at least an edge on the boundary of $P$, in particular if a cell corresponds to a change of direction {\em e-n} (resp. {\em n-e}) individuates an $e$ (resp. $n$) step on the upper (resp. lower) boundary of $P$. So we can say that the path $h$ (resp. $v$)  when encountering the boundary of $P$, determines $m$ (resp. $m'$) steps where $m$ (resp. $m'$) is equal to the number of changes of directions of $h$ (resp. $v$) plus one.
To refer to these steps we agree that the step encountered by $h$ (resp. $v$) for the $i$th time is called $X_i$ or $Y_i$ according if it is a horizontal or vertical one (see Fig. \ref{decomposition}).  We point out that if $P$ is flat all steps $X_i$ and $Y_i$ are distinct, otherwise there may be some indices $i$ for which $X_i=X_{i+1}$ (or $Y_i=Y_{i+1}$), and this happens precisely with the steps determined after the cell $C(P)$ (see Fig. \ref{sama1+sama2} $(b)$, $(c)$). The case $C(P)=S$ can be viewed as a degenerate case where the initial sequence of north (resp. east) steps of $v$ (resp. $h$) has length zero 
and we have to give an alternative definition of these steps, see Figure \ref{sama1+sama2} $(b)$:
\begin{description}
\item{i)} if the first column is made of one cell, i.e. $v$ coincides with $h$, we set $X_1$ to be equal to the leftmost east step of the upper path of $P$, and $Y_2, X_3, \ldots $ are determined as usual by $h$ encountering the boundary of $P$;
\item{ii)} if the lowest row is made of one cell, i.e. $h$ coincides with $v$, we set $Y_1$ to be equal to the leftmost north step of the lower path of $P$, and $X_2, Y_3, \ldots $ are determined as usual by $v$ encountering the boundary of $P$.
\end{description}



\begin{figure}[htbp]
  \begin{center}
     \includegraphics[width=14cm]{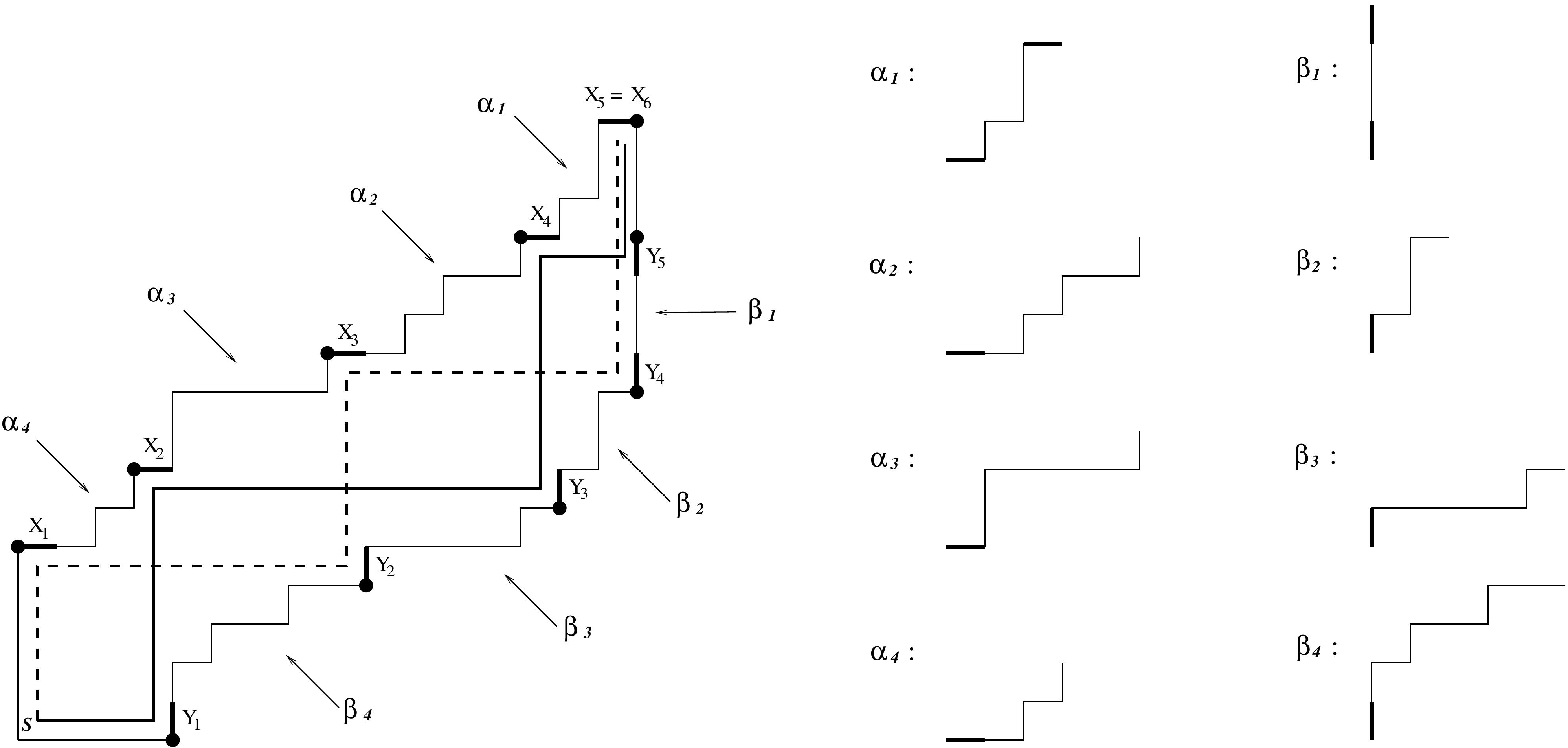}
    \caption{Decomposition of a polyomino of $\mathbb{P}_4^U$.}
    \label{decomposition}
  \end{center}
\end{figure}

Now we decompose the upper (resp. lower) path of $P$ in $k$ (possibly empty) subpaths $\alpha_1, \ldots , \alpha_k$ (resp. $\beta_1, \ldots , \beta_k$) using the following rule: $\alpha_1$ (resp. $\beta_1$) is the path running from the beginning of $X_{k}$ to end of $X_{k+1}$ (resp. from the beginning of $Y_{k}$ to $Y_{k+1}$);
let us consider now the $k-1$ (possibly empty) subpaths, $\alpha_i$ (resp. $\beta_i$) from the beginning of $X_{k+1-i}$ (resp. $Y_{k+1-i}$) to the beginning of $X_{k+2-i}$ (resp. $Y_{k+2-i}$), for $i=2\cdots k$. We observe that these paths are ordered from the right to the left of $P$.
For simplicity we say that a path is {\itshape flat} if it is composed of steps of just one type.


\begin{figure}[htbp]
  \begin{center}
    \includegraphics[width=14.5cm]{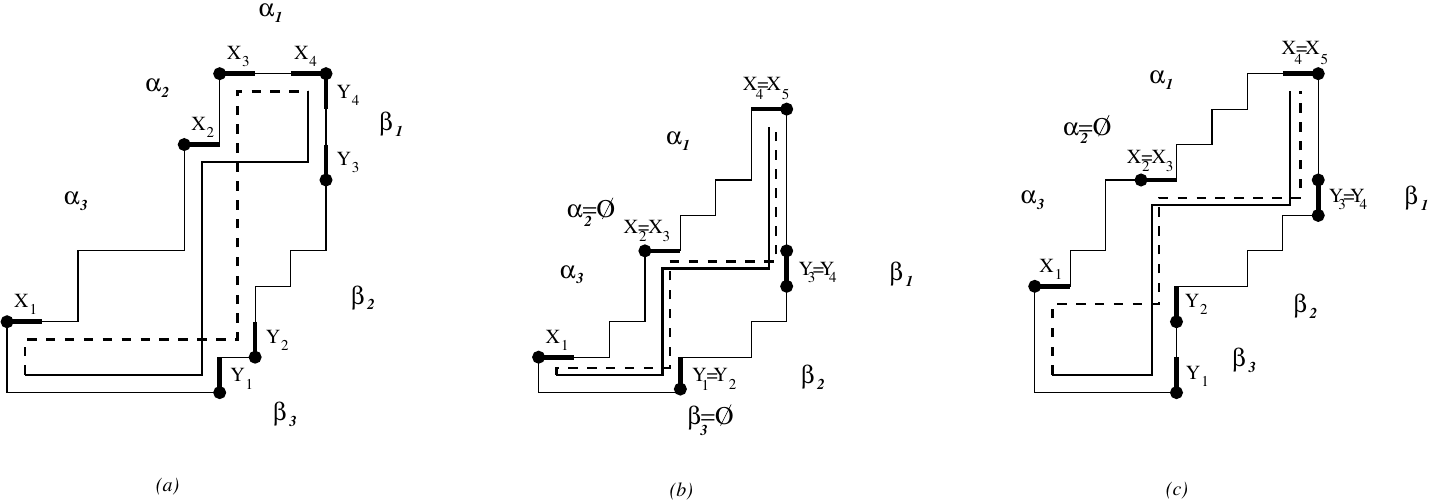}
    \caption{$(a)$ A polyomino $P\in \overline{\mathbb{P}}_3$ in which $\alpha_1$ and $\beta_1$ are flat and each other path is non empty and non flat. $(b)$ A polyomino $P\in \mathbb{P}_3^U$ where: $\beta_3$ is empty,  $\alpha_2$ is empty and $\beta_1$ is equal to a north unit step. $(c)$ A polyomino $P\in \mathbb{P}_3^U$ where $\beta_3$ is flat, $\alpha_2$ is empty and  $\beta_1$ is equal to a north unit step.}
    \label{sama1+sama2}
  \end{center}
\end{figure}

The following proposition provides a characterization of the polyominoes of $\mathbb{P}_k$ in term of the paths  $\alpha_1, \ldots , \alpha_{k}, \beta_1, \ldots , \beta_{k}$, see Figure \ref{decomposition}.

\begin{proposition}\label{propCamm}
 A polyomino $P$ in $\mathbb{P}_k$ is uniquely determined by a sequence of (possibly empty) paths $\alpha_1, \ldots , \alpha_{k}$, $\beta_1, \ldots , \beta_{k}$, each of which made by north and east unit steps. Moreover, these paths have to satisfy the following properties:

\begin{itemize}
 \item $\alpha_i$ and $\beta_{i+1}$ must have the same width, for every $i\neq 1$; if $i=1$, we have that $\alpha_1$ is always non empty and the width of $\alpha_1$ is equal to the width of $\beta_2$ plus one;
 \item $\beta_i$ and $\alpha_{i+1}$ must have the same height, for every $i\neq 1$; if $i=1$, we have that $\beta_1$ is always non empty and the height of $\beta_1$ is equal to the width of $\alpha_2$ plus one;


 \item if $\alpha_i$ ($\beta_i$) is non empty then it starts with an east (north) step, $i\geqslant 1$. In particular, for $i=1$, if  $\alpha_1$ ($\beta_1$) is different from the east (north) unit step, then it must start and end with an east (north) step.
 \end{itemize}

We want to notice that the semi-perimeter of $P$ is obtained as the sum $|\alpha_1|+ |\alpha_{2}|_e + \ldots + |\alpha_{k}|_e + |\beta_1|+|\beta_2|_n+ \ldots +|\beta_{k}|_n$, and that follows directly from our construction.

\end{proposition}

The reader can easily check the decomposition of a polyomino of $\mathbb{P}_4^U$ in Figure \ref{decomposition}.
For clarity sake, we need to remark the following consequence of Proposition~\ref{propCamm}:

\begin{corollary}\label{empty}
Let $P\in \mathbb{P}_k$ be encoded by the paths  $\alpha_1, \ldots , \alpha_{k}, \beta_1, \ldots , \beta_{k}$. We have:
\begin{description}
\item{-} for every $i>1$, we have that $\alpha_i$ ($\beta_i$) is empty if and only if $\beta_{i+1}$ ($\alpha_{i+1}$) is empty or flat;
\item{-} $\alpha_1$ ($\beta_1$) is equal to the east (north) unit step if and only if $\beta_{2}$ ($\alpha_2$) is empty or flat.
\end{description}
\end{corollary}


Figure \ref{sama1+sama2} $(a)$ shows the decomposition of a flat polyomino, $(b)$ shows the case in which $C(P)=S$, so we have that $\beta_3$ is empty, then  $\alpha_2$ is empty, hence $\beta_1$ is a unit north step. Figure \ref{sama1+sama2} $(c)$ shows the case in which $h$ and $v$ coincide after the first change of direction and so we have that $\beta_3$ is flat, then $\alpha_2$ is empty and $\beta_1$ is a unit north step.

Now we provide another characterization of the classes of flat, up, and right polyominoes of $\mathbb{P}_k$ which directly follows from Corollary \ref{empty} and will be used for the enumeration of these objects.
\begin{proposition}\label{flatup}
Let $P$ be a polyomino in $\mathbb{P}_k$. We have:
\item{i)} $P$ is flat if and only if $\alpha_1$ and $\beta_1$ are flat and they have length greater than one. It follows from \ref{empty} that all $\alpha_i$ and $\beta_i$ are non empty paths, $i=2, \ldots , k$.
\item{ii)} $P$ is up (right) if and only if $\beta_1$ ($\alpha_1$) is flat and $\alpha_1$ ($\beta_1$) is non flat.
\end{proposition}
The reader can see examples of the statement of Proposition \ref{flatup} {\em i)} in Figure \ref{sama1+sama2} $(a)$, and of
Proposition \ref{flatup} {\em ii)} in Figure \ref{sama1+sama2} $(b)$ and $(c)$.

As a consequence of Proposition \ref{propCamm}, from now on we will encode every polyomino $P\in \mathbb{P}_k$ in terms of the two sequences:
$$
{\cal A}(P)= \left( \alpha_1, \beta_2, \alpha_3, \ldots , \theta_k \right),
$$
with $ \theta=\alpha$ if $k$ is odd, otherwise $\theta=\beta$, and
$$
{\cal B}(P)= \left( \beta_1, \alpha_2, \beta_3, \ldots , \overline{\theta}_k \right) ,
$$
where $ \overline{\theta}=\alpha$ if and only if $\theta=\beta$. The dimension of ${\cal A}$ (resp. ${\cal B}$) is given by $|\alpha_1|+|\beta_2|_n+|\alpha_3|_e+\ldots $ (resp. $|\beta_1|+|\alpha_2|_e+|\beta_3|_n+\ldots $).
In particular, if $C(P)=S$ and $P$ is an up (resp. right) polyomino then ${\cal B}(P)=\left( \beta_1, \emptyset, \ldots , \emptyset \right )$, (resp. ${\cal A}(P)=\left( \alpha_1, \emptyset, \ldots , \emptyset \right )$) where $\beta_1$ (resp. $\alpha_1$) is the north (resp. east) unit step.

\section{Enumeration of the class $\poly_k$}

This section is organized as follows:
first, we furnish a method to pass from the generating function of the class $\mathbb{P}_k$ to the generating function of $\mathbb{P}_{k+1}$, $k>1$. Then, we provide the enumeration of the trivial cases, i.e. $k=0,1$, and finally apply the inductive step to determine the generating function of  $\mathbb{P}_k$. The enumeration of $\poly_k$ is readily obtained by summing all the generating functions of the classes $\mathbb{P}_s$, $s \leq k$. 
\subsection{Generating function of the class $\poly_k$}
The following theorem establishes a criterion for translating the decomposition of Proposition \ref{propCamm} into generating functions.

\begin{theorem}\label{teo}
 \item{i)} A polyomino $P$ belongs to $\mathbb{P}_2$ if and only if it is obtained from a polyomino of $\mathbb{P}_{1}$ by adding two new paths $\alpha_2$ and $\beta_2$, which cannot be both empty, where the height of $\alpha_2$ is equal to the height of $\beta_{1}$ minus one, and the width of $\beta_2$ is equal to the width of $\alpha_{1}$ minus one.
 \item{ii)} A polyomino $P$ belongs to $\mathbb{P}_k$, $k>2$, if and only if it is obtained from a polyomino of $\mathbb{P}_{k-1}$ by adding two new paths $\alpha_k$ and $\beta_k$, which cannot be both empty, where $\alpha_k$ has the same height of $\beta_{k-1}$ and $\beta_k$ has the same width of $\alpha_{k-1}$.
\end{theorem}
We can see an example of the statement i) of Theorem \ref{teo} in Figure \ref{costruzioneP2}. In $(a)$ we have a polyomino $P\in \mathbb{P}_2^R$ obtained adding to $P'\in \mathbb{P}_1^R$ a path $\beta_2$ with width equal to $3$ and a path $\alpha_2$ with height equal to $2$. While, in $(b)$ we have a polyomino $P\in \mathbb{P}_2^R$ obtained adding to $P'\in \mathbb{P}_1^R$ a path $\beta_2$ with width equal to $0$, $\beta_2$ is flat, and a path $\alpha_2$ with height equal to $2$. We want to notice that in this last case $\beta_2$ could also be empty.

\begin{figure}[htbp]
  \begin{center}
    \includegraphics[width=10cm]{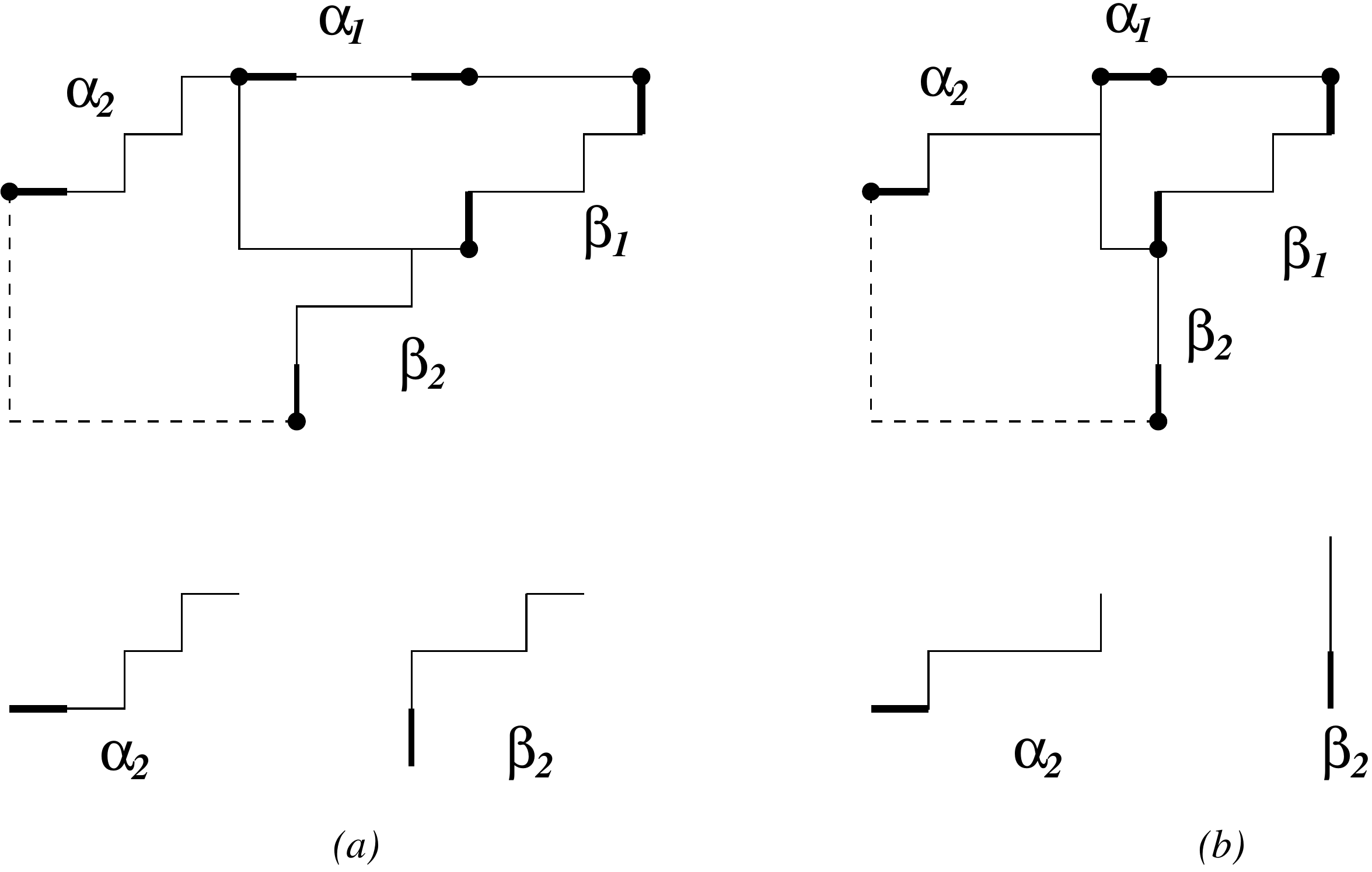}
    \caption{$(a)$ A polyomino $P\in \mathbb{P}_2^R$ in which $\alpha_1$ is flat and $\beta_1$ are flat and every other path is non empty and non flat. $(b)$ A polyomino $P\in \mathbb{P}_2^R$ in which $\alpha_1$ is equal to an east unit step.}
    \label{costruzioneP2}
  \end{center}
\end{figure}

The proof of Theorem \ref{teo} directly follows from our decomposition in Proposition \ref{propCamm}, where the difference between the case $k=2$ and the case $k>2$ is clearly explained. We would like to point out that if $P$ belongs to $\overline{\mathbb{P}}_k$, then neither $\alpha_k$ nor $\beta_k$ can be empty or flat. Following the statement of Theorem \ref{teo}, to pass from $k\geq 1$ to $k+1$ we need to introduce following generating functions:

\begin{description}
 \item {i)} the generating function of the sequence ${\cal A}(P)$. Such a function is denoted by $A_k(x,y,z)$ for up, and by $\overline{A}_k(x,y,z)$ for flat $k$-parallelogram polyominoes, respectively, and, for each function, $x+z$ keeps track of the dimensions of ${\cal A}(P)$, and
    $z$ keeps track of the width of $\theta_k$ if $k$ is odd and of the height of $\theta_k$ if $k$ is even.

\item{ii)} the generating function of the sequence  ${\cal B}(P)$. Such a function is denoted by $B_k(x,y,t)$ for up, and by $\overline{B}_k(x,y,t)$ for flat $k$-parallelogram polyominoes, respectively, and here $y+t$ keeps track of the dimensions of ${\cal B}(P)$, and the variable $t$ keeps track of the height of $\theta_k$ if $k$ is odd and of the width of $\theta_k$ if $k$ is even.
\end{description}

By Proposition \ref{propCamm}, the generating functions $Gf_k^U(x,y,z,t)$,\\
$Gf_k^R(x,y,z,t)$ and $\overline{Gf}_k(x,y,z,t)$, of the classes $\mathbb{P}_k^U$, $\mathbb{P}_k^R$, and $\overline{\mathbb{P}}_k$, respectively, are clearly obtained as follows:
\begin{eqnarray}
Gf_k^U(x,y,z,t) &=& A_k(x,y,z) \cdot B_k(x,y,t) \label{1}\\
\nonumber \\
\overline{Gf}_k(x,y,z,t) &=& \overline{A}_k(x,y,z) \cdot \overline{B}_k(x,y,t) \label{2}\\
\nonumber \\
Gf_k(x,y,z,t) &=&Gf_k^U(x,y,z,t)+Gf_k^R(y,x,t,z)+\overline{Gf}_k(x,y,z,t) \label{3}\\
\nonumber 
\end{eqnarray}
Then, setting $z=t=y=x$, we have the generating functions according to the semi-perimeter. Since $Gf_k^U(x,y,z,t)=Gf_k^R(y,x,t,z)$, for all $k$, then starting from now, we will study only the flat and the up classes.

In this work we use regular expressions to encode the possible paths of the sequences ${\cal A}(P)$ and ${\cal B}(P)$ in order to calculate the corresponding generating functions by applying standard methods, namely the so called Sch\"utzenberger methodology \cite{sch}.

\paragraph{The case $k=0$.}
The class $\mathbb{P}_0$ is simply made of horizontal and vertical bars of any length. We keep this case distinct from the others since it is not useful for the inductive step, so we simply use the variables $x$ and $y$, which keep track of the width and the height of the polyomino, respectively. The generating function is trivially equal to $$Gf_0(x,y)=xy+\frac{x^2y}{1-x}+\frac{xy^2}{1-y} \, ,$$ where the term $xy$ corresponds to the unit cell, and the other terms to the horizontal and vertical bars, respectively.


\begin{figure}[htbp]
  \begin{center}
     \includegraphics[width=8cm]{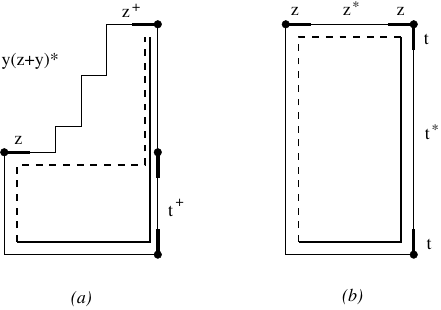}
    \caption{$(a)$ A polyomino $\in \mathbb{P}_1^U$ and $(b)$ a polyomino in $\overline{\mathbb{P}}_1$.}
    \label{upEflat1convex}
  \end{center}
\end{figure}

\paragraph{The case $k=1$.}  Following our decomposition and Figure \ref{upEflat1convex}, we easily obtain
$$A_1(x,y,z)=\frac{z^2y}{(1-z-y)(1-z)}, \qquad B_1(x,y,t)=t +\frac{t^2}{1-t} \, .$$
We point out that we have written $B_1$ as the sum of two terms because, according to Corollary \ref{empty}, we have to treat the case when $\beta_1$ is made by a north unit step separately from the other cases. To this aim, we set $\hat{B}_1(x,y,t)=\frac{t^2}{1-t}$.
Moreover, we have
$$\overline{A}_1(x,y,z)=\frac{z^2}{1-z}, \qquad \overline{B}_1(x,y,t)=\frac{t^2}{1-t} \, .$$
According to \eqref{1} and \eqref{2}, we have that
\begin{equation*}
Gf_1^U(x,y,z,t) = \frac{t y z^2}{(1 - t) (1 - z) (1 - y - z)} \qquad
\overline{Gf}_1(x,y,z,t) = \frac{t^2 z^2}{(1 - t) (1 - z)} \, .\\
\end{equation*}
Now, according to \eqref{3}, and setting all variables equal to $x$, we have the generating function of $1$-parallelogram polyominoes
$$Gf_1(x)=\frac{x^4 (2 x-3)}{(1 - x)^2 (1 - 2 x)} \, .$$

\paragraph{The case $k=2$.} Now we can use the inductive step, recalling that the computation of the case $k=2$ will be slightly different from the other cases, as explained in Theorem \ref{teo}. Using the decomposition in Figure \ref{pippok2} we can calculate the generating functions
$$A_2(x,y,z)=z\cdot A_1\left(x,y,\frac{x}{1-z}\right)=\frac{x^2yz}{(1-x-y-z+yz)(1-x-z)}$$
$$B_2(x,y,t)=\frac{y}{1-t} +t\cdot \hat{B}_1\left(x,y,\frac{y}{1-t}\right)=\frac{y-y^2}{(1-y-t)}=y+\frac{yt}{1-y-t}$$
$$\overline{A}_2(x,y,z)=z\cdot\overline{A}_1\left(x,y,\frac{x}{1-z}\right)=\frac{x^2z}{(1-z)(1-x-z)}$$ $$\overline{B}_2(x,y,t)=t\cdot\overline{B}_1\left(x,y,\frac{y}{1-t}\right)=\frac{y^2t}{(1-t)(1-y-t)} \, .$$

\begin{figure}[htbp]
  \begin{center}
    \includegraphics[width=13cm]{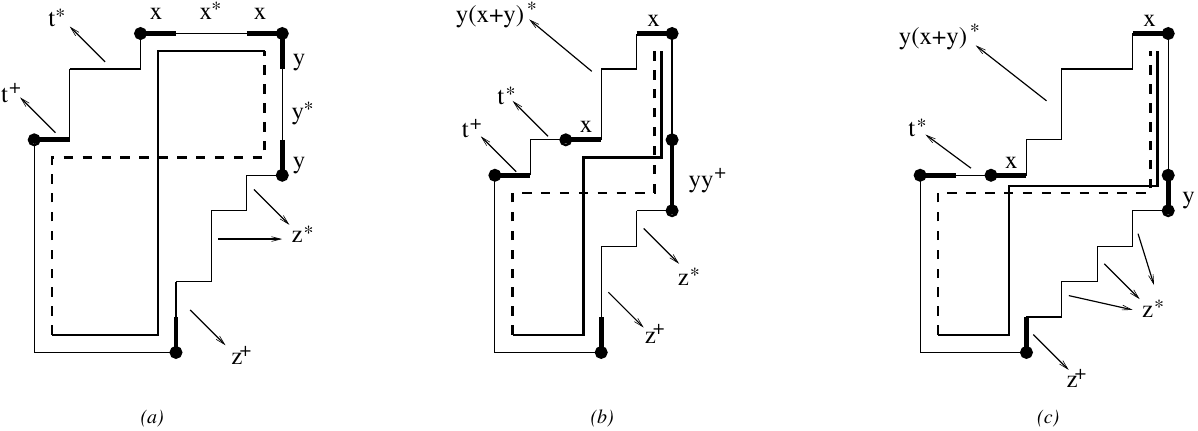}
    \caption{$(a)$ A polyomino in $\overline{\mathbb{P}}_2$, $(b)$ a polyomino in $\mathbb{P}_2^U$ in which $\beta_1$ has at least two north steps and $(c)$ a polyomino in $\mathbb{P}_2^U$ in which $\beta_1$ is equal to an unit north step.}
    \label{pippok2}
  \end{center}
\end{figure}

We observe that the performed substitutions allow us to add the contribution of the terms $\alpha_2$ and $\beta_2$ from the generating functions obtained for $k=1$. Then, using formulas \eqref{1}, \eqref{2} and \eqref{3}, and setting all variables equal to $x$, it is straightforward to obtain the generating function according to the semi-perimeter:
$$Gf_2(x)=\frac{x^5 (2 - 5x + 3x^2 - x^3)}{(1 - x)^2 (1 - 2 x)^2 (1 - 3 x + x^2)}\,\,.$$

\paragraph{The case $k>2$.} The generating functions for the case $k>2$ are obtained in a similar way. Here, for simplicity sake, we set $\hat{B}_k(x,y,t)=B_k(x,y,t)-y$; this trick will help us treat separately the case when $\beta_1$ is made by a north unit step. Then we have
\begin{eqnarray}
\overline{A}_k(x,y,z)&=&\frac{z}{1-z}\cdot\overline{A}_{k-1}\left(x,y,\frac{x}{1-z}\right)    \label{4} \\
\overline{B}_k(x,y,t)&=&\frac{t}{1-t}\cdot\overline{B}_{k-1}\left(x,y,\frac{y}{1-t}\right)  \label{5} \\
A_k(x,y,z)&=&\frac{z}{1-z}\cdot A_{k-1}\left(x,y,\frac{x}{1-z}\right)   \label{6} \\
B_k(x,y,t)&=&\frac{y}{1-t} +\frac{t}{1-t}\cdot \hat{B}_{k-1}\left(x,y,\frac{y}{1-t}\right)\,\,\, .   \label{7}
\end{eqnarray}
We remark that \eqref{4}, \eqref{5}, \eqref{6} and \eqref{7} slightly differ from the respective formulas for $k=2$, according to the statement of Theorem \ref{teo}.

The performed calculations and in particular the substitutions suggest that the above formulas can be written also using continued fractions \cite{flajContFract}, which is a less compact way, but can give to these expressions a deeper combinatorial meaning. For example, instead of \eqref{4} we can write:
$${\tiny \overline{A}_k(x,x,z)=x^{k}z\cdot\left(\frac{1}{\left.\,\,\,\,\,\,\,\,\,\,\,\,\,\,\,1-\frac{x}{1-\frac{\vdots}{1-\frac{x}{1-z}}}\right\}\mbox{{$(k-2)$-times}}}\right)^2\cdot\frac{1}{\left.\,\,\,\,\,\,\,\,\,\,\,\,\,\,\,1-\frac{x}{1-\frac{\vdots}{1-\frac{x}{1-z}}}\right\}\mbox{{$(k-1)$-times}}}}\,\, .$$

The other expressions are quite similar.

\subsection[A formula for the number of $\poly_k$]{A formula for the number of $k$-parallelogram polyominoes}
The formulas found in the previous section allow us in principle to obtain an expression for the generating function of $\mathbb{P}_k(x)$, for all $k>2$. However, the continued fractions representation suggests us a simpler way to express the generating function of the sequences $\overline{A}_k, \overline{B}_k, A_k \,\,\mbox{and}\,\, B_k$ as a quotient of polynomials, using the notion of {\em Fibonacci polynomials}.

First we need to give the following recurrence relation:
\begin{definition}
$$
\left\{
\begin{array}{lll}
F_0(x,z) = F_1(x,z) &=& 1\\
F_2(x,z) &=& 1-z\\
F_k(x,z) &=& F_{k-1}(x,z)-xF_{k-2}(x,z) \,\,\, .
\end{array}
\right.
$$
\end{definition}

\begin{remark}
 Let us observe that the use of three initial conditions instead of two is required to obtain the desired sequence $F_0,F_1,\cdots$. In particular setting only $F_0=F_1=1$ we would have $F_2=1-x$ instead of $1-z$ and we need also of the term $F_0$ because of it appears in the final expression of the generating function. 
\end{remark}
These objects are already known as {\em Fibonacci polynomials} \cite{knuth}

\begin{remark}
To avoid any confusion, let us notice that Fibonacci polynomials are perhaps more commonly known with the expression given by
$$
\left\{
\begin{array}{lll}
F_0(x) = F_1(x) &=& 1\\
F_k(x) &=& F_{k-1}(x)+xF_{k-2}(x) \,\,\, .
\end{array}
\right.
$$
\end{remark}

In the sequel, unless otherwise specified, we will denote $F_k(x,x)$ with $F_k$. Notice that $F_k(-1,-1)$ give the $kth$ Fibonacci number.

The closed Formula of $F_k$ obtained using standard methods is:
$$F_k=\frac{b(x)^{k+1}-a(x)^{k+1}}{\sqrt{1-4x}} \, .$$
where $a(x)$ and $b(x)$ are the solutions of the equation $X^2-X+x=0$, \ie $a(x)=\left(\frac{1-\sqrt{1-4x}}{2}\right)$ and $b(x)=\left(\frac{1+\sqrt{1-4x}}{2}\right)$.

These polynomials have been widely studied, and have several combinatorial properties. Below we list just a few of these properties, the ones that we will use in order to provide alternative expressions for formulas $\overline{A}_k, \overline{B}_k, A_k \,\,\mbox{and}\,\, B_k$.

We start to provide some elementary identities involving Fibonacci polynomials.
\begin{proposition}
For any $k\geq 1$ the following relations hold
$$
 \begin{array}{lcl}
  F_k^2-xF_{k-1}^2&=&F_{2k}\\ 
  \nonumber \\
  F_{k+1}-xF_{k-1}&=&\frac{F_{2k+1}}{F_k}\\
  \nonumber \\
  F_{k-1}^2&=&x^{k+1}+F_kF_{k+2}\\
  \nonumber \\
  \frac{F_k}{F_{k+1}}&=&\frac{1}{\left.\,\,\,\,\,\,\,\,\,\,\,\,\,\,\,1-\frac{x}{1-\frac{\vdots}{1-\frac{x}{1-z}}}\right\}\mbox{{\tiny $(k-1)$-times}}}\,.
 \end{array}$$
\end{proposition}

\begin{proof}
 
These identities are obtained by performing standard computation, and using the following:
$$
\left\{
\begin{array}{lll}
a(x)+b(x) &=& 1\\
b(x)-a(x) &=& \sqrt{1-4x}\\
b(x)\cdot a(x) &=& x \,\,\, .
\end{array}
\right.
$$
Thus, we only show how we get the first equality, then the other ones can be proved in a similar way. For brevity sake we write $a$ instead of $a(x)$ and $b$ instead of $b(x)$.

$$
\begin{array}{llll}
F_k^2-xF_{k-1}^2&=&\frac{(b^{k+1}-a^{k+1})^2}{(\sqrt{1-4x})^2}-x\cdot\frac{(b^{k+1}-a^{k+1})^2}{(\sqrt{1-4x})^2}&=\\\\
 &=&\frac{b^{2k+2}+a^{2k+2}-2b^{k+1}+2xb^ka^k-xb^{2k}-xa^{2k}}{1-4x}&=\\
 & & & \\
 &=&\frac{b^{2k+1}(b-a)-a^{2k+1}(b-a)}{1-4x}&=\\
 & & & \\
 &=&\frac{b-a}{\sqrt{1-4x}}\cdot\frac{b^{2k+1}-a^{2k+1}}{\sqrt{1-4x}}&=\,\,F_{2k}\,\,.\\
\end{array}
$$

\end{proof}

In order to express the functions $A_k$, $B_k$, $\overline{A}_k$, and $\overline{B}_k$ in terms of the Fibonacci polynomials we need to state the following lemma:
\begin{lemma}\label{lemmaF}
For every $k\geq 1$ $$F_k\left(x,\frac{x}{1-z}\right)=\frac{F_{k+1}(x,z)}{1-z} \, .$$
\end{lemma}

\begin{proof} The proof is easily obtained by induction.\\
 {\bfseries Basis}: We show that the statement holds for $k=1$.
 $$F_1\left(x,\frac{x}{1-z}\right)=1=\frac{1-z}{1-z}=\frac{F_{2}(x,z)}{1-z}\,.$$
 {\bfseries Inductive}: Assume that Lemma \ref{lemmaF} holds for $k-1$. Let us show that it holds also for $k$, \ie
 $$F_k\left(x,\frac{x}{1-z}\right)=\frac{F_{k+1}(x,z)}{1-z}\,.$$
 Using the definition of $F_k(x,z)$ the left-hand side of the above equation can be rewritten as
 $$F_{k-1}\left(x,\frac{x}{1-z}\right)-xF_{k-2}\left(x,\frac{x}{1-z}\right)\,.$$
 Now, using the induction hypothesis, we obtain:
 $$\frac{F_{k}(x,z)}{1-z}-x\frac{F_{k-1}(x,z)}{1-z}=\frac{F_{k}(x,z)-xF_{k-1}(x,z)}{1-z}=\frac{F_{k+1}(x,z)}{1-z}\,\,.$$
\end{proof}

Letting $y=x$, we can write $A_1(x,z)=\frac{xz^2}{F_2(x,z)F_3(x,z)}$. Now, iterating Formula \eqref{4}, and using Lemma \ref{lemmaF}, we obtain
$$A_k(x,z) =\frac{zx^{k+1}}{F_{k+1}(x,z)F_{k+2}(x,z)} \,\,\, .$$
Performing the same calculations on the other functions we obtain:
\begin{eqnarray*}
B_k (x,z)           &=& \frac{xF_{k}}{F_{k+1}(x,z)}\\
\overline{A}_k(x,z) &=& \overline{B}_k(x,z)=\frac{zx^k}{F_{k}(x,z)\cdot F_{k+1}(x,z)} \,\,\,\, .
\end{eqnarray*}
From these new expressions for the functions $A_k$, $B_k$, $\overline{A}_k$, and $\overline{B}_k$, by setting all variables equal to $x$, we can calculate the generating function of the class $\mathbb{P}_k$ in an easier way:
$$Gf_k(x)=2A_k(x,x)B_k(x,x) + (\overline{A}_k)^2(x,x)$$
$$Gf_k(x)=\frac{2x^{k+3}F_k}{F_{k+1}^2F_{k+2}}+\frac{x^{2k+2}}{F_{k}^2F_{k+1}}\, .$$
Then we have the following:
\begin{theorem}\label{formulaP}
The generating function of $k$-parallelogram polyominoes $\poly_k$ is given by
\begin{eqnarray*}
P_{k}(x) =\sum^{k}_{n=0} Gf_n(x)= x^2\cdot\left(\frac{F_{k+1}}{F_{k+2}}\right)^2-x^2\cdot\left(\frac{F_{k+1}}{F_{k+2}}-\frac{F_{k}}{F_{k+1}}\right)^2 \, .\label{P}
\end{eqnarray*}

\end{theorem}

As an example, the generating functions of $\mathbb{P}_k$, for the first values of $k$ are:
%
%

$$
\begin{array}{ll}
P_0(x)=\frac{x^2 (1 + x)}{1-x} & P_1(x)=\frac{x^2 (1 - 2 x + 2 x^2)}{(-1 + x)^2 (1 - 2 x)}\\
\\
P_2(x)=\frac{ x^2 (1-x)(1 - 4 x + 4 x^2 + x^3)}{(1-2x)^2 (1 - 3 x + x^2)}
& P_3(x)=\frac{x^2 (1-2x) (1 - 6 x + 11 x^2 - 6 x^3 + 2 x^4)}{(1-x) (1-3x) (1 - 3 x + x^2)^2}\\
\end{array}
$$
The coefficients of $P_1$ are an instance of sequence $A000247$ \cite{Sl}, whose first few terms are:
$$0,3,10,25,56,119,246,501,1012,\cdots\,\,.$$

As one would expect we have the following corollary:
\begin{corollary}
Let $C(x)=\frac{1-\sqrt{1-4x}}{2x}$ be the generating function of Catalan numbers, we have: 
$$\lim_{k\to \infty}\,P_k(x)=C(x) \, .$$
\end{corollary}

\begin{proof}
We have that $C(x)$ satisfies the equation $C(x)=1+xC^2(x)$, and $a(x)b(x)=x$, $a(x)=xC(x)$, so we can write
$$F_k=\frac{1-x^{k+1}C^{2(k+1)}(x)}{C^{k+1}(x)\sqrt{1-4x}} \, .$$
Now we can prove the following statements:
\begin{eqnarray}
\lim_{k\to \infty}\,\frac{F_k}{F_{k+1}} & = & C(x)\,,                             \label{lim1}\\
\lim_{k\to \infty}\,\frac{{F_k}^2}{{F_{k+1}}^2} & = & \frac{C(x)-1}{x}\,,  \label{lim2}\\
\lim_{k\to \infty}\,\frac{F_k}{F_{k+2}} & = & \frac{C(x)-1}{x}\,.                 \label{lim3}
\end{eqnarray}

Using the previous identities we can write in an alternative way the argument of Limit \ref{lim1}
$$
\lim_{k\to \infty}\,\frac{\frac{1-x^{k+1}C^{2(k+1)}(x)}{C^{k+1}(x)\sqrt{1-4x}}}{\frac{1-x^{k+2}C^{2(k+2)}(x)}{C^{k+2}(x)\sqrt{1-4x}}}=\lim_{k\to \infty}\,C(x)\cdot\frac{1-x^{k+1}C^{2(k+1)}(x)}{1-x^{k+2}C^{2(k+2)}(x)}\,\,,
$$
and so Limit \ref{lim1} holds. In a similar way we can prove also Limit \ref{lim2} and \ref{lim3}.

From Theorem \ref{formulaP}, and using the above results, we obtain the desired proof.
\end{proof}

\section[A bijective proof for the number of $\poly_k$]{A bijective proof for the number of \\$k$-parallelogram polyominoes}
In \cite{knuth} it is proved that $\,\, x\cdot \frac{F_k}{F_{k+1}}\,\,$ is the generating function of planted plane trees having height less than or equal to $k+1$. Hence, the generating function obtained for $P_k(x)$ in $(\ref{P})$ can be expressed as the difference between the generating functions of pairs of planted plane trees having height at most $k+2$, and pairs of planted plane trees having height exactly equal to $k+2$. 

Our aim is now proceed trying to provide a combinatorial explanation to this fact, by establishing a bijective correspondence between $k$-parallelogram polyominoes and planted plane trees having height less than or equal to a fixed value; first we will show how to build the planted plane tree associated  with a given parallelogram polyomino $P$ and then we will show what is the link between the convexity degree of $P$ and the corresponding tree.
We recall that a {\em planted plane tree} is a rooted tree which has been embedded in the plane so that the relative order of subtrees at each branch is part of its structure. Henceforth we shall say simply {\em tree} instead of planted plane tree. Let $T$ be a tree, the {\em height} of $T$, denoted by $|T|$, is the number of nodes on a maximal simple path starting at the root.
Figure \ref{trees_h5_sp6} depicts the seven trees having exactly $6$ nodes and height equal to $5$.

\begin{figure}[htbp]
  \begin{center}
    \includegraphics[width=13cm]{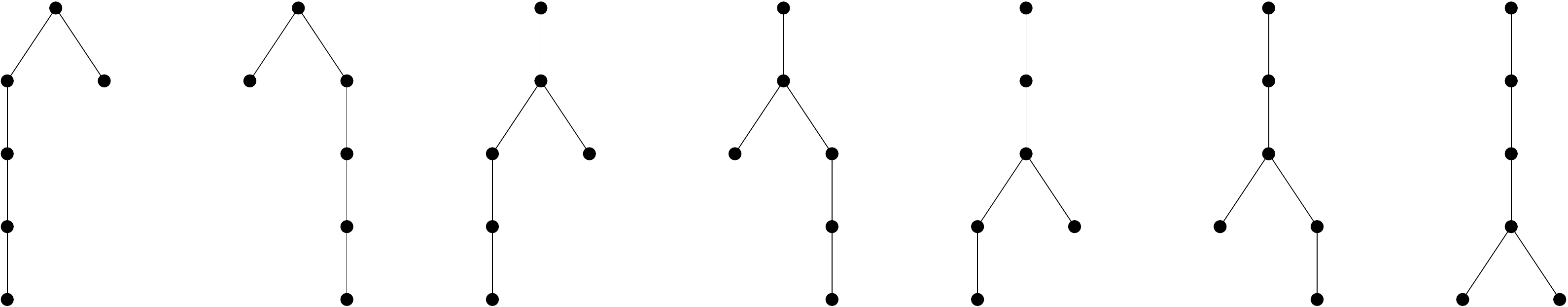}
    \caption{The seven trees with $6$ nodes and height exactly $5$.}
    \label{trees_h5_sp6}
  \end{center}
\end{figure}

\subsection{From parallelogram polyominoes to planted plane trees}
To construct the bijection that we will see in the follows we had take inspiration from \cite{ABBS}. 

Given a parallelogram polyomino $P$ we begin by labeling:
\begin{itemize}
 \item [-] each step $e$ of the upper boundary of $P$ with the integer numbers from $1$ to the width of $P$ and moving from right to left;
 \item [-] each step $n$ of the lower boundary of $P$ with marked integer numbers from $\overline{1}$ to the height of $P$ and moving from top to bottom.
\end{itemize}

The reader can see an example of this labeling process in Figure \ref{biezione} $(a)$. We want to notice that the labeling of a polyomino is uniquely determined by construction and that every label $l$ (resp. $\overline{l}$) identifies a column (resp. a row) into $P$.
\begin{definition}\label{n(l)}
 Let $P$ be a parallelogram polyomino. We denote by $e(\overline{l})$ the array of labels (except for the label $1$), which are an edge of a cell belonging to the row determined by $\overline{1}$.
 For every label $l\geq 1$ (resp. $\overline{l}\geq 2$) we take into consideration the column (resp. the row) determined by it. We denote by $n(l)$ (resp. $e(\overline{l})$) the array of labels, which correspond to an edge of a cell belonging to this column (resp. row).
\end{definition}
It is clear that each label, of the just defined array, corresponds to a step $n$ (resp. $e$) on the lower (resp. upper) boundary of $P$.
 
We can better understand Definition \ref{n(l)} seeing an example of it in Figure \ref{biezione}. For instance, here we have:
$$n(7)=(\overline{5},\overline{6})\,\,\,\,\mbox{and}\,\,\,\,e(\overline{5})=(9,10,11)\,.$$

At this point we are able to construct the corresponding tree, called $T(P)$, in the following way:
\begin{itemize}
 \item [-] we associate to any label of $P$ one node in $T(P)$, in particular the root will be the node labeled with $1$;
 \item [-] the children of the node $1$ are exactly the ones labeled with the labels in $n(1)$, ordered from left to right; more in general the children of a node with label $l$ (resp. $\overline l$) are exactly the ones labeled with the labels $n(l)$ (resp. $e(\overline l)$), ordered from left to right.
\end{itemize}
Figure \ref{biezione} $(b)$ shows an example of our correspondence.
\begin{proposition}
Let be $\mathfrak{P}_n$ and $\mathfrak{T}_n$ respectively the set of parallelogram\\ polyominoes with semi-perimeter $n$ and the set of trees with n nodes.\\
The following correspondence
 $$T:\mathfrak{P}_n \rightarrow \mathfrak{T}_n$$ is a bijection.
\end{proposition}

\begin{proof}
The injectivity follows directly from our construction. $\T$ is also surjective. 
 It is easy to see that the number of nodes of $T(P)$ is equal to the semi-perimeter of $P$. In fact the number of nodes is equal to the number of labels that is equal to the sum of steps $e$ of the upper boundary of $P$ and of steps $n$ of the lower boundary of $P$. Since we are talking about parallelogram polyominoes which are first of all convex polyominoes, such a sum corresponds exactly to the semi-perimeter of $P$.
 So, given a tree $T$ we will build the corresponding parallelogram polyomino, denoted by $P(T)$. Starting from a fixed  point of the plain we will go to construct the upper path and the lower path of $P(T)$ in two different phases:
 \begin{enumerate}
  \item we start with a step $o$ which corresponds to the root. For every node labeled with $\overline{l}$, with $\overline{l}\geq 1$, we draw as many steps $o$ as the number of its children and one step $s$.
  \item  For every node labeled with $l$, with $l\geq 1$, we draw as many steps $s$ as the number of its children and one step $o$.
 \end{enumerate}
Clearly this construction guarantees that the upper path and the lower path of $P(T)$ has the same length and that they are disjoint except at their common ending points, otherwise $T$ does not be a tree.   
\end{proof}

\begin{figure}[htbp] 
  \begin{center}
    \includegraphics[width=13cm]{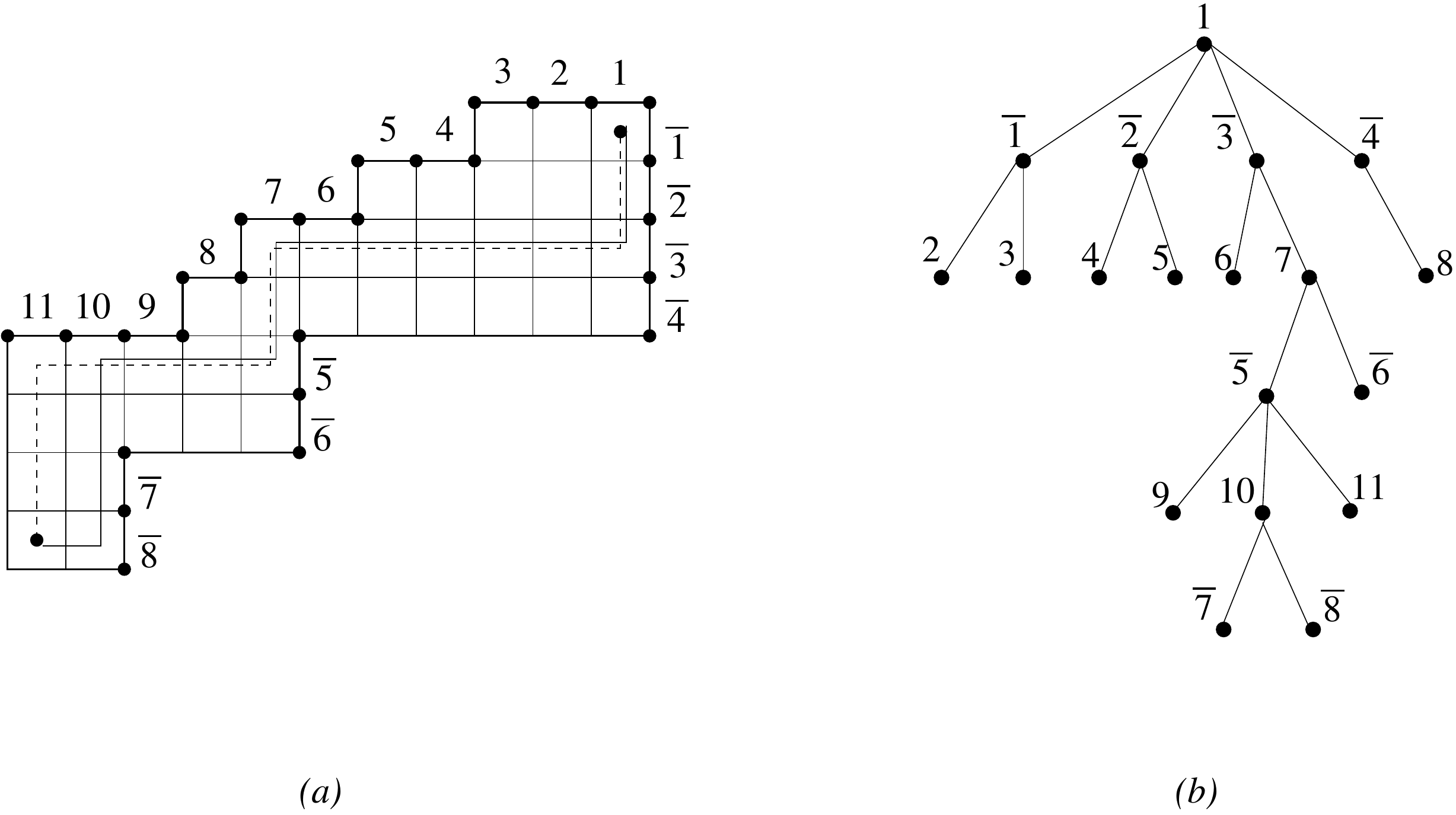}
    \caption{$(a)$ A polyomino in $\mathbb{P}_4^U$ and in $(b)$ its corresponding tree.}
    \label{biezione}
  \end{center}
\end{figure}

Moreover, we can notice that:
\begin{itemize}
 \item [-] in every tree obtained by our correspondence, the root labeled with $1$ has at least the node labeled with $\overline{1}$ as a child;
 \item [-] the nodes labeled with $l$ and $\overline{l}$ are at alternate levels;
 \item [-] the labeling is uniquely determined as in the case of parallelogram polyominoes. So, from now on, when we deal of trees, we mean labeled trees as we have just seen.
\end{itemize}

\subsection[$k$-convexity degree and the height of a tree]{The link between the $k$-convexity degree and the height of a tree}
We can point out that, by our construction, given a polyomino $P$ and its associated tree $T(P)$, the greater node labeled with $l$ (resp. $\overline{l}$) corresponds into $P$ at the cell in which the path $v$ (resp. $h$) has/makes the first change of direction. 
\begin{definition}
 Let $T$ be a tree having height equal to $i$. According to the parity of $i$ we can define two sequences of nodes.
 \begin{itemize}
  \item {\bfseries case $i$ odd:}\\
  we call $v_T$ (resp. $h_T$) the sequence of nodes of the simple path starting from the rightmost node at the height $i$ (resp. $i-1$) and ending when reaching either the node $1$ or $\overline{1}$.  
  \item {\bfseries case $i$ even:}\\
  we call $h_T$ (resp. $v_T$) the sequence of nodes of the simple path starting from the rightmost node at the height $i$ (resp. $i-1$) and ending when reaching either the node $1$ or $\overline{1}$.
 \end{itemize}
\end{definition}

For example, let $T$ be the tree in Figure \ref{biezione} $(b)$. $|T|$ is equal to $6$, so we are in the even case, and following the previous definition we are able to write $h_T=(\overline{8},10,\overline{5},7,\overline{3},1)$ and $v_T=(11,\overline{5},7,\overline{3},1)$. 

These two just defined sequences have an important property. More in details, the nodes of the sequence $h_T$ (resp. $v_T$) correspond to the cells of $P(T)$ in which $h$ (resp. $v$) makes a change of direction, hence there are exactly the $m$ (resp. $m'$) steps determined by the path $h$ (resp. $v$) when encountering the boundary of $P(T)$, that we called, in our decomposition \ref{dec}, $X_i$ or $Y_i$ depending on it is a horizontal or vertical one (see Fig. \ref{decomposition}).
So, the convexity degree of $P(T)$ is equal to the minimal number of nodes among the two paths $h_{T(P)}$ and $v_{T(P)}$ minus one.

In general, we can observe that the height of $T$ is strictly related to the number of nodes of $h_T$ and $v_T$ and by definition $h_T$ and $v_T$ can be referred to the node $1$ or $\overline{1}$. So, the height of $T$ is equal to the maximal number of nodes among the two paths plus one, and the following proposition holds.
\begin{proposition}\label{|T|}
 Let $P$ be a polyomino in $\poly_k$. The height of $T(P)$ is less than or equal to $k+3$.
\end{proposition}

The proof follows directly from our construction.

As we said before, Equation $(\ref{P})$ suggests us to take into consideration a pair of trees, so we identify every tree $T(P)$ with a pair of trees $T_1$ and $T_2$, denoted $(T_1,T_2)$, which are respectively the ones obtained taking the subtree having the node labeled with $\overline{1}$ as a root, and the remaining subtree having the node labeled with $1$ as a root. More formally
\begin{definition}\label{T1T2}
 Let be $T_1$ and $T_2$ a pair of trees, we denote with $T=(T_1,T_2)$ the tree obtained putting $T_1$ as a left subtree of $T_2$.
\end{definition}
We remark that generally the pairs $T=(T_1,T_2)$ and $T'=(T_2,T_1)$ correspond to a different tree. Figure \ref{T1T2} depicts the decomposition of the tree of Figure \ref{biezione} $(b)$.

\begin{figure}[htbp] 
  \begin{center}
    \includegraphics[width=10cm]{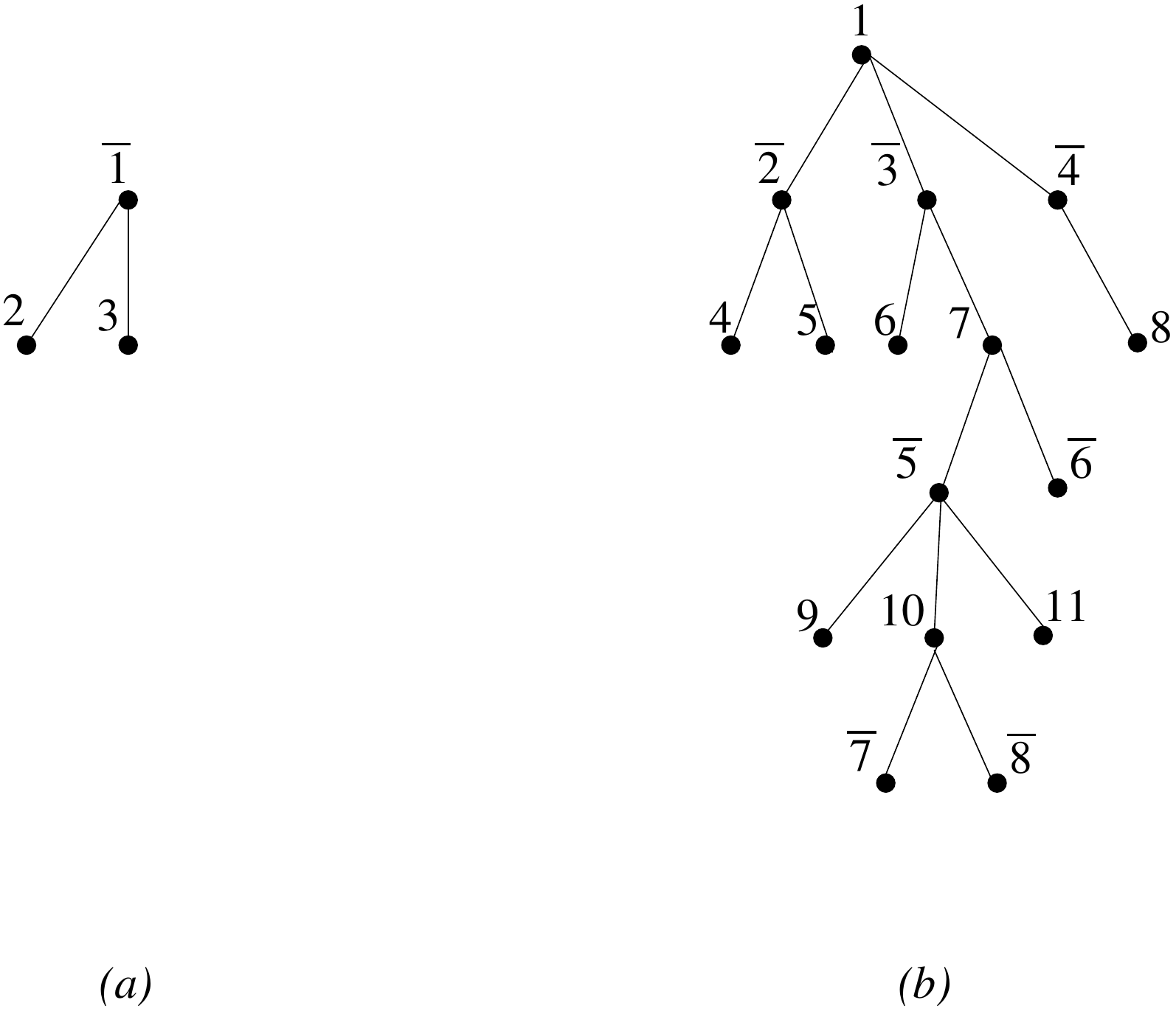}
    \caption{The pair $(T_1,T_2)$ of trees corresponding to tree depicted in Figure \ref{biezione}, in particular $T_1$ in $(a)$ and $T_2$ in $(b)$.}
    \label{T1T2}
  \end{center}
\end{figure}

Now we are ready to provide a bijective proof of the combinatorial explanation of Equation $(\ref{P})$.
\begin{proposition}
 The number of $k$-parallelogram polyominoes is equal to the number of pairs of trees having height less than or equal to $k+2$ minus the number of pairs of trees having height exactly equal to $k+2$.
\end{proposition}

\begin{proof}

Let $P$ be a polyomino of $\poly_k$, with $k\geq 0$. We start with the assumption that $k$ is even (the case when $k$ is odd can be treated in a similar way). Basing on Proposition \ref{|T|}, we have that $|T(P)|\leq k+3$. Being $T=(T_1,T_2)$, for Definition \ref{T1T2}, we can deduce that
$$|T_1|\leq k+2\,\,\,\,\,\mbox{and}\,\,\,\,\,|T_2|\leq k+3\,\,.$$
Since the set $\poly_k$ contains all the polyominoes having convexity degree less than or equal to $k$, we can restrict our analysis assuming that $P$ has convexity degree exactly equal to $k$ and so
$$k+1\leq |T_1|\leq k+2\,\,\,\,\,\mbox{and}\,\,\,\,\,k+1\leq |T_2|\leq k+3\,\,.$$
By this fact and previous considerations we have to consider only the following four cases:
\begin{enumerate}
 \item {\em $|T_1|\leq k+2$ and $|T_2|\leq k+3$.}\\
 We have to take into consideration the borderline case, $|T_1|= k+2$ and $|T_2|= k+3$ and so $|T(P)|=|(T_1,T_2)|=k+3$ (even). Both the nodes with the greatest labels $l$ and $\overline{l}$ belong to $T_2$ then the number of nodes of $h_T$ is equal to $k+3$ and the number of nodes of $v_T$ is equal to $k+2$. As we said before, the convexity degree of $P$ is equal to the minimal number of nodes among $h_T$ and $v_T$ minus one. In this case this is
 $$min(k+3,k+2)-1=k+1$$
 and so $P$ belongs to $\poly_{k+1}$ and not to $\poly_k$ against the hypothesis.
 

\begin{figure}[htbp] 
  \begin{center}
    \includegraphics[width=12cm]{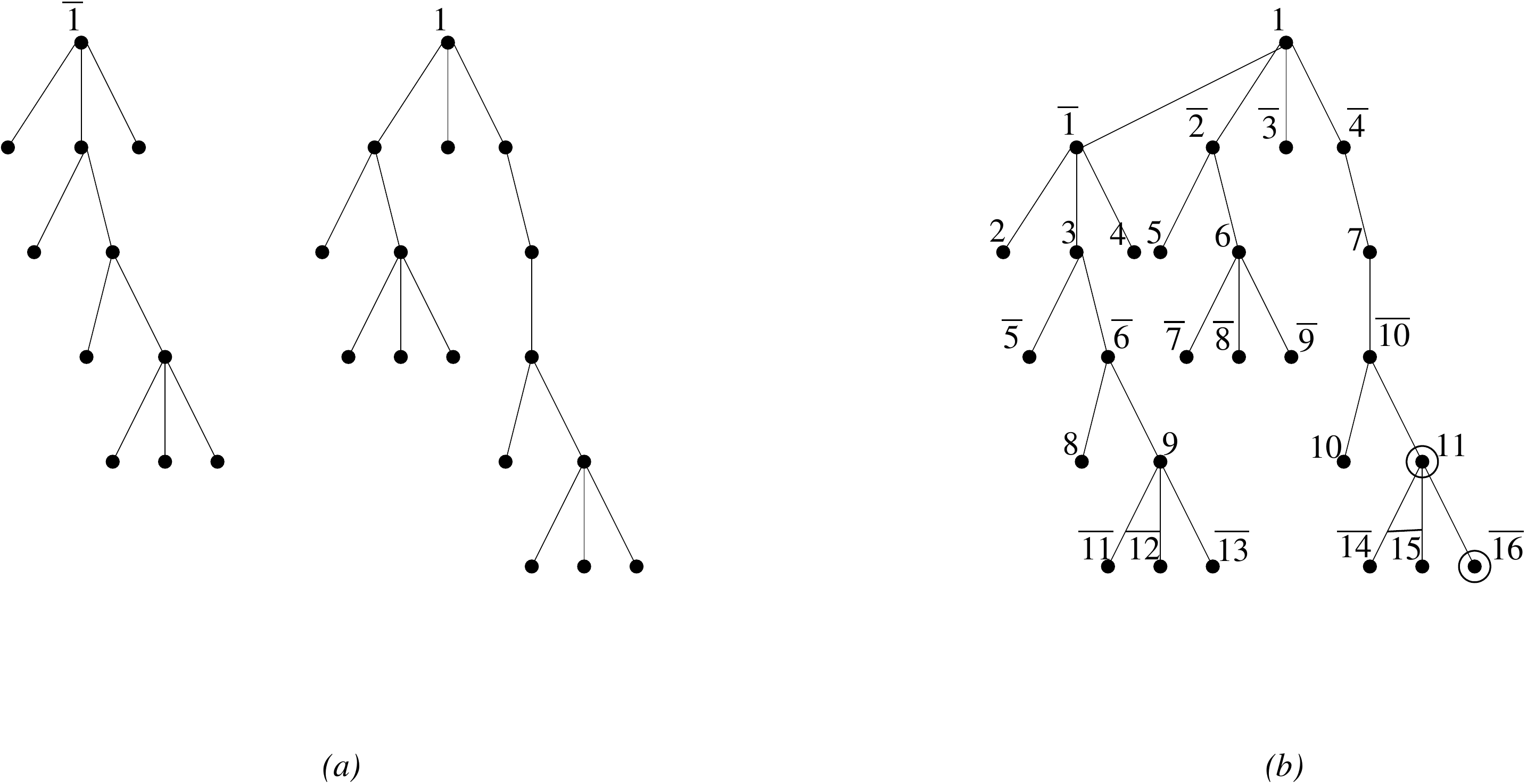}
    \caption{Case $1$. $(a)$ A pair of trees $T_1$ and $T_2$ and its corresponding tree $T=(T_1,T_2)$ in $(b)$.}
    \label{bijProof1A}
  \end{center}
\end{figure}

\begin{figure}[htbp] 
  \begin{center}
    \includegraphics[width=12cm]{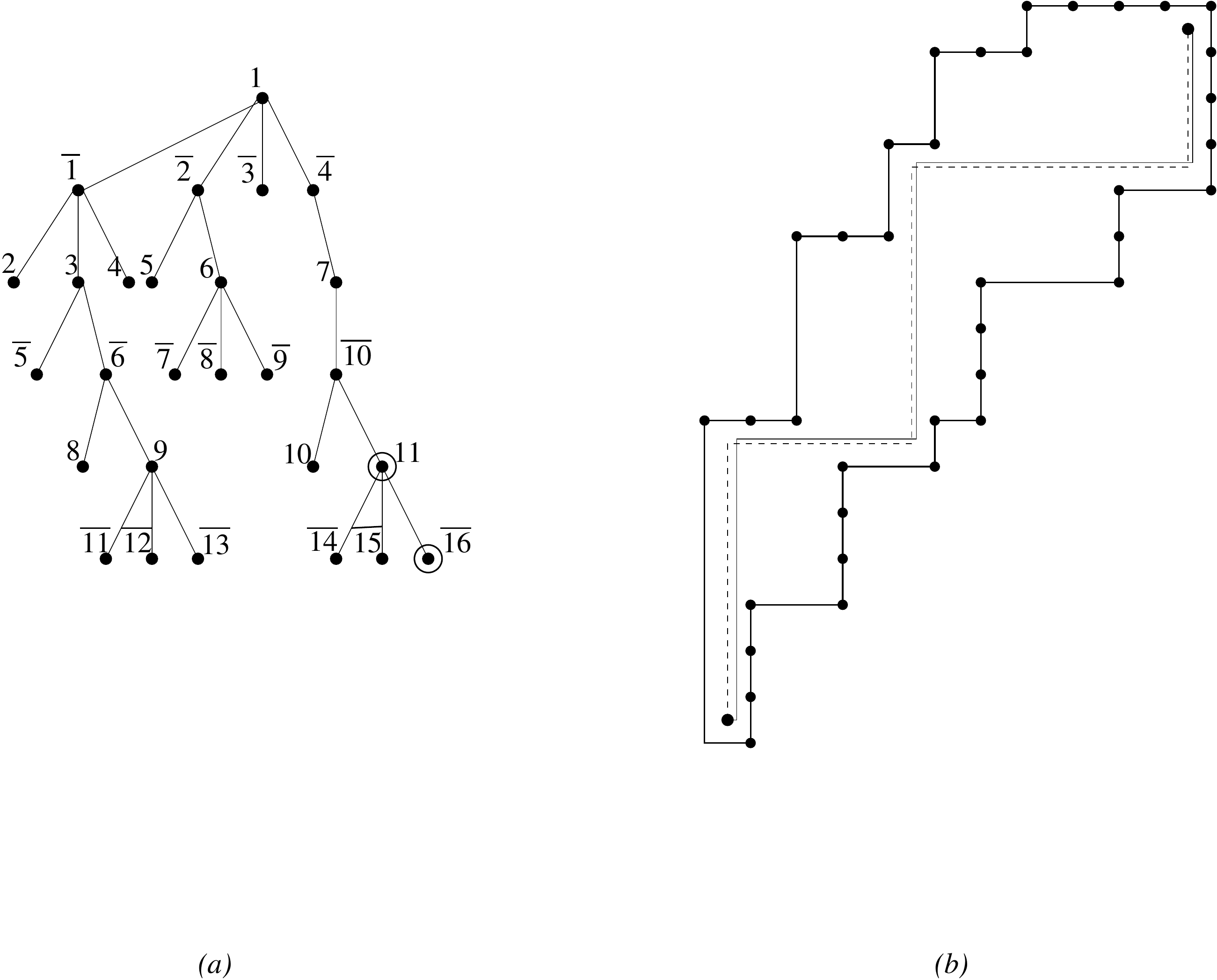}
    \caption{Case $1$. $(a)$ The tree $T=(T_1,T_2)$ and in $(b)$ its corresponding parallelogram polyomino $P(T)$.}
    \label{bijProof1B}
  \end{center}
\end{figure}

An example is shown in Figure \ref{bijProof1A}. Here $k=3$, $|T_1|=5$ and $T_2=6$ and so $|T|=|(T_1,T_2)|=6$.
$$h_T=(\overline{16},11,\overline{10},7,\overline{4},1)\,\,\,\,\mbox{and}\,\,\,\,v_T=(11,\overline{10},7,\overline{4},1)\,.$$
So the convexity degree of $P(T)$ is equal to $min(6,5)-1=4$ as we can see in Figure \ref{bijProof1B}.

 \item {\em $|T_1|\leq k+2$ and $|T_2|\leq k+2$.}\\
 Also here we take into consideration the borderline case, $|T_1|= k+2$ and $|T_2|= k+2$ and so $|T(P)|=|(T_1,T_2)|=k+3$. 
 Since the height of $T$ is even we have that the node with the greatest label $\overline{l}$ belongs to $T_1$, then the number of nodes of $h_T$ is equal to $k+3$, and the node with the greatest label $l$ belongs to $T_2$ and the number of nodes of $v_T$ is equal to $k+2$. As we said before, the convexity degree of $P$ is equal to the minimal number of nodes among $h_T$ and $v_T$ minus one. In this case this is
 $$min(k+3,k+2)-1=k+1$$
 and so $P$ belongs to $\poly_{k+1}$ and not to $\poly_k$ against the hypothesis.
 The reader can see an example of case $2.$ in Figure \ref{bijProof2A} and in Figure \ref{bijProof2B}.

\begin{figure}[htbp] 
  \begin{center}
    \includegraphics[width=12cm]{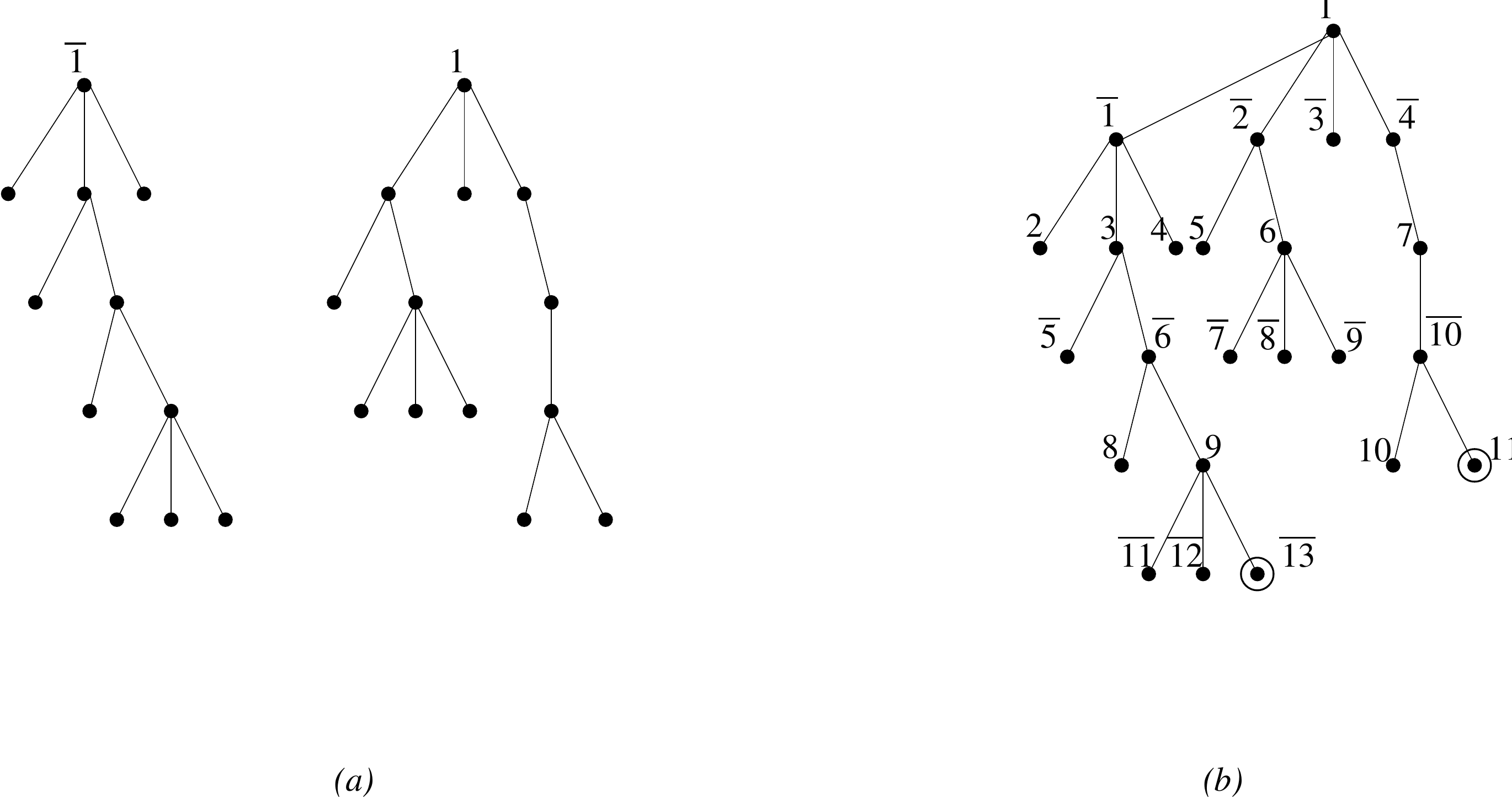}
    \caption{Case $2$. $(a)$ A pair of trees $T_1$ and $T_2$ and its corresponding tree $T=(T_1,T_2)$ in $(b)$.}
    \label{bijProof2A}
  \end{center}
\end{figure}

\begin{figure}[htbp] 
  \begin{center}
    \includegraphics[width=12cm]{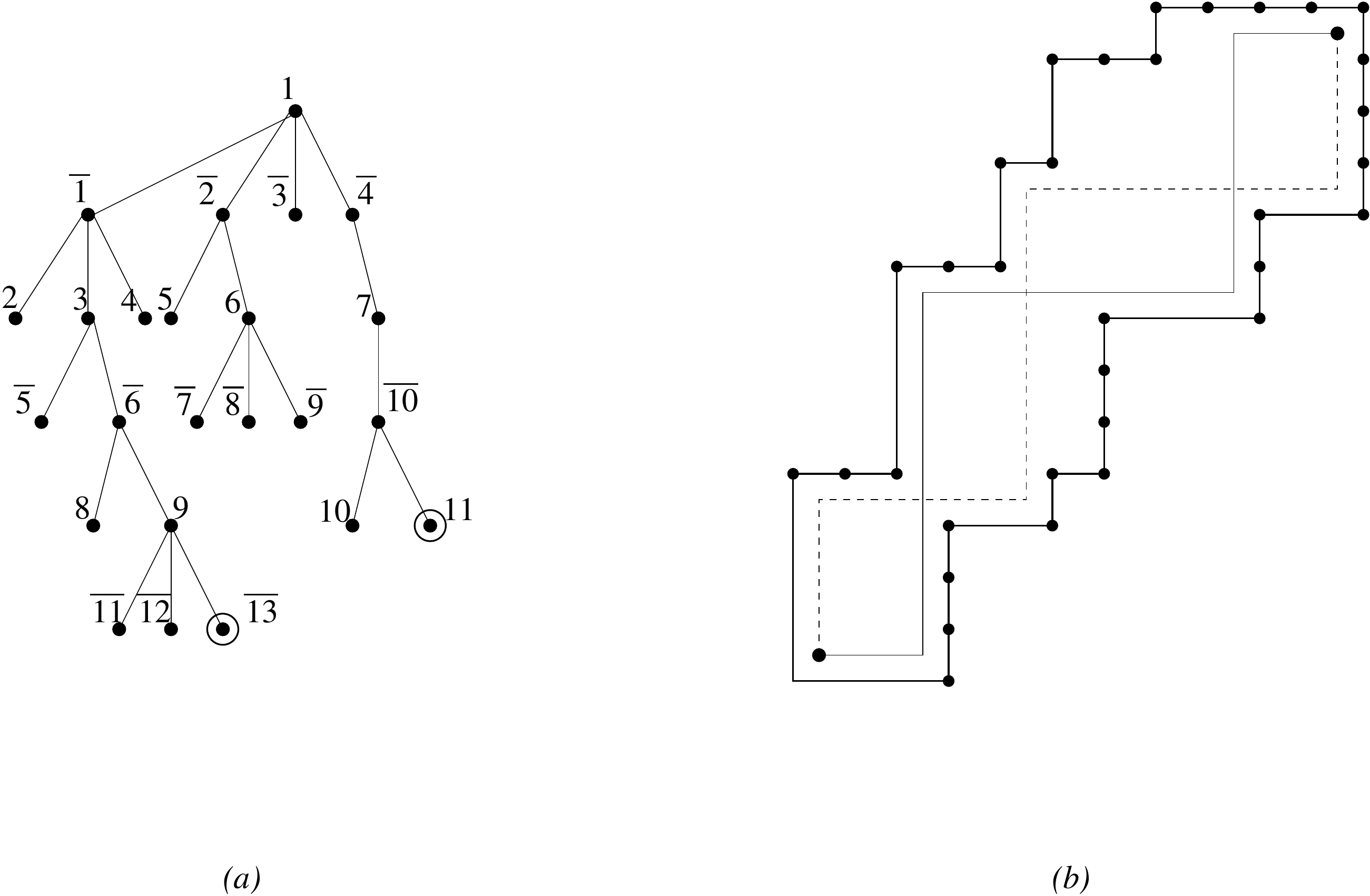}
    \caption{Case $2$. $(a)$ The tree $T=(T_1,T_2)$ and in $(b)$ its corresponding parallelogram polyomino $P(T)$.}
    \label{bijProof2B}
  \end{center}
\end{figure}
 
 \item {\em $|T_1|\leq k+1$ and $|T_2|\leq k+2$.}\\
As in the previous cases we analyze the borderline situation, $|T_1|= k+1$ and $|T_2|= k+2$ and so $|T|=|(T_1,T_2)|=k+2$.

Both the nodes with the greatest labels $l$ and $\overline{l}$ belong to $T_1$, it is possible to check such a situation using  for instance Figure\ref{bijProof3A}. According to the parity of $|T(P)|$ (odd), we have that the number of nodes of $v_T$ is equal to $k+2$ and the number of nodes of $h_T$ is equal to $k+1$. The convexity degree of $P$ is equal to the minimal number of nodes among $h_T$ and $v_T$ minus one. In this case is
 $$min(k+1,k+2)-1=k\,\,,$$
so $P$ belongs to $\poly_{k}$, see Figure \ref{bijProof3B}.
 

\begin{figure}[htbp] 
  \begin{center}
    \includegraphics[width=12cm]{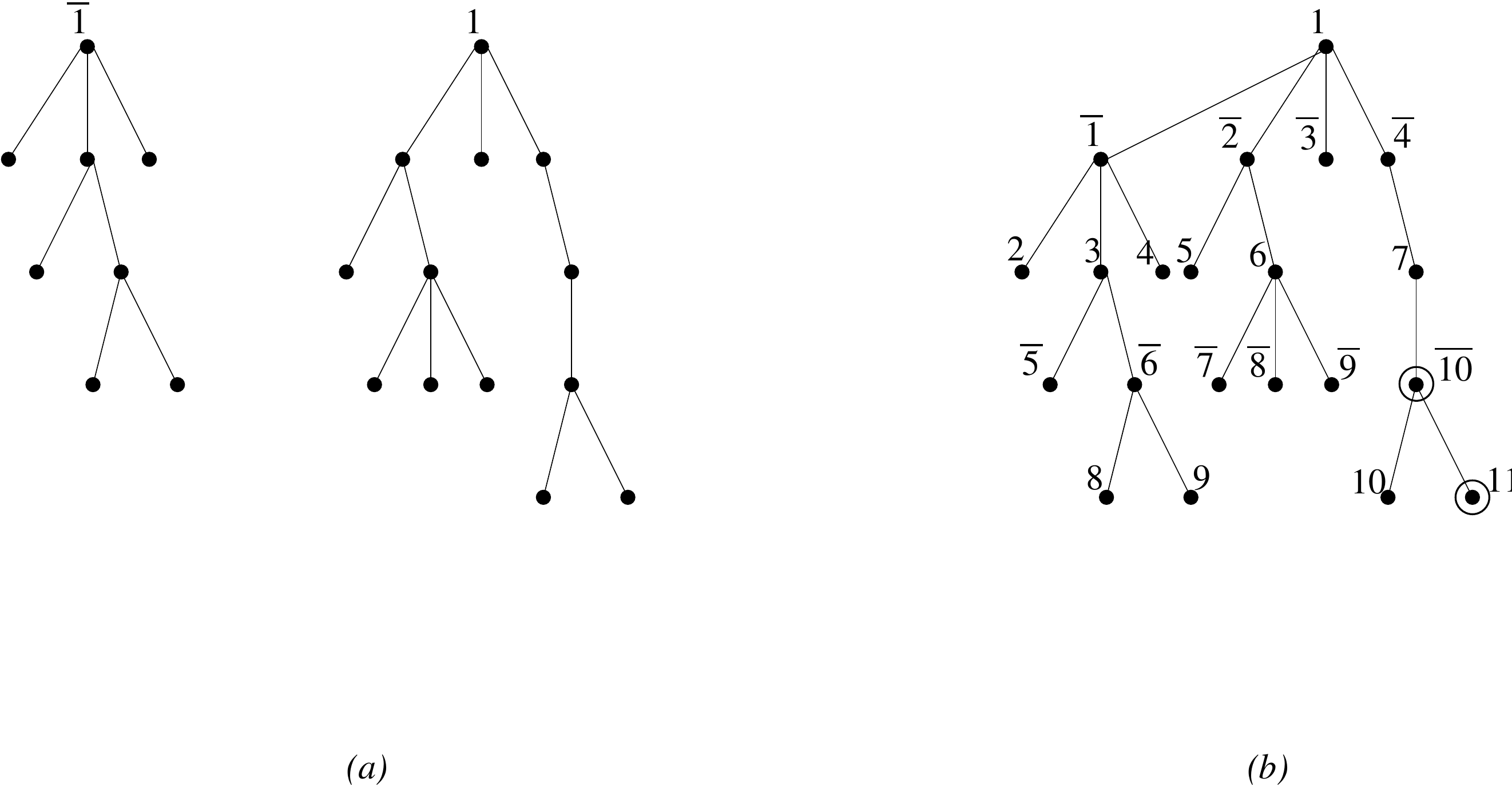}
    \caption{Case $3$. $(a)$ A pair of trees $T_1$ and $T_2$ and its corresponding tree $T=(T_1,T_2)$ in $(b)$.}
    \label{bijProof3A}
  \end{center}
\end{figure}

\begin{figure}[htbp] 
  \begin{center}
    \includegraphics[width=12cm]{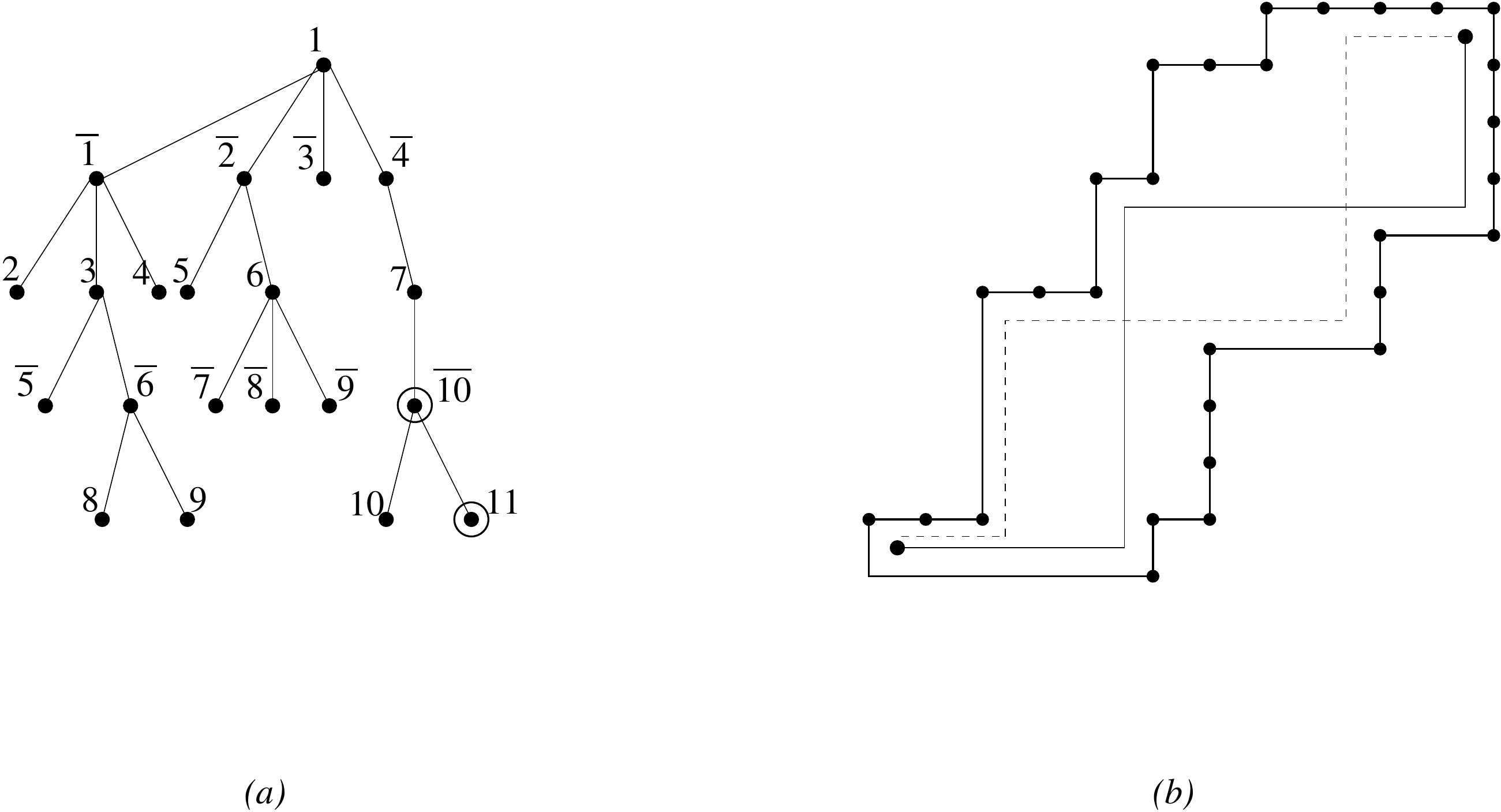}
    \caption{Case $3$. $(a)$ The tree $T=(T_1,T_2)$ and in $(b)$ its corresponding parallelogram polyomino $P(T)$.}
    \label{bijProof3B}
  \end{center}
\end{figure}
 \item {\em $|T_1|\leq k+2$ and $|T_2|\leq k+1$.}\\
 Also here we take into consideration the borderline case, $|T_1|= k+2$ and $|T_2|= k+1$ and so $|T(P)|=|(T_1,T_2)|=k+3$. 
 We have that both the nodes with the greatest label $\overline{l}$  and the greatest label $l$ belong to $T_1$, and since the height of $T$ is even the number of nodes of $h_T$ is equal to $k+2$, and the number of nodes of $v_T$ is equal to $k+1$. As before, the convexity degree of $P$ is equal to the minimal number of nodes among $h_T$ and $v_T$ minus one. In this case it is
 $$min(k+2,k+1)-1=k$$
 so, as in the previous case, $P$ belongs to $\poly_{k}$, see Figure \ref{bijProof4B}.


\begin{figure}[htbp] 
  \begin{center}
    \includegraphics[width=12cm]{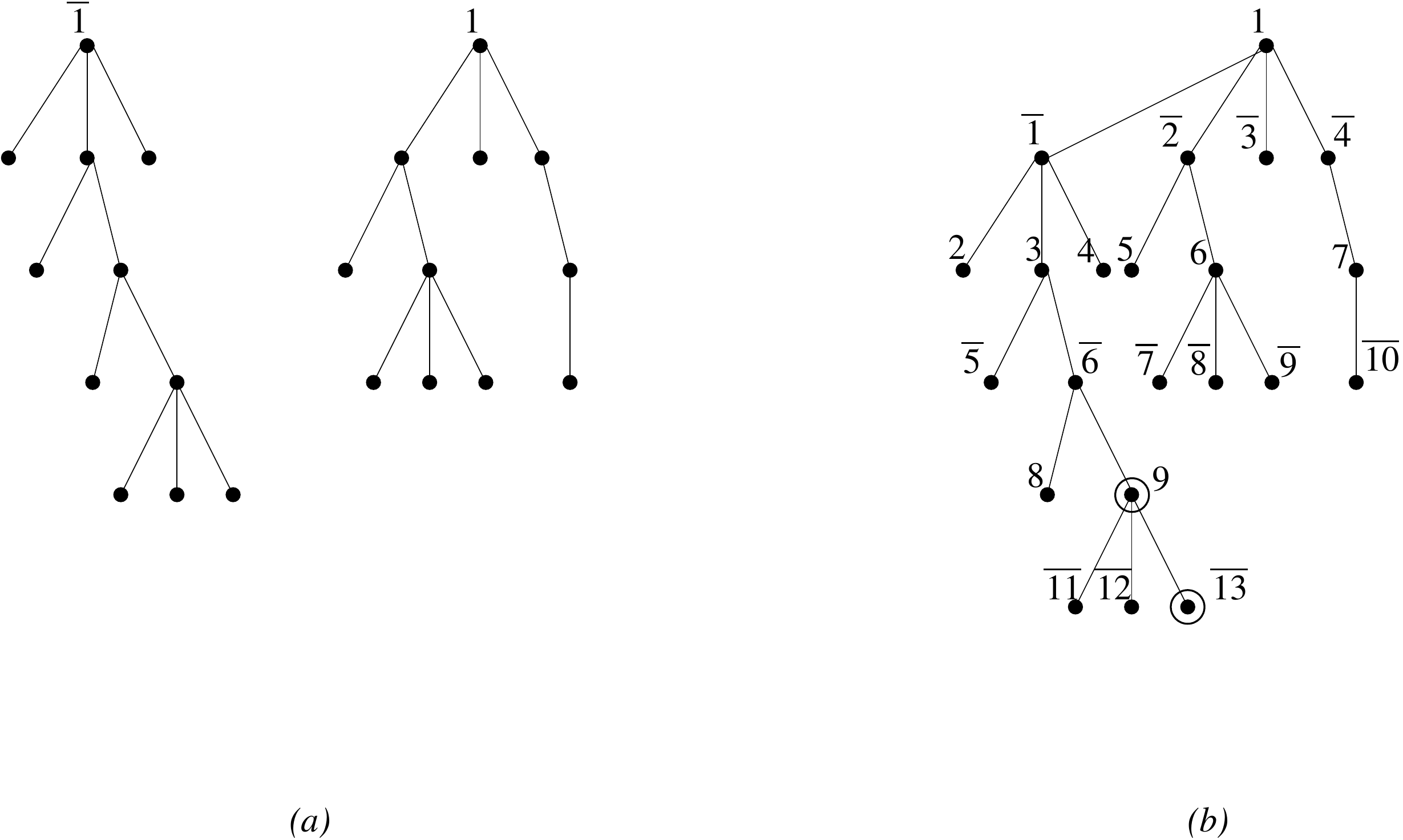}
    \caption{Case $4$. $(a)$ A pair of trees $T_1$ and $T_2$ and its corresponding tree $T=(T_1,T_2)$ in $(b)$.}
    \label{bijProof4A}
  \end{center}
\end{figure}

\begin{figure}[htbp] 
  \begin{center}
    \includegraphics[width=12cm]{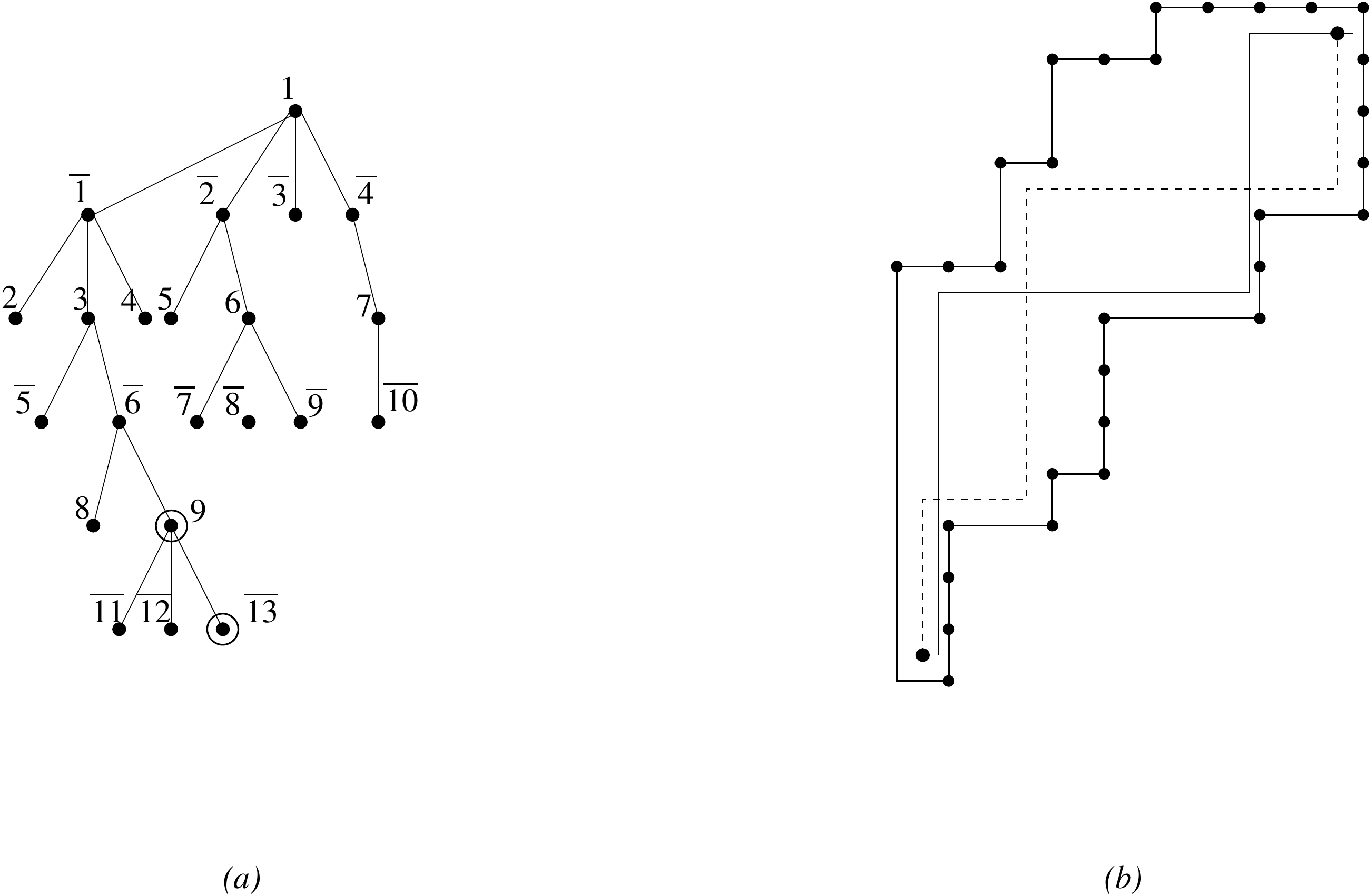}
    \caption{Case $4$. $(a)$ The tree $T=(T_1,T_2)$ and in $(b)$ its corresponding parallelogram polyomino $P(T)$.}
    \label{bijProof4B}
  \end{center}
\end{figure}
\end{enumerate}

We have therefore shown that in cases $3.$ and $4.$ $P$ is a $k$-parallelogram polyomino,
while in case $1.$ and $2.$ $P$ is a parallelogram polyomino which is exactly $(k+1)$-parallelogram polyomino, that is absurd for hypothesis. So, we have that the number of polyominoes in $P_{k}(x)$ will be given by considering
only the cases $3.$ and $4.$, then we have the thesis.

\end{proof}

Let us now see how we can deduce some properties of $P(T)$ given a tree $T=(T_1,T_2)$. Let us consider $v_T=(v_1,\cdots,v_j)$ and $h_T=(h_1,\cdots,h_j')$, where $j$ and $j'$ can be equal or differ at most by one. We want to remark that $v_1$ and $h_1$ are equal to the greatest label of type $l$ and $\overline{l}$ respectively. Basing on the value of $j$ and $j'$ we can say that:
\begin{itemize}
 \item [-] $j=j'$ then $P(T)$ is a {\em flat} $k$-parallelogram polyomino;
 \item [-] $j\neq j'$ and $j$ odd (resp. $j'$ is even) then $P(T)$ is an up $k$-parallelogram polyomino;
 \item [-] $j\neq j'$ and $j$ even (resp. $j'$ is odd) then $P(T)$ is a right $k$-parallelogram polyomino.
\end{itemize}

 Supposing that we are in the second or in the third case, in particular that $j=j'+1$ (resp. $j'=j+1$). There exists an index $c$, $2\leq c\leq j$ (resp. $2\leq c\leq j'$), such that starting respectively from the $c$-th and the $(c-1)$-th element of $v_T$ and $h_T$ (resp. of $h_T$ and $v_T$), the sequences coincide.
 In particular, the node with the label $v_T(c-1)$ (resp. $h_T(c-1)$) corresponds in $P(T)$ to the cell $C$, defined previously in Section \ref{cellC}.

 For example if we consider that $T$ is the tree in Figure \ref{biezione} $(b)$, we have that
 $$v_T=(11,\overline{5},7,\overline{3},1)\,\,\,\,\mbox{and}\,\,\,\,h_T=(\overline{8},10,\overline{5},7,\overline{3},1)$$
 then $j=5$ and $j'=6$. Starting from the second and the third element of $v_T$ and $h_T$ respectively, the sequences coincide. As a consequence of that, the node with the label $h_T(2)=10$ corresponds in $P(T)$ to the cell $C$, as we can see in Figure \ref{biezione} $(a)$.
Therefore, we are in the case when $j$ is even then we can say that $P(T)$ is an up $k$-parallelogram polyomino. 

\section{Further work}

We have extended some of the results in some recent researches that we have obtained for the class of $k$-parallelogram polyominoes to another remarkable subclass of convex polyominoes, the $k$-convex polyominoes which are also {\em directed polyominoes}, called for brevity {\em $k$-directed polyominoes} and denoted by $\mathbb{D}_k$.

More in details, we were able to apply our decomposition, explained in Section \ref{dec}, to the set of $k$-directed polyominoes. This is principally due to the fact that we have found an analogous of Proposition \ref{convexitydegree} that holds also for this new considered class. In fact, also in this case, to find out the convexity degree of a directed convex
polyomino $P$ it is sufficient to check the changes of direction required to any
path running from the {\em source} $S$ to the ``furthest cells'' of $P$.
Then, giving a $k$-directed polyomino $P$, we can provide a definition of two paths $h(P)$ and $v(P)$, which is analogous to that of Definition \ref{def:hEv}. These two paths - as we can check in Figure \ref{fig2} - identify some vertical/horizontal steps on the boundary of $P$. These steps are called, analogously to the case of $k$-parallelogram polyominoes (see \ref{dec}), $X_i$ or $Y_i$ depending on it is a horizontal or vertical one. Furthermore, we can apply the same decomposition technique, which is graphically shown in Figure \ref{fig2}. 

\begin{figure}[htbp]
\begin{center}
\includegraphics[width=12cm]{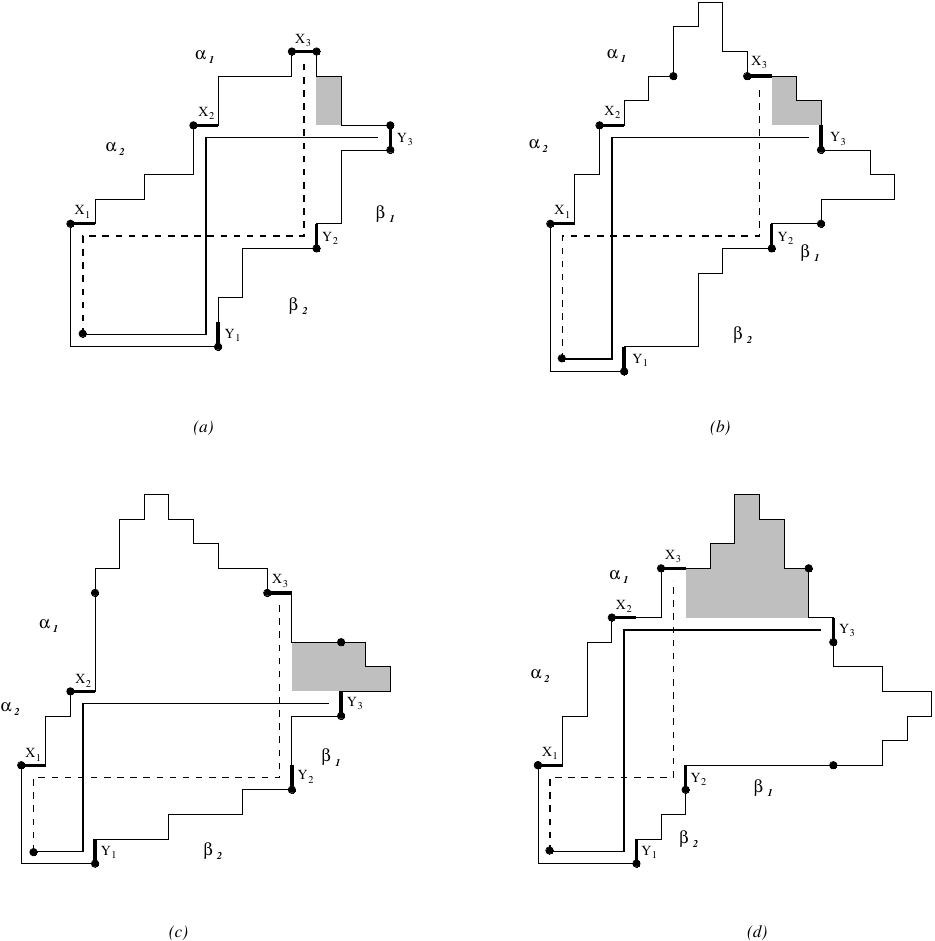}
\caption{$(a)$ A polyomino $P\in U_3$.
$(b)$ A polyomino $P\in V_3$. $(c)$ A
polyomino $P\in W_3$ in which the uppermost cells of $P$ are on the left of $X_3$, and the rightmost cells of $P$ are above $Y_3$. $(d)$ A
polyomino $P\in W_3$ in which the uppermost cells of $P$ are on the right of $X_3$, and the rightmost cells of $P$ are below $Y_3$.}
\label{fig2}
\end{center}
\end{figure}

\begin{remark}
Let $P$ be a $k$-directed polyomino. We can prove that the cells of $P$ which require the maximal number of changes of direction to be reached are the ones on the right of the step $X_k$ and over the step $Y_k$. These cells are the ones shaded in Figure \ref{fig2}. 
\end{remark}

The next step is to give a classification of the polyominoes of $\mathbb{D}_k$ based on
the position of the steps $X_k$ and $Y_k$. A polyomino $P$ belongs to the class:
\begin{description}
\item{-} $U_k$ if at least one of the uppermost cells of $P$ is on the right of
$X_k$, and at lest one of the rightmost cells of $P$ is above $Y_k$, see Figure \ref{fig2} $(a)$ ;
\item{-} $V_k$ if the uppermost cells of $P$ are on the left of $X_k$ (except the
one containing $X_k$ itself), and the rightmost cells of $P$ are below $Y_k$ (except
the one containing $Y_k$ itself), see Figure \ref{fig2} $(b)$;   
\item{-} $W_k$ otherwise, see Figure \ref{fig2} $(c)$ and $(d)$.
\end{description}

The three classes have to be enumerated separately, then the generating function of
$\mathbb{D}_k$ can be obtained by summing the three generating functions.

Unfortunately, unlike the case of $k$-parallelogram polyominoes, each of these classes have to be split in several subclasses, in order to take care of all possible configurations which can occur. It follows that the obtained formulas have a rather complex expression, and in particular, we have not been able to express them in terms of the Fibonacci polynomials, as it was for parallelogram polyominoes.

Another interesting problem, that we have begin to investigate, is to determine the asymptotic behavior of the class of $k$-parallelogram polyominoes. Starting from our expression of the generating function, given in Theorem \ref{formulaP}, and basing on some results from \cite{flajolet}, we believe it is possible to obtain a general solution for all $k$. 
\chapter{Permutation and polyomino classes}\label{chap:cap4}

\section{Introduction}

The concept of a {\em pattern} within a {\em combinatorial structure} is undoubtedly one of the most investigated notions in combinatorics. It has been deeply studied for permutations, starting first with~\cite{Kn}. Analogous definitions
were provided in the context of many other structures, such as set partitions~\cite{Go,Kl,Sa},
words~\cite{Bj,Bu}, trees~\cite{DPTW,R}, and paths~\cite{ferrari}, see Section~\ref{sec:permutations}. 

In the following we will recall some important definitions that will be useful to better understand  the work described in this chapter. 


\section{Permutation classes and polyomino classes}
\label{sec:def_classes}

\subsection{Permutation patterns and permutation classes}
%

%

The relation of containment $\sympattern$ is a partial order relation on the set $\sym$ of all permutations. 
Moreover, properties of the poset $(\sym,\sympattern)$ have been described in the literature~\cite{poset} 
and we recall some of the most well-known here: 
$(\sym,\sympattern)$ is a well founded poset (\emph{i.e.} it does not contains infinite descending chains), 
but it is not well ordered, since it contains infinite antichains (\emph{i.e.}infinite sets of pairwise incomparable elements); 
moreover, it is a graded poset (the rank function being the size of the permutations). 

\begin{definition}
A \emph{permutation class} (sometimes called \emph{pattern class} or \emph{class} for short) is a set of permutations \C that is downward closed for $\sympattern$: 
for all $\sigma \in \C$, if $\pi \sympattern \sigma$, then $\pi \in \C$.
\label{def:permutation_class}
\end{definition}

For any set $\B$ of permutations, denoting $\Avperm(\B)$ the set of all permutations that avoid every pattern in $\B$, 
we clearly have that $\Avperm(\B)$ is a permutation class.
The converse statement is also true. Namely:

\begin{proposition}
For every permutation class \C, there is a unique antichain $\B$ such that $\C = \Avperm(\B)$. 
The set $\B$ consists of all minimal permutations (in the sense of $\sympattern$) that do not belong to \C. 
\label{prop:perm_basis}
\end{proposition}
The reader can find more details about this proposition in Section~\ref{sec:patternAv}.

In the usual terminology, $\B$ is called the \emph{basis} of \C. 
Here, we shall rather call $\B$ the \emph{permutation-basis} (or \emph{$p$-basis} for short), to distinguish from other kinds of bases that we introduce later.

Notice that because $(\sym,\sympattern)$ contains infinite antichains, the basis of a permutation class may be infinite, see for instance, the permutations introduced in~\cite{pinPerm}, called {\em pin permutations}. 

Actually, Proposition~\ref{prop:perm_basis} does not hold only for permutation classes, but for all well-founded posets, 
which will be important for our purpose. First of all we want to recall the notion of well-founded poset:

\begin{definition}
A poset $(\mathfrak{X},\preccurlyeq)$ is called well-founded, if $\mathfrak{X}$ has no infinite descending chain $\{a_0,a_1,\cdots,a_n,\cdots\}$ with $a_0>a_1>\cdots>a_n>\cdots\,\,.$
 \label{def:well_founded_poset}
\end{definition}

\begin{proposition}
For any well-founded poset $(\mathfrak{X},\preccurlyeq)$, 
for any subset $\C$ of $\mathfrak{X}$ that is downward-closed for $\preccurlyeq$, 
there exists a unique antichain $\B$ of $\mathfrak{X}$ such that 
$\C = \Av_{\mathfrak{X}}(\B) = \{ x \in \mathfrak{X} : $ for all $b \in \B, b \preccurlyeq x$ does not hold$\}$. 
The set $\B$ consists of all minimal elements of $\mathfrak{X}$ (in the sense of $\preccurlyeq$) that do not belong to \C. 
\label{prop:basis_of_downward_closed_subposet}
\end{proposition}

\begin{proof}
Let $\C$ be a subset of $\mathfrak{X}$ that is downward closed for $\preccurlyeq$. 
The complement $\mathfrak{X} \setminus \C$ of $\C$ with respect to $\mathfrak{X}$ is upward closed for $\preccurlyeq$. 
Let us define $\B$ to be the set of minimal elements of $\mathfrak{X} \setminus \C$: 
$\B = \{b \in \mathfrak{X} \setminus \C  \mid \forall x \in \mathfrak{X} \setminus \C, \text{ if } x \preccurlyeq b \text{ then }x=b\}$. 
This is equivalent to characterizing $\B$ as the set of minimal elements of $\mathfrak{X}$ (in the sense of $\preccurlyeq$) that do not belong to \C. 
Because $\mathfrak{X}$ is well-founded, we have that $x \in \mathfrak{X} \setminus \C$ if and only if $\exists b \in \B$ such that $b \preccurlyeq x$. 
By contraposition, we immediately get that $\C = \Av_{\mathfrak{X}}(\B)$. 
In addition, by minimality, the elements of $\B$ are pairwise incomparable, so that $\B$ is indeed an antichain. 

To further ensure uniqueness, it is enough to notice that for two different antichains $\B$ and $\B'$ 
the sets $\C = \Av_{\mathfrak{X}}(\B)$ and $\C' = \Av_{\mathfrak{X}}(\B')$ are also different.
\end{proof}

Permutation classes have been extensively studied from the seventies until now, see Section~\ref{sec:permutations}. Nowadays, the research on permutation classes is being developed into several directions. One of them is to define notions of patterns analogous to Definition~\ref{def:permutation_pattern} (see Section~\ref{sec:patternAv}) in other combinatorial objects, 
and to find out which of the nice properties of permutation classes, or of the order $\sympattern$, or of the associated poset $(\sym,\sympattern)$, \ldots extend to a more general setting. 
The work presented here goes into this direction, and is specifically interested in matrix patterns in polyominoes and in permutations. 

\subsection{Permutation matrices and the submatrix order}
\label{subsec:submatrix}

Permutations are in (obvious) bijection with permutation matrices, 
\ie binary matrices with exactly one entry $1$ in each row and in each column. 
To any permutation $\sigma$ of $\sym_n$, we may associate a permutation matrix $M_{\sigma}$ of dimension $n$ 
by setting $M_{\sigma}(i,j) = 1$ if $i = \sigma(j)$, and $0$ otherwise. 
Throughout this work we adopt the convention that rows of matrices are numbered from bottom to top, 
so that the $1$ in $M_{\sigma}$ are at the same positions as the dots in the diagram of $\sigma$ -- see an example on Figure~\ref{fig:perm-matrix_diagram}.

\begin{figure}[htd]
\begin{center}
\includegraphics[width=9cm]{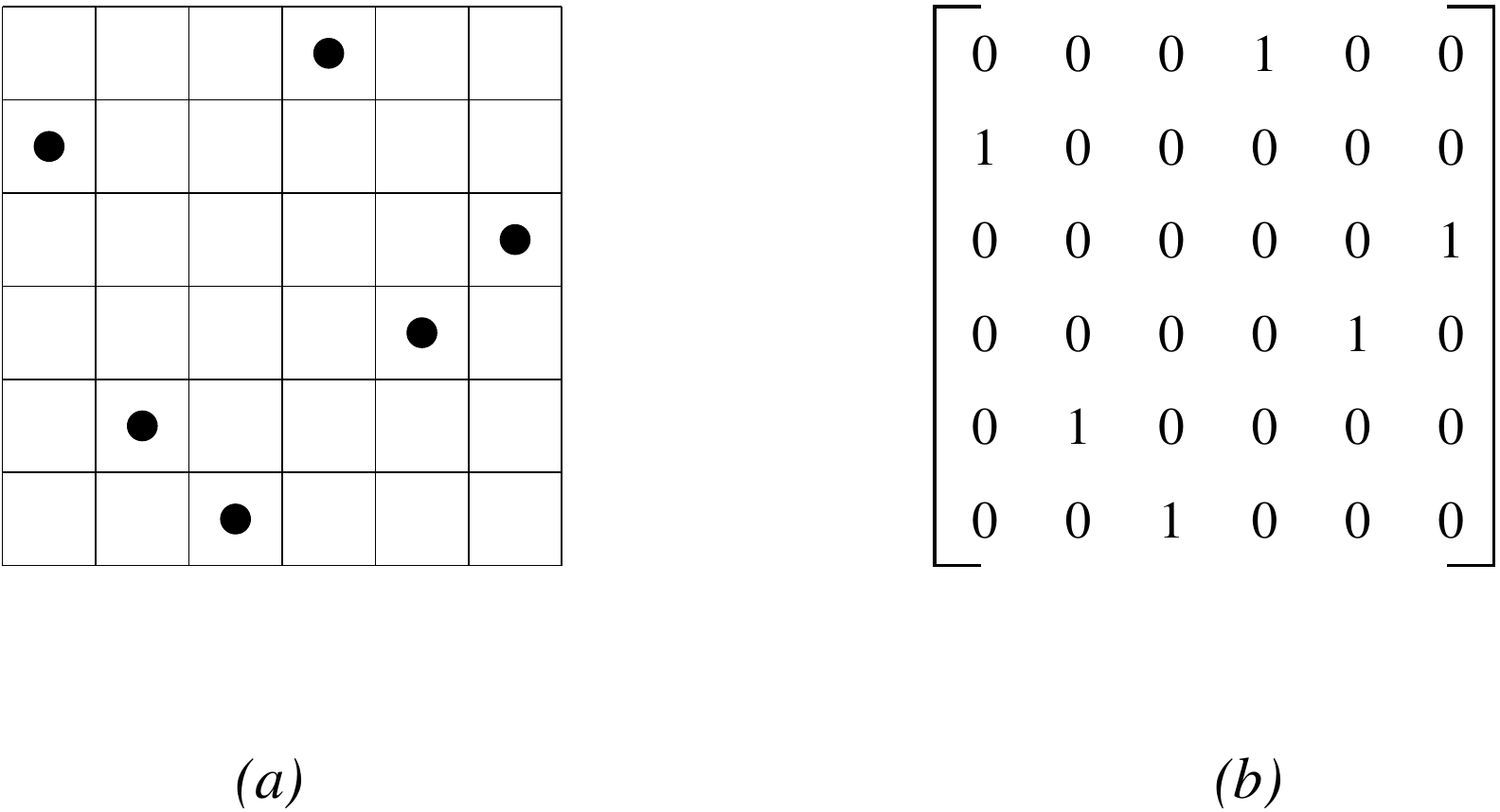}
\caption{$(a)$ Graphical representation (or diagram) of the permutation $\sigma=521634$. $(b)$ The permutation matrix corresponding to $\sigma$.}
\label{fig:perm-matrix_diagram}
\end{center}
\end{figure}

Let \matr be the class of binary matrices (\emph{i.e.} with entries in $\{0,1\}$). 
We denote by $\preccurlyeq$ the usual submatrix order on \matr, 
\emph{i.e.} $M' \preccurlyeq M$ if $M'$ may be obtained from $M$ by deleting any collection of rows and/or columns. 

Of course, whenever $\pi \sympattern \sigma$, we have that $M_{\pi}$ is a submatrix of $M_{\sigma}$. 
Notice however that not all submatrices of $M_{\sigma}$ are permutation matrices, 
and we will discuss in Subsection~\ref{subsec:matrix_patterns} some consequences of this simple remark in the study of permutation classes. 

Another simple fact that follows from identifying permutations with the corresponding permutation matrices 
is that we may rephrase the definition of permutation classes as follows: 
A set \C of permutations is a class if and only if, for every $\sigma \in \C$, every submatrix of $\sigma$ which is a permutation is in \C.
This does not say much by itself, 
but it allows to define analogues of permutation classes for other combinatorial objects that are naturally represented by matrices, like polyominoes. 

\subsection{Polyominoes and polyomino classes}\label{sec:papc}


A polyomino $P$ may be represented by a binary matrix $M$ whose dimensions are those of the minimal bounding rectangle of $P$: 
drawing $P$ in the positive quarter plane, in the unique way that $P$ has contacts with both axes, 
an entry $(i,j)$ of $M$ is equal to $1$ if the unit square $[j-1,j]\times[i-1,i]$ of $\mathbb{Z} \times\mathbb{Z}$
is a cell of $P$, $0$ otherwise (see Figure~\ref{polMatr}). 
Notice that, according to this definition, in a matrix representing a polyomino the first (resp. the last) row (resp. column) should contain at least a $1$.

\begin{figure}[htd]
\begin{center}
\includegraphics[width=9cm]{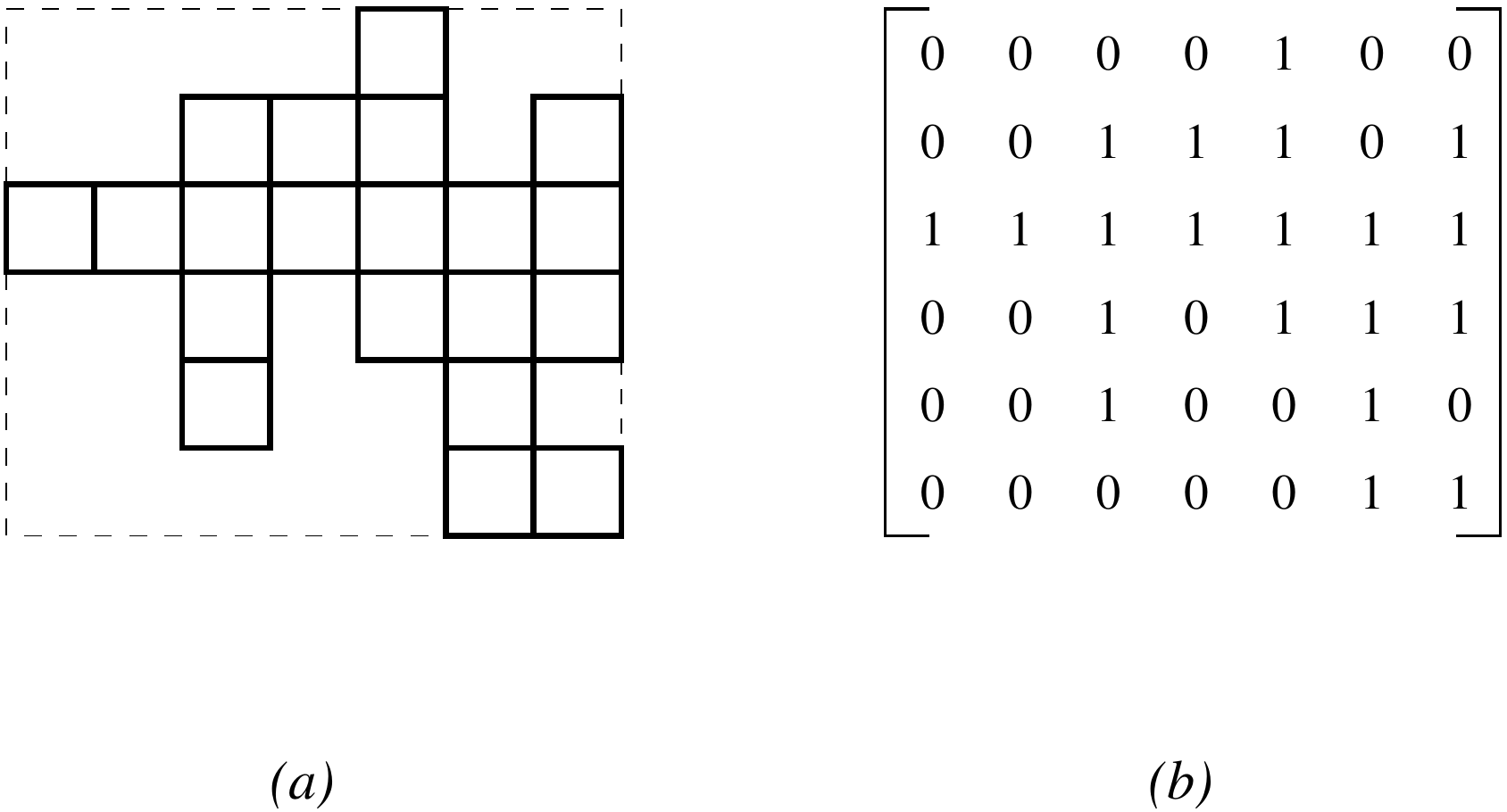}
\caption{A polyomino and its representation as a binary matrix.}
\label{polMatr}
\end{center}
\end{figure}

Let us denote by \poly the set of polyominoes, viewed as binary matrices as explained above. 
We can consider the restriction of the submatrix order $\preccurlyeq$ on \poly. 
This defines the poset $(\poly,\polypattern)$ and the pattern order between polyominoes: 
a polyomino $P$ is a \emph{pattern} of a polyomino $Q$ (which we denote $P \polypattern Q$) 
when the binary matrix representing $P$ is a submatrix of that representing $Q$. 

We point out that the order $\polypattern$ has already been studied in~\cite{CR} under the name of {\em subpicture order}. 
The main point of focus of~\cite{CR} is the family of {\em $L$-convex polyominoes} defined by the same authors in~\cite{lconv2}. 
But~\cite{CR} also proves that $\polypattern$ is not a partial well-order, since $(\poly,\polypattern)$ contains infinite antichains. 
Remark also that $(\poly,\polypattern)$ is a graded poset (the rank function being the semi-perimeter of the bounding box of the polyominoes). 

This implies in particular that $(\poly,\polypattern)$ is well-founded. 

Notice that these properties are shared with the poset $(\sym,\sympattern)$ of permutations. 
This allows to introduce a natural analogue of permutation classes for polyominoes:

\begin{definition}
A \emph{polyomino class} is a set of polyominoes \C that is downward closed for $\polypattern$: 
for all polyominoes $P$ and $Q$, if $P \in \C$ and $Q \polypattern P$, then $Q \in \C$.
\label{def:polyomino_class}
\end{definition}

The reader can exercise in finding simple examples of polyomino classes, such as, for instance: 
the family of polyominoes having at most three columns, the family of polyominoes having a rectangular shape, or the whole family of polyominoes. 
Some of the most famous families of polyominoes are indeed polyomino classes, 
like the {\em convex polyominoes} and the $L$-convex polyominoes. 
This will be investigated in more details in Section~\ref{sec:known_classes_poly}.
However, there are also well-known families of polyominoes which are not polyomino classes, like:
the family of polyominoes having a square shape, the family of polyominoes having exactly three columns, 
or the family of polyominoes {\em with no holes} (i.e. polyominoes whose boundary is a simple path). 
Figure~\ref{fig:buco} shows that a polyomino in this class may contain a polyomino with a hole.

\begin{figure}[ht]
\begin{center}
\includegraphics[width=7cm]{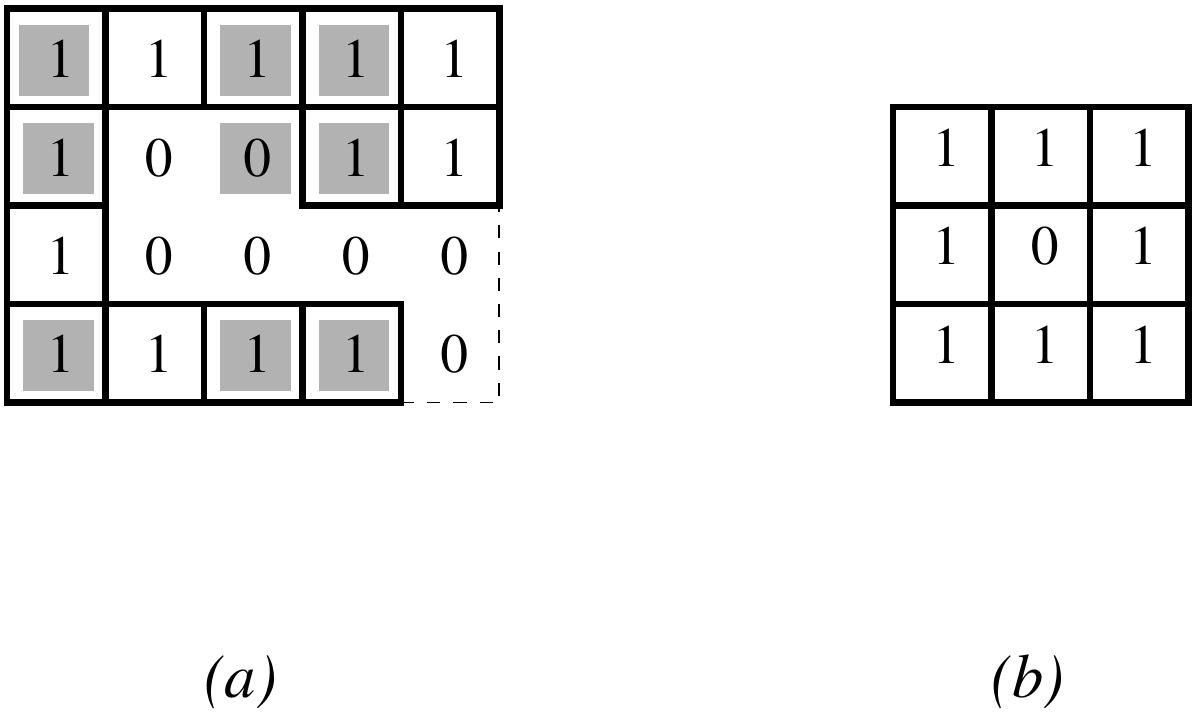}
\caption{$(a)$ A polyomino $P$ with no holes; $(b)$ A polyomino $P' \polypattern P$ containing a hole.}
\label{fig:buco}
\end{center}
\end{figure}

\smallskip

Similarly to the case of permutations, for any set $\B$ of polyominoes, 
let us denote by $\Avp(\B)$ the set of all polyominoes that do not contain any element of $\B$ as a pattern. 
Every such set $\Avp(\B)$ of polyominoes defined by pattern avoidance is a polyomino class. 
Conversely, like for permutation classes, every polyomino class may be characterized in this way. 

\begin{proposition}
For every polyomino class \C, there is a unique antichain \B of polyomminoes such that  $\C=\Avp(\B)$. 
The set \B consists of all minimal polyominoes (in the sense of \polypattern) that do not belong to \C.
\label{prop:polyomino_basis}
\end{proposition}

\begin{proof}
Follows immediately from Proposition~\ref{prop:basis_of_downward_closed_subposet} and the fact that $(\poly,\polypattern)$ is a well-founded poset. 
\end{proof}

As in the case of permutations we call \B the \emph{polyomino-basis} (or \emph{$p$-basis} for short), to distinguish from other kinds of bases.

Recall that $(\poly,\polypattern)$ contains infinite antichains~\cite{CR}, so there exist polyomino classes with infinite $p$-basis. 
We will show an example of a polyomino class with an infinite $p$-basis in Proposition~\ref{prop:infinite_basis}. 
However, we are not aware of \emph{natural} polyomino classes whose $p$-basis is infinite. 

\section{Characterizing classes with excluded submatrices}
\label{sec:excluded_submatrices}

\subsection{Submatrices as excluded patterns} 
\label{subsec:matrix_patterns}

We have noticed in Subsection~\ref{subsec:submatrix} that not all submatrices of permutation matrices are themselves permutation matrices.  
More precisely: 

\begin{remark}
The submatrices of permutation matrices
are exactly those that contain at most one $1$ in each row and each column. 
We will call such matrices \emph{quasi-permutation matrices} in the rest of this section. 
\label{rem:submatrices_of_perms}
\end{remark}

For polyominoes, it also holds that not all submatrices of polyominoes are themselves polyominoes, 
but the situation is very different from that of permutations:

\begin{remark}
Every binary matrix is a submatrix of some polyomino. 
\label{rem:submatrices_of_pol}
\end{remark}

Indeed, for every binary matrix $M$, 
it is always possible to add rows and columns of $1$ to $M$ 
in such a way that all $1$ entries of the resulting matrix are connected. 

\smallskip

From Remarks~\ref{rem:submatrices_of_perms} and~\ref{rem:submatrices_of_pol}, 
it makes sense to examine sets of permutations (resp. polyominoes) 
that avoid submatrices that are not themselves permutations (resp. polyominoes).

\begin{definition}
For any set $\M$ of quasi-permutation matrices (resp. of binary matrices), 
let us denote by $\Avperm(\M)$ (resp. $\Avp(\M)$) the set of all permutations (resp. polyominoes) 
that do not contain any submatrix in $\M$. 
\label{def:submatrix_avoidance}
\end{definition}

In Definition~\ref{def:submatrix_avoidance}, for the case of permutations, 
we may as well consider sets $\M$ containing arbitrary binary matrices. 
But from Remark~\ref{rem:submatrices_of_perms}, excluding a matrix $M$ which is not a quasi-permutation matrix 
is not actually introducing any restriction: no permutation contains $M$ as a submatrix. 
Therefore, in our work, when considering $\Avperm(\M)$, we will always take $\M$ to be a set of quasi-permutation matrices. 
Figure~\ref{motivo} illustrates Definition~\ref{def:submatrix_avoidance} in the polyomino case. 

\begin{figure}[ht]
\begin{center}
\includegraphics[width=12cm]{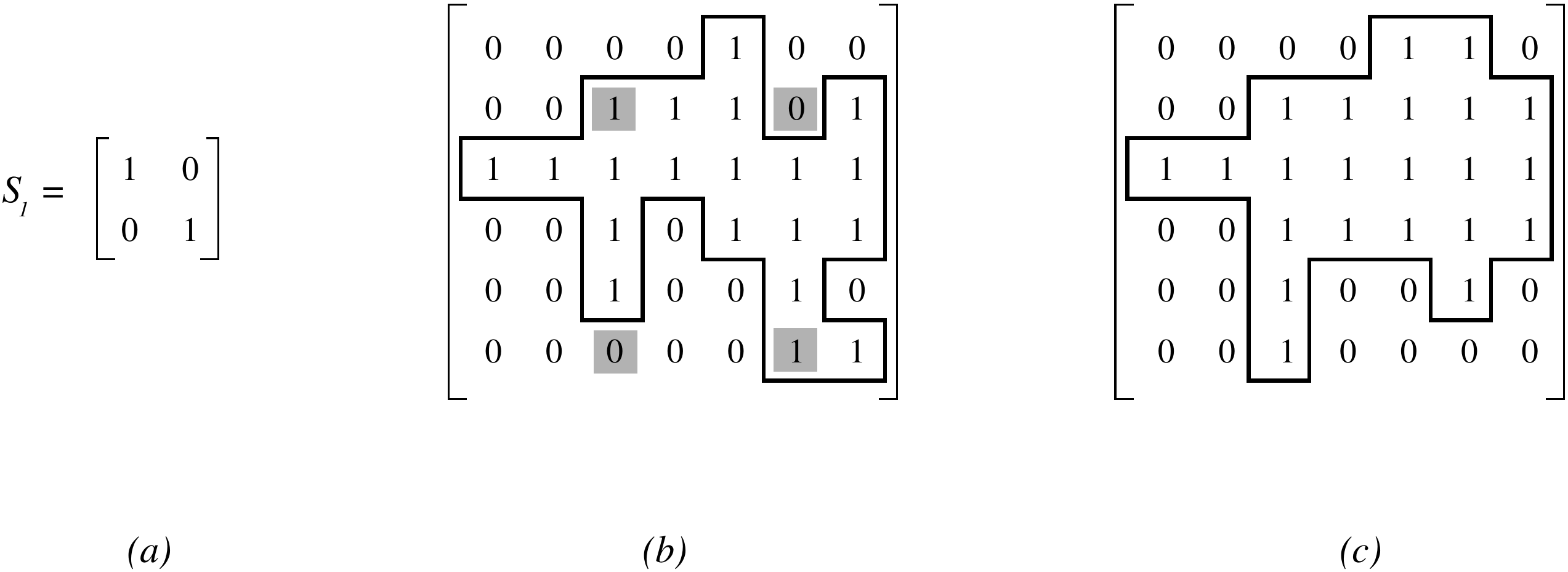}
\caption{$(a)$ a matrix $S_1$; $(b)$ a polyomino that contains $S_1$ as a submatrix, hence does not belong to $\Avp(S_1)$; 
$(c)$ a polyomino that does not contain $S_1$, \emph{i.e.} that belongs to $\Avp(S_1)$.}
\label{motivo}
\end{center}
\end{figure}

The followings facts, although immediate to prove, will be useful in our work:

\begin{remark}
When $\M$ contains only permutations (resp. polyominoes), 
these definitions of $\Avperm(\M)$ and $\Avp(\M)$ coincide with the ones given in Section~\ref{sec:def_classes}. 
\end{remark}

\begin{remark}
Denoting $\Avm(\M)$ the set of binary matrices that do not have any submatrix in $\M$, we have 
\[
\Avperm(\M)=\Avm(\M)\cap\sym \,\,\,\mbox{and}\,\,\, \Avp(\M)=\Avm(\M)\cap\poly \text{.}
\]
\label{rem:intersection_with_permutations_or_polyominoes}
\end{remark}

\begin{remark}
Sets of the form $\Avperm(\M)$  are downward closed for $\sympattern$, \ie are permutation classes. 
Similarly, the sets $\Avp(\M)$ are polyomino classes. 
\label{rem:submatrix_avoidance_implies_class}
\end{remark}

We believe it is quite natural to characterize some permutation or polyomino classes 
by avoidance of submatrices, and will provide several examples in Sections~\ref{sec:known_classes_perm} and~\ref{sec:known_classes_poly}. 
In the present section, we investigate further 
the description of permutation and polyomino classes by avoidance of matrices, 
and in particular how canonical and concise such a description can be.

\paragraph*{Remark on a different notion of containment/avoidance of binary matrices in permutation matrices.} ~ \\
To avoid any confusion, let us notice that 
another definition of containment of a binary matrix in a permutation matrix 
(different from the submatrix containment) 
has been around in the permutation patterns literature.  
It has in particular been used in the Marcus-Tardos proof 
of the Stanley-Wilf conjecture~\cite{marcTard}, 
and reads as follows:
A binary matrix $P = (p_{i,j})$ is contained in a permutation matrix $M$ 
if $M$ contains a submatrix $Q= (q_{i,j})$ of the same dimension as $P$ such that $q_{i,j} =1$ as soon as $p_{i,j} =1$. 

This notion of containment of binary matrices in permutations is different, but related to the classical submatrix containment. 
Indeed, $P$ being contained in $M$ in the Marcus-Tardos sense means that 
$M$ contains a submatrix that is either $P$ or some $P'$ obtained from $P$ by replacing some $0$ entries in $P$ by $1$. 
Actually, from Remark~\ref{rem:submatrices_of_perms}, 
this statement can be restricted \emph{w.l.o.g.} to matrices $P'$ obtained from $P$ 
by replacing some \emph{uncovered} $0$ entries in $P$ by $1$. 
By \emph{uncovered} $0$ entry, we mean a $0$ entry which does not have any entry $1$ in the same row nor in the same column.

Specifically, the set of permutations that avoid all the binary matrices in the set $\B$ in the Marcus-Tardos sense
is a permutation class, which may be described by the set of excluded submatrices 
$\B'$, where $\B' = \{P'$ obtained from $P \in \B$ by replacement of some (uncovered) $0$ entries by $1 \}$.
In this work, we view the Marcus-Tardos definition of avoidance of a matrix as a shortcut 
to mean avoidance in the submatrix sense of a set of matrices, 
and from now on, we focus on the (usual) notion of submatrix avoidance.

\subsection{Matrix bases of permutation and polyomino classes}

We have seen in Propositions~\ref{prop:perm_basis} and~\ref{prop:polyomino_basis} that 
for each permutation (resp. polyomino) class \C, 
the set of excluded permutation (resp. polyomino) patterns that characterizes \C is uniquely determined. 
In view of Proposition~\ref{prop:basis_of_downward_closed_subposet}, 
it is also not hard to associate with every permutation (resp. polyomino) class $\C$
a set $\M$ of matrices such that $\C = \Avperm(\M)$ (resp. $\C = \Avp(\M)$). 
Given a class $\C$, we can define such a set $\M$ in a canonical way (see Definition~\ref{def:canonical_m-basis}). 
However, we shall see in the following that for some permutation (resp. polyomino) classes $\C$, 
there exist \emph{several} antichains $\M'$ such that $\C = \Avperm(\M')$ (resp. $\C = \Avp(\M')$). 

\begin{definition}
Let \C be a class of permutations (resp. polyominoes). 
Denote by $\C^+$ the set of matrices that appear as a submatrix of some element of \C, \emph{i.e.} 
\[
\C^+ = \{ M \in \matr \mid \exists P \in \C, \text{ such that }M \preccurlyeq P\} \text{.}
\]
Denote by \M the set of all minimal matrices in the sense of $\preccurlyeq$ that do not belong to $\C^+$. 
\M is called the \emph{canonical matrix-basis} (or \emph{canonical $m$-basis} for short) of \C. 
\label{def:canonical_m-basis}
\end{definition}

Of course, the canonical $m$-basis of a class $\C$ is uniquely defined, and is always an antichain for $\preccurlyeq$. 
Moreover, Proposition~\ref{prop:canonical_m-basis} shows that it indeed provides a description of \C by avoidance of submatrices. 

\begin{proposition}
Let \C be a class of permutations (resp. polyominoes), 
and denote by $\M$ its canonical $m$-basis. 
We have $\C = \Avperm(\M)$ (resp. $\C = \Avp(\M)$).
\label{prop:canonical_m-basis}
\end{proposition}

\begin{proof}
Working in the poset $(\matr, \preccurlyeq)$, Proposition~\ref{prop:basis_of_downward_closed_subposet} 
ensures that $\C^+ = \Avm(\M)$. 
And since $\C = \C^+ \cap \sym$ (resp. $\C = \C^+ \cap \poly$), 
Remark~\ref{rem:intersection_with_permutations_or_polyominoes} yields the conclusion.
\end{proof}

\begin{example}
\label{ex:1-12-21_canonical_m-basis}
For the (trivial) class of permutations $\T=\{1,12,21\}$, we have
\begin{align*}
\T^{+} = & \bigg\{
\left[ \begin{array}{c}
0 \end{array} \right], 
\left[ \begin{array}{c}
1 \end{array} \right], 
\left[ \begin{array}{cc}
1 & 0  \end{array} \right], 
\left[ \begin{array}{cc}
0 & 1 \end{array} \right],
\left[ \begin{array}{c}
1 \\ 
0 \end{array} \right], 
\left[ \begin{array}{c}
0 \\ 
1 \end{array} \right], 
\left[ \begin{array}{cc}
1 & 0 \\ 
0 & 1 \end{array} \right], 
\left[ \begin{array}{cc}
0 & 1 \\ 
1 & 0 \end{array} \right] 
\bigg\}
\end{align*}
and the canonical $m$-basis of $\T$ is 
$\left\{
  \left[\begin{array}{cc}
          0 & 0
         \end{array}
  \right],
  \left[\begin{array}{c}
          0\\
          0
         \end{array}
  \right]
\right\}
$.
\end{example}

\begin{example}
\label{ex:av321_231_312_canonical_m-basis}
Let $\A$ be the permutation class $\Avperm(321,231,312)$. The canonical $m$-basis of $\A$ is $\{Q_1,Q_2\}$, with
\[
 Q_1 = \left[\begin{array}{cc}
          1 & 0\\
          0 & 0\\
          0 & 1
         \end{array}
 \right] \text{ and } Q_2 = \left[\begin{array}{ccc}
          1 & 0 & 0\\
          0 & 0 & 1
         \end{array}
 \right]\text{.}
\]
Indeed, it can be readily checked that $Q_1$ and $Q_2$ do not belong to $\A^{+}$ and are minimal for this property. 
Conversely if $\M \notin \A^{+}$ then $M$ contains one of the permutation matrices of $321$, $231$ and $312$, 
and hence contains $Q_1$ or $Q_2$ (and actually contains both of them). 
\end{example}

\begin{example}
\label{ex:vertical_bars_canonical_m-basis}
Let $\V$ be the class of polyominoes made of exactly one column (\ie vertical bars). The canonical $m$-basis of $\V$ is
$
\{\left[\begin{array}{c}
          0
         \end{array}
 \right],\left[\begin{array}{cc}
          1 & 1
         \end{array}
 \right]\}
$.
\end{example}

\begin{example}
\label{ex:rectangles_canonical_m-basis}
Let $\R$ be the class of polyominoes of rectangular shape. 
The canonical $m$-basis of $\R$ consists only of the matrix $\left[\begin{array}{c}
         0
         \end{array}
  \right]$.
\end{example}

There is one important difference between $p$-basis and canonical $m$-basis. 
Every antichain of permutations (resp. polyominoes) is the $p$-basis of a class. 
On the contrary, every antichain $\M$ of binary matrices describes a permutation (resp. polyomino) class $\Avperm(\M)$ (resp. $\Avp(\M)$), 
but not every such antichain is the canonical $m$-basis of the corresponding permutation (resp. polyomino) class -- 
see Examples~\ref{ex:1-12-21_m-basis} to~\ref{ex:rectangles_m-basis} below. 
Imposing the avoidance of matrices taken in an antichain being however a natural way of describing permutation and polyomino classes, 
let us define the following weaker notion of basis. 

\begin{definition}
Let \C be a class of permutations (resp. polyominoes). 
Every antichain $\M$ of matrices such that $\C=\Avperm(\M)$ (resp. $\Avp(\M)$) 
is called a \emph{matrix-basis} (or \emph{$m$-basis}) of $\C$. 
\label{def:m-basis}
\end{definition}

Examples~\ref{ex:1-12-21_m-basis} to~\ref{ex:rectangles_m-basis}
show several examples of $m$-bases of permutation and polyomino classes 
which are different from the canonical $m$-basis. 

\begin{example}
\label{ex:1-12-21_m-basis}
Consider the set $\M$ consisting of the following four matrices: 
$$M_1 = \left[ \begin{array}{cc}
1 & 0 \\ 
0 & 0 \end{array} \right], \, \, \, M_2 = \left[ \begin{array}{cc}
0 & 1 \\ 
0 & 0 \end{array} \right], \, \, \,  M_3 = \left[ \begin{array}{cc}
0 & 0 \\ 
1 & 0 \end{array} \right], \, \, \, M_4 = \left[ \begin{array}{cc}
0 & 0 \\ 
0 & 1 \end{array} \right] \, . $$
We may check that every permutation of size $3$ contains a matrix pattern $M \in \M$,
and that it actually contains each of these four $M_i$. 
Moreover, $\M$ is an antichain, and so is obviously each set $\{M_i\}$. 
Therefore, $\T = \Avperm(\M) =\Avperm(M_i)$, for each $1\leq i \leq 4$, 
eventhough these antichains characterizing $\T$ are not the canonical $m$-basis of $\T$ 
(see Example~\ref{ex:1-12-21_canonical_m-basis}). 
\end{example}

\begin{example}
\label{ex:av321_231_312_m-basis}
As explained in Example~\ref{ex:av321_231_312_canonical_m-basis}, 
$\A= \Avperm(Q_1)=\Avperm(Q_2)$ eventhough the canonical $m$-basis of 
$\A$ is $\{Q_1,Q_2\}$.
\end{example}

\begin{example}
\label{ex:vertical_bars_m-basis}
Recall from Example~\ref{ex:vertical_bars_canonical_m-basis} 
that the canonical $m$-basis of the class \V of vertical bars is $
\{\left[\begin{array}{c}
          0
         \end{array}
 \right],\left[\begin{array}{cc}
          1 & 1
         \end{array}
 \right]\}
$. But, we also have $\Avp\left(\left[\begin{array}{cc}
          1 & 1
         \end{array}
 \right]\right) = \V$.
\end{example}

\begin{example}
\label{ex:rectangles_m-basis}
Consider the sets 
$$\M_1=\left\{\left[ 
  \begin{array}{cc}
  1 & 0\end{array} 
\right],
\left[ 
  \begin{array}{cc}
  0 & 1\end{array} 
\right],
\left[
  \begin{array}{cc}
  0 & 0\end{array} 
\right],
\left[
  \begin{array}{c}
  0\\ 
  0\end{array}
\right],
\left[
  \begin{array}{c}
  1\\ 
  0\end{array} 
\right],
\left[
  \begin{array}{c}
  0\\ 
  1\end{array} 
\right]\right\}$$
and 
$$\M_2=\left\{\left[ 
  \begin{array}{cc}
  1 & 0\end{array} 
\right],
\left[
  \begin{array}{cc}
  0 & 1\end{array} 
\right],
\left[
  \begin{array}{c}
  1\\ 
  0\end{array} 
\right],
\left[
  \begin{array}{c}
  0\\ 
  1\end{array} 
\right]\right\}$$ 
We may easily check that $\M_1$ and $\M_2$ are antichains, 
and that their avoidance characterize the rectangular polyominoes of Example~\ref{ex:rectangles_canonical_m-basis}: 
$\R=\Avp(\M_1)=\Avp(\M_2)$.
\end{example}

Examples~\ref{ex:1-12-21_m-basis},~\ref{ex:av321_231_312_m-basis} and~\ref{ex:vertical_bars_m-basis} 
show in addition that the canonical $m$-basis is not always the more concise way 
of describing a class of permutations or of polyominoes 
by avoidance of submatrices. This motivates the following definition:

\begin{definition}
Let \C be a class of permutations (resp. polyominoes). 
A \emph{minimal $m$-basis} of $\C$ is an $m$-basis of $\C$ satisfying the following additional conditions: 
\begin{itemize}
 \item[$(1.)$] \M is a minimal subset subject to $\C=\Avperm(\M)$ (resp. $\Avp(\M)$), \\
 \emph{i.e.} for every strict subset $\M'$ of $\M$, $\C \neq \Avperm(\M')$ (resp. $\Avp(\M')$);
 \item[$(2.)$] for every submatrix $M'$ of some matrix $M \in \M$, we have 
\begin{itemize}
  \item[$i.$] $M'=M$ or
  \item[$ii.$] with $\M'=\M\setminus\{M\}\cup\{M'\}$, $\C \neq \Avperm(\M')$ (resp. $\Avp(\M')$).
 \end{itemize}
\end{itemize}
\label{def:minimal_m-basis}
\end{definition}

Condition~$(1.)$ ensures minimality in the sense of inclusion, 
while Condition~$(2.)$ ensures that it is not possible to replace a matrix of the minimal $m$-basis 
by another one of smaller dimensions. 
For future reference, let us notice that with the notations of Definition~\ref{def:minimal_m-basis}, 
the statement $\C \neq \Avperm(\M')$ (resp. $\Avp(\M')$) in Condition~$2.ii.$ is equivalent to 
$\C \varsubsetneq \Avperm(\M')$ (resp. $\Avp(\M')$), since the other inclusion always holds. 

To illustrate the relevance of Condition~$(2.)$, consider for instance
the $m$-basis $\{M_1\}$ of $\T$ (see Example~\ref{ex:1-12-21_m-basis}), with 
$$M_1 = \left[ \begin{array}{cc}
1 & 0 \\ 
0 & 0 \end{array} \right]\,\,.$$
Of course it is minimal in the sense of inclusion, however noticing that 
$$\T = \Avperm\left(\left[ \begin{array}{cc} 0 & 0 \end{array} \right]\right) = \Avperm\left(\left[ \begin{array}{c}
 0 \\ 
 0 \end{array} \right]\right)\,\,,$$
 with these excluded submatrices being submatrices of $M_1$, 
it makes sense \emph{not} to consider $\{M_1\}$ as a \emph{minimal} $m$-basis. 
This is exactly the point of Condition~$(2.)$.
\noindent
Actually, $\left\{\left[ \begin{array}{cc} 0 & 0 \end{array} \right]\right\}$ and 
$\left\{\left[ \begin{array}{c}
 0 \\ 
 0 \end{array} \right] \right\}$ both satify Conditions~$(1.)$ and~$(2.)$, \ie
are minimal $m$-basis of $\T$. 

This also illustrates the somewhat undesirable property that a class may have several minimal $m$-bases. 
This is not only true for the trivial class $\T$, but also for instance for $\A$: 
the $m$-bases $\{Q_1\}$ and $\{Q_2\}$ of $\A$ (see Example~\ref{ex:av321_231_312_m-basis}) 
are minimal $m$-bases of $\A$. 
We can see in the following some examples of polyomino classes in which the minimal $m$-basis is not unique.
\begin{example}[Injections]
\label{inj}
Let ${\cal I}$ be the class of {\em injections}, i.e. polyominoes having at most a zero entry for each row and column such as, for instance

 \begin{center}
\includegraphics[width=7cm]{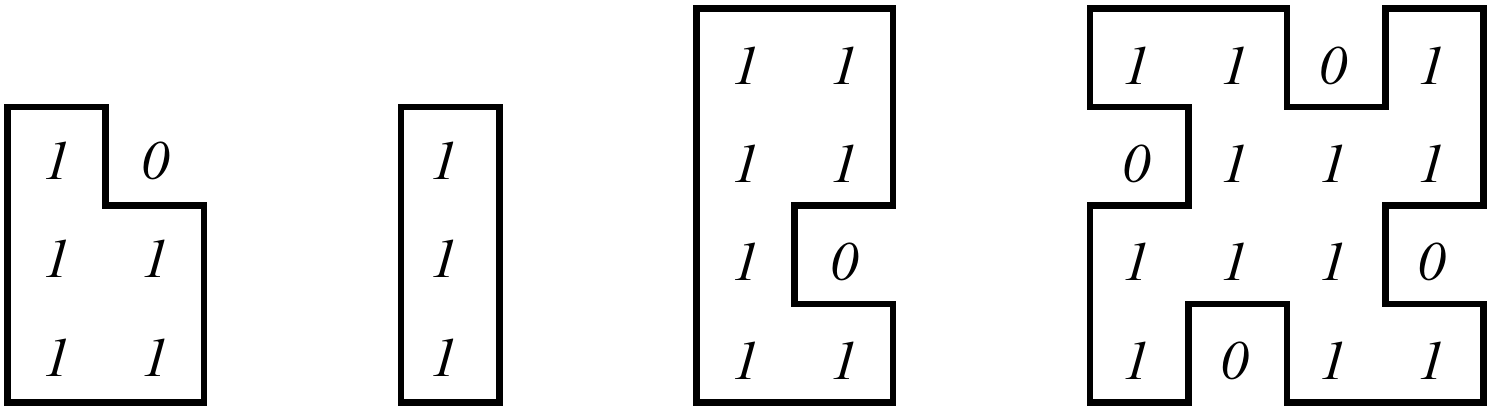}
\label{es2b}
\end{center}

The set ${\cal I}$ is clearly a polyomino class, and its $p$-basis is given by the minimal polyominoes which are not injections, i.e. the twelve polyominoes on the top of Fig.~\ref{fig:inj}. An $m$-basis of $\cal I$ is clearly given by set $${\cal M}=\left\{ \left[\begin{array}{cc}
         0 &0
         \end{array}
  \right], \, \left[\begin{array}{c}
         0 \\
         0
         \end{array}
  \right] \right\} \, .$$
Moreover, consider the sets
$$\M_1=\left\{\left[
  \begin{array}{ccc}
  0 &1 & 0\end{array}
\right],
\left[
  \begin{array}{ccc}
  1 &0 &0\end{array}
\right],
\left[
  \begin{array}{ccc}
  0 & 0 &1\end{array}
\right],
\left[
  \begin{array}{c}
  0\\
  1\\
  0\end{array}
\right],
\left[
  \begin{array}{c}
  0\\
  0\\
  1\end{array}
\right],
\left[
  \begin{array}{c}
  1\\
  0\\
  0\end{array}
\right]\right\}$$
and
$$\M_2=\left\{\left[
  \begin{array}{ccc}
  0 &1 & 0\end{array}
\right],
\left[
  \begin{array}{ccc}
  1 &0 &0\end{array}
\right],
\left[
  \begin{array}{ccc}
  0 & 0 &1\end{array}
\right],
\left[
  \begin{array}{c}
  0\\
  1\\
  0\end{array}
\right],
\left[
  \begin{array}{c}
  0\\
  0\\
  1\end{array}
\right],
\left[
  \begin{array}{c}
  1\\
  0\\
  0\end{array}
\right],
\left[
  \begin{array}{c}
  0\\
  0\\
  0\end{array}
\right]\right\}$$
We may easily check that $\M_1$ and $\M_2$ are antichains (see Fig.~\ref{fig:inj}),
and that their avoidance characterizes injections:
${\cal I}=\Avp(\M_1)=\Avp(\M_2)$. So, also $\M_1$ and $\M_2$ are $m$-bases.
\end{example}

\begin{figure}[ht]
\begin{center}
\includegraphics[width=13cm]{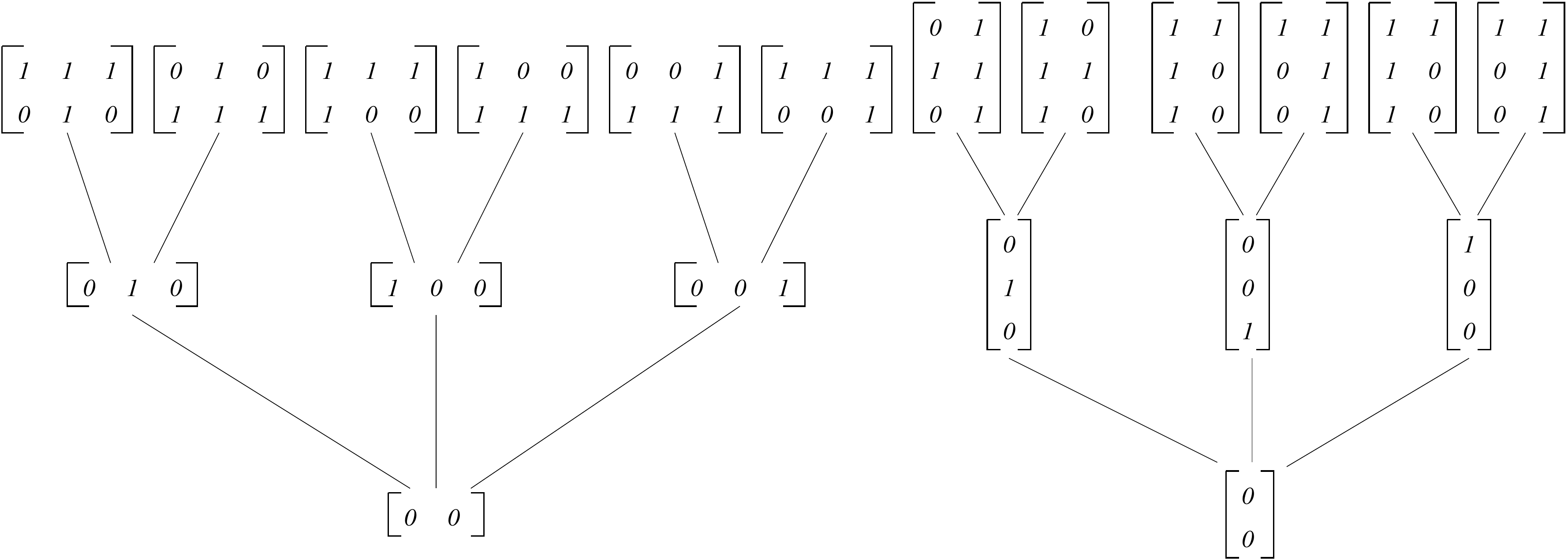}
\caption{The $p$-basis and some $m$-bases of $\cal I$.}
\label{fig:inj}
\end{center}
\end{figure}

However, the minimal $m$-bases of a class are relatively constrained:  

\begin{proposition}
Let \C be a class of permutations (resp. polyominoes) and let $\M$ be its canonical $m$-basis. 
The minimal $m$-bases of $\C$ are the subsets $\B$ of $\M$ that are minimal (for inclusion) under the condition $\C=\Avperm(\B)$ (resp. $\Avp(\B)$). 
\label{prop:minimal_m-basis}
\end{proposition}

\begin{proof}
For simplicity of the notations, let us forget the indices and write $\Av(\B)$ instead of $\Avperm(\B)$ (resp. $\Avp(\B)$).

Consider a subset $\B$ of $\M$ that is minimal for inclusion under the condition $\C=\Av(\B)$, 
and let us prove that $\B$ is a minimal $m$-basis of $\C$. 
$\B$ is clearly an $m$-basis of $\C$ satifying Condition~$(1.)$. 
Assume that $\B$ does not satisfy Condition~$(2.)$: there is some $M \in \B$ and some proper submatrix $M'$ of $M$, 
such that $\C=\Av(\B')$ for $\B'=\B\setminus\{M\}\cup\{M'\}$. 
By definition of the canonical $m$-basis, $M' \in \C^+$ (or $M$ would not be minimal for $\preccurlyeq$), 
so there exists a permutation (resp. polyomino) $P \in \C$ such that $M' \preccurlyeq P$. 
But then $P \notin \Av(\B') = \C$ bringing the contradiction that ensures that $\B$ satisfies Condition~$(2.)$. 

Conversely, consider a minimal $m$-basis $\B$ of $\C$ and a matrix $M \in \B$, 
and let us prove that $M$ belong to $\M$. 
Because of Condition~$(1.)$, this is enough to conclude the proof. 
First, notice that $M \notin \C^+$. 
Indeed, otherwise there would exist a permutation (resp. polyomino) $P \in \C$ such that $M \preccurlyeq P$, 
and we would also have $P \notin \Av(\B) = \C$, a contradiction. 
By definition, $\C^+ = \Avm(\M)$, so there exists $M' \in \M$ such that $M' \preccurlyeq M$. 
Since $\B$ is a minimal $m$-basis we either have $M=M'$, which proves that $M \in \M$, 
or we have $\Av(\B') \varsubsetneq \C$ for $\B'=\B\setminus\{M\}\cup\{M'\}$, 
in which case we derive a contradiction as follows. 
If $\Av(\B') \varsubsetneq \C$, then there is some permutation (resp. polyomino) $P \in \C$ 
which has a submatrix in $\B'$. It cannot be some submatrix in $\B\setminus\{M\}$, because $\C = \Av(\B)$. 
So $M' \preccurlyeq P$, which is a contradiction to $P \in \C = \Av(\M)$.
\end{proof}

\begin{example}
On our running examples, Proposition~\ref{prop:minimal_m-basis} ensures that both $\T$ and $\A$ each have two minimal $m$-basis, 
namely $\left\{\left[ \begin{array}{cc} 0 & 0 \end{array} \right]\right\}$ and 
$\left\{\left[ \begin{array}{c}
 0 \\ 
 0 \end{array} \right] \right\}$, 
and $\{Q_1\}$ and $\{Q_2\}$ respectively. 
The polyomino class $\V$ (resp. $\R$) has however a unique minimal $m$-basis: 
$\left\{\left[\begin{array}{cc}
          1 & 1
         \end{array}
 \right]\right\}$ (resp. $\left\{\left[\begin{array}{c}
         0
         \end{array}
  \right]\right\}$).
\end{example}

A natural question is then to ask for a characterization of the permutation (resp. polyomino) classes 
which have a unique minimal $m$-basis. 
In this direction we give the following remark.
\begin{remark}
Given a class of permutations (resp. polyominoes) \C, if the $p$-basis of \C is a minimal $m$-basis of \C, then the condition of Proposition~\ref{prop:minimal_m-basis} trivially holds and so the $p$-basis is the unique minimal $m$-basis of \C. This happens, for instance, in the case of the class \V of vertical bars (see Example \ref{ex:vertical_bars_m-basis}) or in the case of parallelogram polyominoes (in Section~\ref{sec:known_classes_poly}).
\end{remark}

\section{Relations between the $p$-basis and the $m$-bases}
\label{sec:from_one_basis_to_another}

\subsection{From an $m$-basis to the $p$-basis}

A permutation (resp. polyomino) class being now equipped with several notions of basis, 
we investigate how to describe one basis from another, and focus here on describing the $p$-basis from any $m$-basis. 

\begin{proposition}
Let \C be a permutation (resp. polyomino) class, and let \M be an $m$-basis of  $\C$.
Then the $p$-basis of \C consists of all permutations (resp. polyominoes) 
that contain a submatrix in \M, 
and that are minimal (w.r.t. $\sympattern$ resp. $\polypattern$) for this property. 
\label{prop:description_p-basis}
\end{proposition}

\begin{proof}
By Proposition~\ref{prop:perm_basis} (resp.~\ref{prop:polyomino_basis}),
the $p$-basis of \C is the set of minimal permutations (resp. polyominoes) that do not belong to $\C$ 
and are minimal w.r.t. $\sympattern$ (resp. $\polypattern$) for the property. 
The conclusion then follows by definition of $\M$ being an $m$-basis of $\C$: 
permutations (resp. polyominoes) not belonging to $\C$ are exactly those that contain a submatrix in \M. 
\end{proof}

\begin{example}
Figures~\ref{fig:basis321_231_312} and~\ref{fig:basisRectangularClass} 
give the $p$-basis of the classes $\A$ and $\R$ of Examples~\ref{ex:av321_231_312_canonical_m-basis} and~\ref{ex:rectangles_canonical_m-basis}, 
and illustrate its relation to their canonical $m$-basis.
\label{ex:relation_p-_and_m-basis}
\end{example}

\begin{figure}[ht]
\begin{center}
\includegraphics[width=12cm]{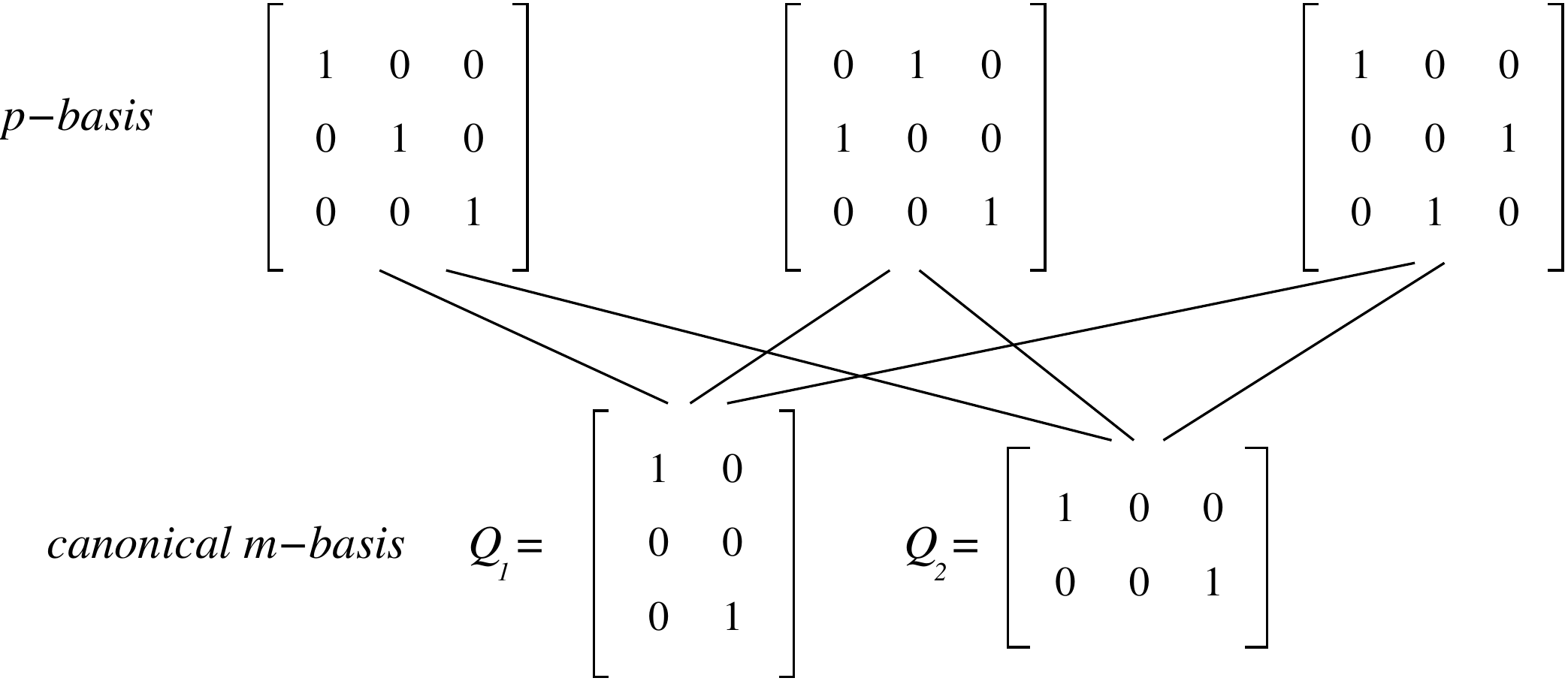}
\caption{The $p$-basis and the canonical $m$-basis of $\A = \Avperm(321,231,312)$.}
\label{fig:basis321_231_312}
\end{center}
\end{figure}

\begin{figure}[ht]
\begin{center}
\includegraphics[width=11cm]{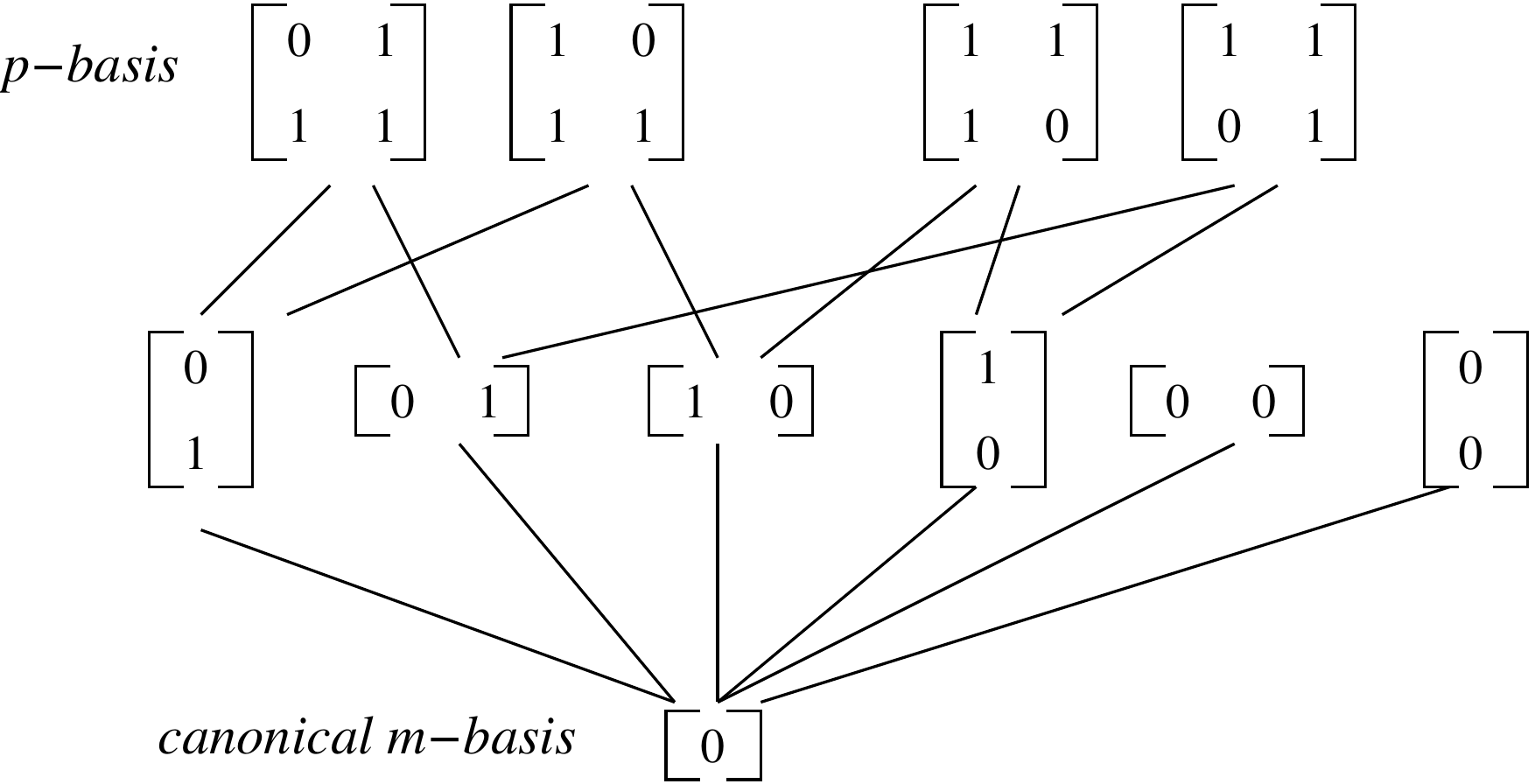}
\caption{The $p$-basis and the canonical $m$-basis of the class $\R$ of polyominoes having rectangular shape.}
\label{fig:basisRectangularClass}
\end{center}
\end{figure}

\medskip

In the case of permutation classes, 
Proposition~\ref{prop:description_p-basis} allows to compute the $p$-basis of any class \C, given a $m$-basis of \C. 
Indeed, the minimal permutations (in the sense of $\sympattern$) that contain a given matrix pattern $M$ are easily described:

\begin{proposition}
Let $M$ be a quasi-permutation matrix. 
The minimal permutations that contain $M$ are exactly those that may be obtained from $M$ 
by insertions of rows (resp. columns) with exactly one entry $1$, 
which should moreover fall into a column (resp. row) of $0$ of $M$. 

In particular, if $M$ has $k$ rows, $y$ of which are rows of $0$, 
and $\ell$ columns, $x$ of which are columns of $0$, 
then minimal permutations containing $M$ have size $k+x = \ell +y$. 
\label{prop:minimal_perm_containing_M}
\end{proposition}

It follows from Proposition~\ref{prop:minimal_perm_containing_M} 
that the $p$-basis of a permutation class \C can be easily computed from an $m$-basis of \C.
Also, Proposition~\ref{prop:minimal_perm_containing_M} implies that:

\begin{corollary}
If a permutation class has a finite $m$-basis 
(\ie is described by the avoidance of a finite number of submatrices) 
then it has a finite $p$-basis.
\label{cor:finiteness_of_m-basis_for_permutations}
\end{corollary}

\medskip

The situation is more complex if we consider polyomino classes. 
The description of the polyominoes containing a given submatrix 
is not as straightforward as in Proposition~\ref{prop:minimal_perm_containing_M}, 
and the analogue of Corollary~\ref{cor:finiteness_of_m-basis_for_permutations} does not hold for polyomino classes. 

\begin{proposition}
The polyomino class $\C= \Avp(M)$ defined by the avoidance of 
\[
M=
\left[\begin{array}{cccc}
1 & 0 & 0 & 1\\
1 & 1 & 0 & 1
\end{array}\right]
\]
has an infinite $p$-basis.
\label{prop:infinite_basis}
\end{proposition}

\begin{proof}
It is enough to exhibit an infinite sequence of polyominoes containing $M$, 
and that are minimal (for $\polypattern$) for this property. 
By minimality of its elements, such a sequence is necessarily an antichain, 
and it forms an infinite subset of the $p$-basis of $\C$. 
The first few terms of such a sequence are depicted in Figure~\ref{fig:infinite}, 
and the definition of the generic term of this sequence should be clear from the figure. 
We check by comprehensive verification that 
every polyomino $P$ of this sequence contains $M$, 
and additionaly that occurrences of $M$ in $P$ always involve 
the two bottommost rows of $P$, its two leftmost columns, and its rightmost column. 
Moreover, comprehensive verification also shows that 
every polyomino $P$ of this sequence is minimal for the condition $M \polypattern P$, 
\ie that every polyomino $P'$ occurring in such a $P$ as a proper submatrix avoids $M$. 
Indeed, the removal of rows or columns from such a polyomino $P$ either disconnects it or removes all the occurrences of $M$. 
\end{proof}

\begin{figure}[ht]
\begin{center}
\includegraphics[width=13cm]{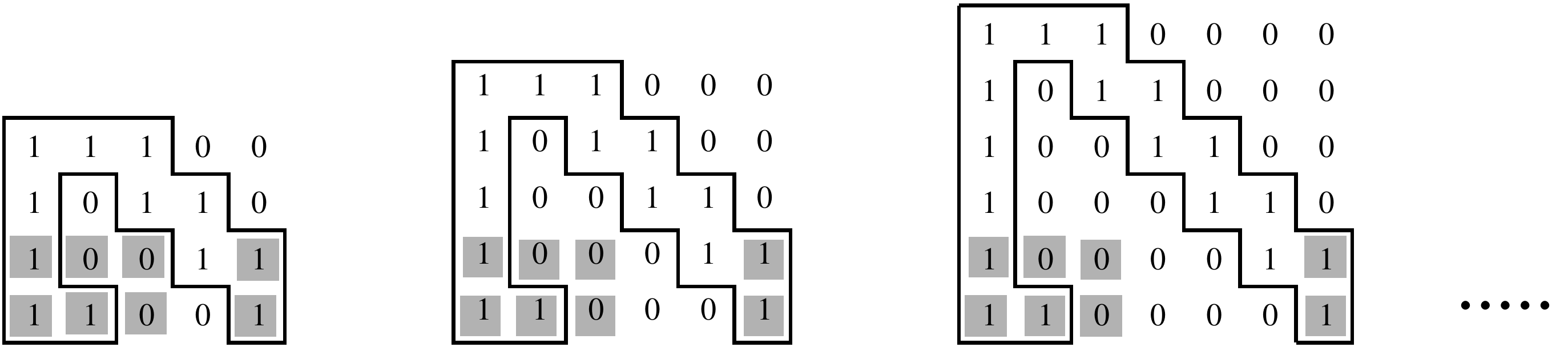}
\caption{An infinite antichain of polyominoes belonging to the $p$-basis of $\Avp(M)$.}
\label{fig:infinite}
\end{center}
\end{figure}

\subsection{Robust polyomino classes}


\begin{definition}
A class is \emph{robust} when all $m$-bases contain the $p$-basis.
\label{def:robust_class}
\end{definition}
For instance, the class  ${\cal I}$ of injections, considered in Example~\ref{inj}, is not robust, since there are $m$-bases disjoint from the $p$-basis.

The $p$-basis of a robust class has remarkable properties. 

\begin{proposition}\label{prop:minimalBasisRobust}
Let $\cal C$ be a robust class, and let $\cal P$ be the $p$-basis. Then, $\cal P$ is the unique $m$-basis \M which satisfies:
\begin{itemize}
 \item[$(1.)$] \M is a minimal subset subject to $\C=\Avp(\M)$, 
 \emph{i.e.} for every strict subset $\M'$ of $\M$, $\C \neq \Avp(\M')$;
 \item[$(2.)$] for every submatrix $M'$ of some matrix $M \in \M$, we have
 $M'=M$ or $\C \neq \Avp(\M')$, with $\M'=\M\setminus\{M\}\cup\{M'\}$.
\end{itemize}
\end{proposition}

\begin{proof}
 Let ${\cal P}$ be a set of polyominoes and let $\C=\Avp({\cal P})$ be a robust class. Condition $(1.)$ follows directly by Proposition~\ref{prop:description_p-basis}.  We can now proceed to prove the condition $(2.)$. 
 Let us suppose that there exists a proper submatrix $M'$ of some matrix $M \in \cal P$ such that $\C = \Avp(\cal P')$, with ${\cal P'}={\cal P}\setminus\{M\}\cup\{M'\}$. So we have that ${\cal P'} \polypattern {\cal P}$ and $\cal P'$ is an $m$-basis of $\C$. Since $\C$ is a robust class we have that ${\cal P} \polypattern {\cal P'}$ and then $\cal P = \cal P'$, in particular $M = M'$.
 
 Suppose that there exists another minimal $m$-basis $\M\neq{\cal P}$. By Proposition~\ref{prop:description_p-basis} on the $p$-basis, every pattern of $\cal M$ is contained in some pattern of $\cal P$ thus $\cal P$ contains $\cal M$. But $\C$ is a robust class, then $\cal M$ has to contain $\cal P$ and so $\cal P =\cal M$.
\end{proof}

\begin{remark}
 We notice that if $\cal P$ is the $p$-basis of a robust class $\C$ as a consequence of Proposition~\ref{prop:minimalBasisRobust}, $\cal P$ is also the minimal $m$-basis of $\C$.
\end{remark}


\begin{example}
\label{es:es5}
Let be $\C=\Avp(P,P')$, where $P, P'$ are depicted in Figure~\ref{fig:figEs5}. The class $\C$ is not robust, in fact there is an $m$-basis $M$ disjoint from the $p$-basis: 

\begin{figure}[htd]
\begin{center}
\includegraphics[width=7.5cm]{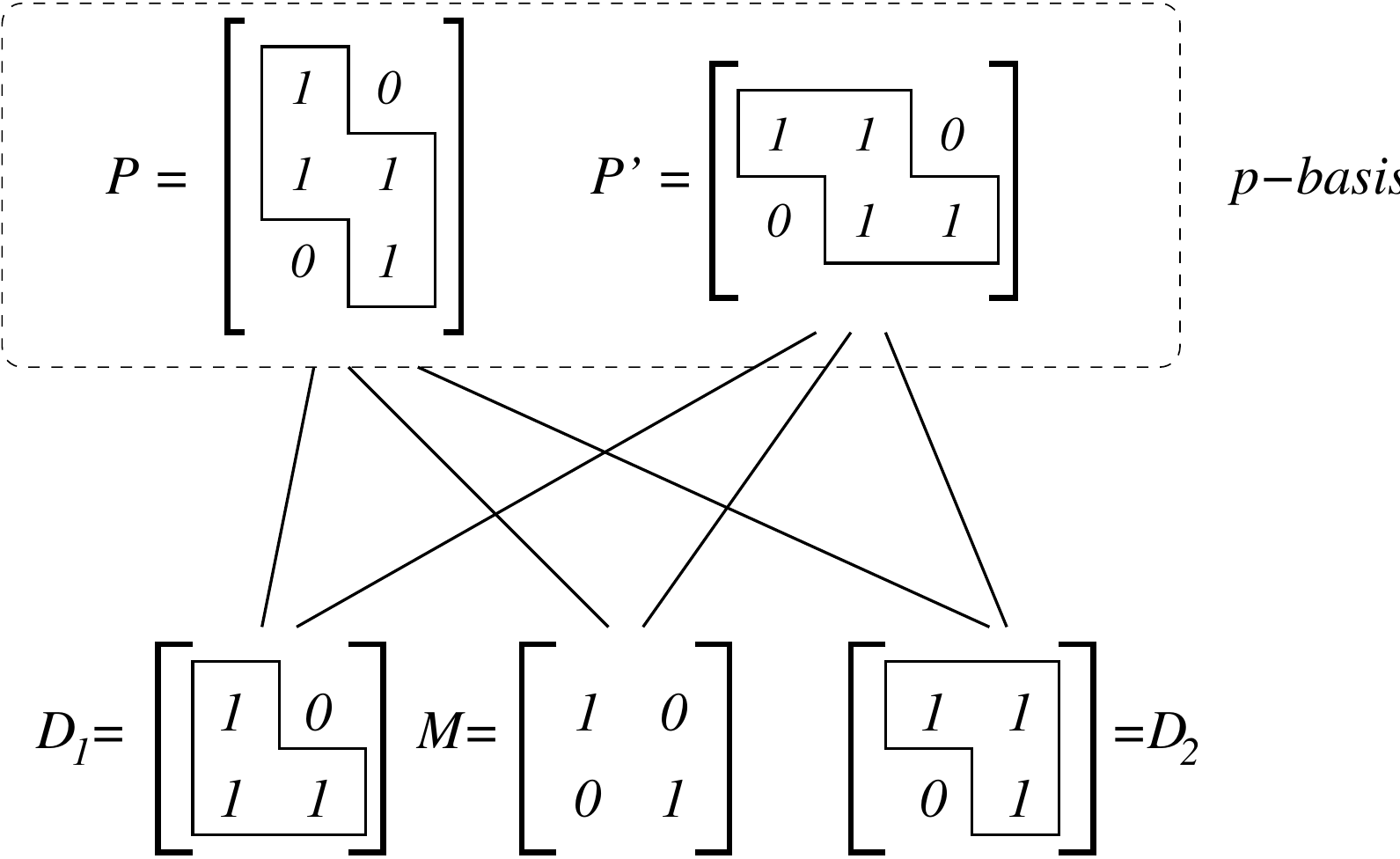}
\caption{A non robust class.}
\label{fig:figEs5}
\end{center}
\end{figure}

In practice, $P$ and $P'$ are precisely the minimal polyominoes which contain $M$ as a pattern, then by Proposition~\ref{prop:description_p-basis}, $\Avp(P,P')=\Avp(M)$.
We also notice that the meet of $P$ and $P'$ in the poset of polyominoes,  denoted $P\wedge P'$, is $\{ M, D_1, D_2 \}$, but the reader can check that $\Avp(D_1,D_2)\varsubsetneq \C$.


\end{example}

In this section, we try to establish some criteria to test the robustness of a class of polyominoes. First, we prove that it is easy to test robustness of a class whose basis is made of just one element:

\begin{proposition}\label{rb}
Let $M$ be a pattern. Then, $\Avp(M)$ is robust if and only if  $M$ is a polyomino.
\end{proposition}

\noindent {\em (Sketch of proof.)}

\noindent ($\Rightarrow$) If $M$ is not a polyomino, then it has a $p$-basis different from $M$. 

\smallskip

\noindent ($\Leftarrow$) Let be $M'$ a matrix that is not a polyomino such that $M'\polypattern M$.
Since $M'$ is not a polyomino then it contains at least two disconnected elements $B$ and $C$, and there are at least two possible ways to connect $B$ and $C$ (by rows or by columns). So, there exists at least another polyomino $P\neq M$ such that $M'\polypattern P$ and $P$ belongs to the $p$-bases of $\Avp(M')$. Thus, $\Avp(M')$ is properly included in $\Avp(M)$.

\smallskip

Thus, according to Proposition~\ref{rb}, the class  $\Avp\left(\begin{array}{cc}
          1 & 1\\
          0 & 1
         \end{array}\right)$ is robust. Now, it would be interesting to extend the previous result to a generic set of polyominoes, i.e. find sufficient and necessary conditions such that, given set of polyominoes $\cal P$, the class $\Avp({\cal P})$ is robust. 

\begin{proposition}
Let be $P_1, P_2$ two polyominoes and let be $\C=\Avp(P_1,P_2)$.
If for every element $\overline{P}$ in $P_1\wedge P_2$ we have that:
    \begin{description}
     \item $(b_1)$ $\overline{P}$ is a polyomino, or 
     \item $(b_2)$ every chain from $\overline{P}$ to $P_1$ (resp. from $\overline{P}$ to $P_2$) contains at least a polyomino $P'$ (resp. $P''$), different from $P_1$ (resp. $P_2$), such that $\overline{P} \polypattern P' \polypattern P_1$ (resp. $\overline{P} \polypattern P'' \polypattern P_2$),
    \end{description}
    then $\C$ is robust.
    \label{prop:condition_rob}
\end{proposition}

\begin{proof}

Clearly, if $P_1\wedge P_2$ contains only polyominoes, then \C is robust. On the other side, if $P_1\wedge P_2$ contains some patterns which are not polyominoes, then, for each of them, we have to check condition $(b_2)$. In fact, supposing that every chain from $\overline{P}$ to $P_1$ (resp. from $\overline{P}$ to $P_2$) contains at least a polyomino $P'$ (resp. $P''$), different from $P_1$ (resp. $P_2$), if \C was not robust, $P'$ (resp. $P''$) should belong to the $p$-basis instead of $P_1$ (resp. $P_2$).

\begin{figure}[htbp]
 \begin{center}
  \includegraphics[width=10cm]{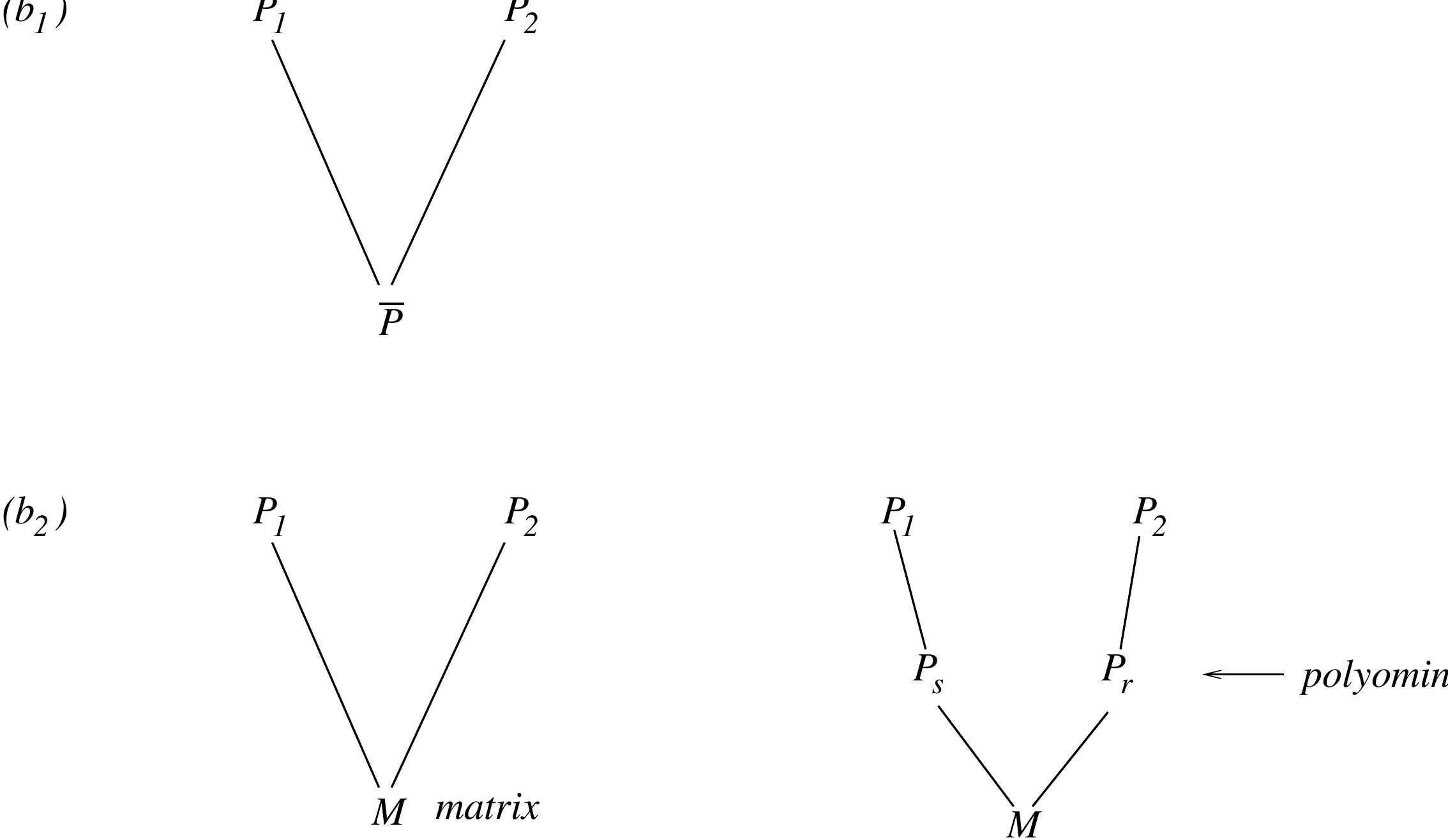}
  \caption{A graphical representation of the conditions of Proposition~\ref{prop:condition_rob}.}
\label{fig:figLemma1}
 \end{center}
\end{figure}
\end{proof}

\begin{example}\label{es:es6}

Let us consider the class $\C=\Avp(P_1,P_2)$, where $P_1$ and $P_2$ are the polyominoes depicted in Figure~\ref{fig:Es6}.

 \begin{figure}[htbp]
  \begin{center}
   \includegraphics[width=12cm]{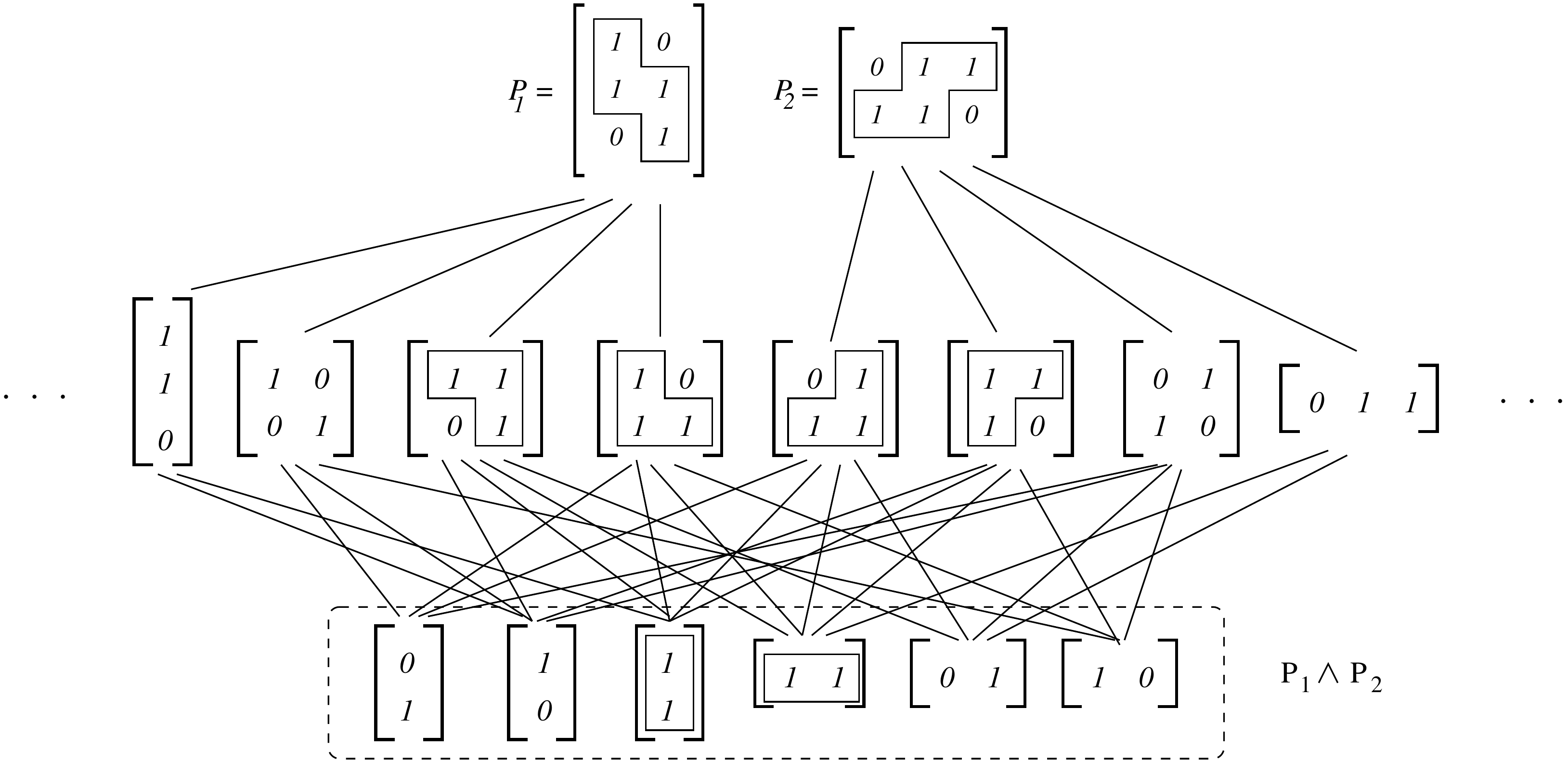}
  \caption{A robust class.}  
  \label{fig:Es6}
  \end{center}
 \end{figure}

Here, as shown in the picture, $P_1\wedge P_2$ contains six elements, and four of them are not polyominoes. However, one can check that, for each item $\overline{P}$ of these four matrices, there is a polyomino in the chain from $\overline{P}$ to $P_1$ (resp. from $\overline{P}$ to $P_2$). Thus, by Proposition~\ref{prop:condition_rob}, the class $\C$ is robust. 
\end{example}

However, the statement of Proposition~\ref{prop:condition_rob} cannot be inverted, as we can see in the following example.
 
 \begin{example}[Parallelogram polyominoes]\label{paral}
We can easily prove that parallelogram polyominoes, see Section~\ref{sec:polyFamilies} for its definition, can be represented by the avoidance of the submatrices:
$$
  M_1= \left[\begin{array}{cc}
          1 & 0\\
          1 & 1
         \end{array}
  \right], \, \, 
  M_2 = \left[\begin{array}{cc}
          1 & 1\\
          0 & 1
         \end{array}
  \right]\,\,.$$
These two patterns form a $p$-basis for the class $\cal P$ of parallelogram polyominoes. Clearly 
$$M_1 \wedge M_2 = 
\left\{\left[
  \begin{array}{cc}
  1 &1 \end{array}
\right], \, \, 
\left[
  \begin{array}{c}
  1 \\1 \end{array}
\right], \, \, 
\left[
  \begin{array}{c}
  0 \end{array}
\right]  
\right\} $$
If $\cal P$  was  not robust, then $M=\left[
  \begin{array}{c}
  0 \end{array}
\right]$ should belong to an $m$-basis of $\cal P$; precisely, we should have $\Avp(M)=\cal P$. But this is not true, since clearly $\Avp(M)$ is the class of rectangles. Thus, $\cal P$ is robust.
Observe that the set $\{ M_1, M_2, \left [ \, 1 \, 0 \, 1 \, \right ] \}$ forms an $m$-basis of the class, but it is not minimal w.r.t. set inclusion. 

\end{example}

%
%
%
\section{Some classes of permutations defined by submatrix avoidance}
\label{sec:known_classes_perm}

Many notions of pattern avoidance in permutations have been considered in the literature. 
The one of submatrix avoidance that we have considered is yet another one. 
In this section, we start by explaining how it relates to other notions of pattern avoidance. 
Then, using this approach to pattern avoidance, we give simpler proofs of the enumeration of some permutation classes. 
Finally, we show how these proofs can be brought to a more general level, 
to prove Wilf-equivalences of many permutation classes. 

\subsection[Generalized permutation patterns]{Submatrix avoidance and generalized permutation patterns}

A first generalization of pattern avoidance in permutations 
introduces \emph{adjacency constraints} among the elements of a permutation that should form an occurrence of a (otherwise classical) pattern. 
Such patterns with adjacency constraints are known as \emph{vincular} and \emph{bivincular} patterns, see Section \ref{sec:genPatt}, and a generalization with additional border constraints has recently be introduced  by~\cite{U}.

In some sense, the avoidance of submatrices in permutation is a dual notion 
to the avoidance of vincular and bivincular patterns. 
Indeed, in an occurrence of some vincular (resp. bivincular) pattern in a permutation $\sigma$, 
we impose that some elements of $\sigma$ must be adjacent (resp. have consecutive values). 
Whereas if there is a column (resp. row) of $0$ in a quasi-permutation matrix $M$, then 
an occurrence of $M$ in $\sigma$ is 
an occurrence of the largest permutation contained in $M$ in $\sigma$ 
where some elements of $\sigma$ are not allowed to be adjacent
(resp. to have consecutive values). 

For instance, 
\begin{itemize}
 \item $\Avperm \left(\left[ \begin{array}{cccc}
0 & 0 & 1 & 0 \\ 
1 & 0 & 0 & 0 \\
0 & 0 & 0 & 1 \end{array} \right]\right)$ denotes the set of all permutations such that 
in any occurrence of $231$ the elements mapped by the $2$ and the $3$ are at adjacent positions; 
 \item $\Avperm \left(\left[ \begin{array}{cccc}
0 & 0 & 1 & 0 \\ 
1 & 0 & 0 & 0 \\ 
0 & 0 & 0 & 0 \\
0 & 0 & 0 & 1 \end{array} \right]\right)$ denotes the set of all permutations such that 
in any occurrence of $231$ the elements mapped by the $2$ and the $3$ are at adjacent positions,
and the elements mapped by the $1$ and the $2$ have consecutive values;
 \item $\Avperm \left(\left[ \begin{array}{ccc}
0 & 0 & 0 \\ 
0 & 1 & 0 \\ 
1 & 0 & 0 \\
0 & 0 & 1 \end{array} \right]\right)$ denotes the set of all permutations such that 
in any occurrence of $231$ the element mapped by the $3$ is the maximum of the permutation. 
\end{itemize}

As noticed in Remark~\ref{rem:submatrix_avoidance_implies_class}, 
sets of permutations defined by avoidance of submatrices are permutation classes. 
On the contrary, sets of permutations defined by avoidance of vincular or bivincular patterns are not downward closed for in general. 
This shows that, 
even though they are very useful for characterizing important families of permutations, like Baxter permutations~\cite{GBaxter}, 
the adjacency constraints introduced vincular and bivincular pattern 
do not fit really well in the study of permutation \emph{classes}. 
The above discussion suggests that, with this regard, it is more convienent to introduce instead 
\emph{non-adjacency constraints} in permutation patterns, which corresponds to rows and columns of $0$ in quasi-permutation matrices. 

\medskip

Let us consider now another generalization of permutation patterns, the \emph{Mesh patterns}. It has itself been generalized in several way, in particular by \'Ulfarsson who introduced in~\cite{U} the notion of marked mesh patterns. 
It is very easy to see that the avoidance of a quasi-permutation matrix with no uncovered $0$ entries 
can be expressed as the avoidance of a special form of marked-mesh pattern. 
Indeed, a row (resp. $k$ consecutive rows) of $0$ in a quasi-permutation matrix with no uncovered $0$ entries 
corresponds to a mark, spanning the whole pattern horizontally, 
indicating the presence of at least one (resp. at least $k$) element(s). 
The same holds for columns and vertical marks. 
Figure~\ref{fig:marked-mesh_pattern} shows an example of a quasi-permutation matrix with no uncovered $0$ entries 
with the corresponding marked-mesh pattern. 

\begin{figure}[ht]
  \begin{center}
   \includegraphics[scale=0.6]{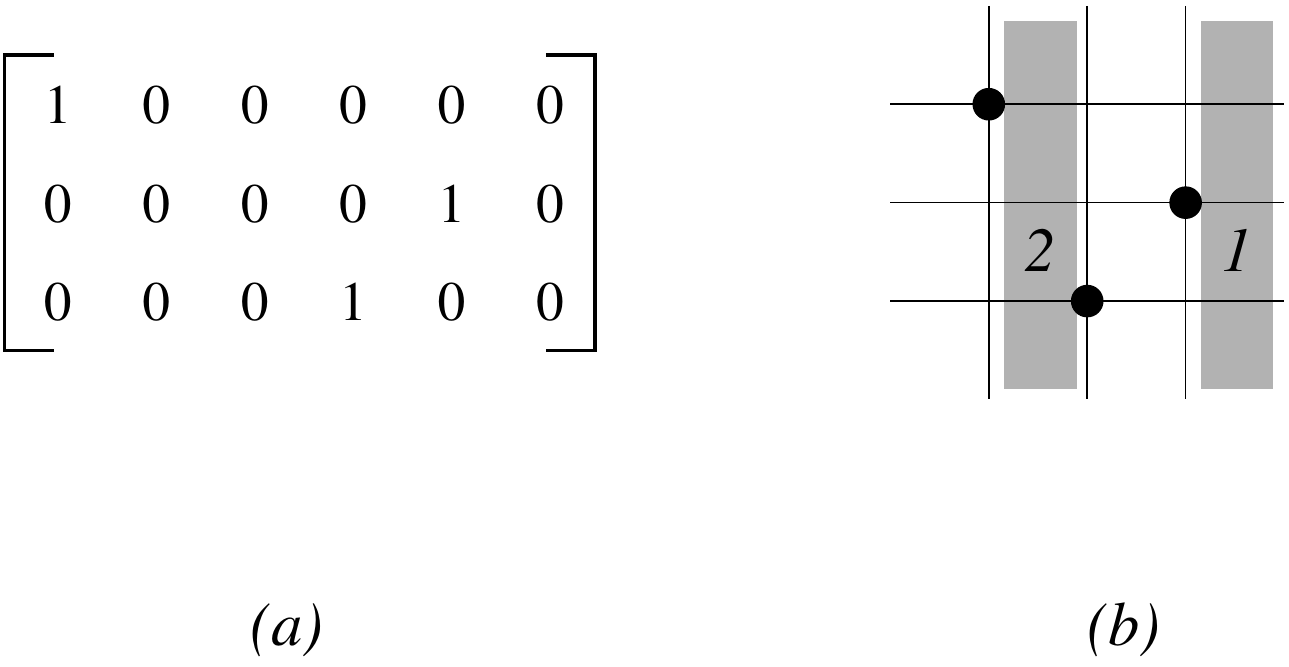}
   \caption{A quasi-permutation matrix with no uncovered $0$ entries 
    and the corresponding marked-mesh pattern.}
\label{fig:marked-mesh_pattern}
\end{center}
\end{figure}


\subsection{A different look at some known classes of permutations}

Several permutation classes avoiding three patterns of size $3$ or four patterns of size $4$ that have been studied in the literature (and are referenced in Guibert's catalogue~\cite[Appendix A]{GuibertThese}) 
are easier to describe with the avoidance of just one submatrix, as we explain in the following. 
For some of these classes, the description by submatrix avoidance also allows to provide a simple proof of their enumeration. 

In this paragraph, we do not consider classes that are equal up to symmetry (reverse, inverse, complement, and their compositions). 
But the same results of course apply (up to symmetry) to every symmetry of each class considered. 
 
\paragraph*{The class $\F = \Avperm(123,132,213)$.}
This class was studied in the Simion-Schmidt article~\cite{SimionSchmidt}, 
about the systematic enumeration of permutations avoiding patterns of size $3$. 
An alternative description of $\F$ is 
\[
\F = \Avperm(M_F)  \text{ where } M_F = \left[ \begin{array}{ccc}
0 & 0 & 1 \\ 
1 & 0 & 0 \end{array} \right]\text{.}
\]
It follows immediately since $123,132$ and $213$ are exactly the permutations which cover $M_F$ 
in the sense of Propositions~\ref{prop:description_p-basis} and~\ref{prop:minimal_perm_containing_M}. 

\cite{SimionSchmidt} shows that $\F$ is enumerated by the Fibonacci numbers. 
Of course, it is possible to use the description of $\F$ by the avoidance of $M_F$ 
to see that every permutation $\sigma \in \F$ decomposes as $\sigma = s_1 \cdot s_2 \cdots s_k$, 
where the sequences $s_i$ are either $x$ or $x (x+1)$ for some integer $x$, 
and are such that for $i<j$ the elements of $s_i$ are larger that those of $s_j$. 
And from this description, an easy induction shows that $\F$ is enumerated by the Fibonacci numbers. 
However, this is just rephrasing the original proof of~\cite{SimionSchmidt}. 

\paragraph*{The class $\G = \Avperm(123,132,231)$.}
This class is also studied in~\cite{SimionSchmidt}, where it is shown that there are $n$ permutations of size $n$ in $\G$. 
The enumeration is obtained by a simple inductive argument, 
which relies on a recursive description of the permutations in $\G$. 

For the same reasons as in the case of $\F$, $\G$ is alternatively described by 
\[
\G = \Avperm(M_G)  \text{ where } M_G = \left[ \begin{array}{ccc}
0 & 1 & 0 \\ 
1 & 0 & 0 \end{array} \right]\text{.}
\]
Any occurrence of a pattern $12$ in a permutation $\sigma$ can be extended to an occurrence of $M_G$, 
as long as it does not involve the last element of $\sigma$. 
So from this characterization, it follows that the permutations of $\G$ are all decreasing sequences followed by one element. 
This describes the permutations of $\G$ non-recursively, and give immediate access to the enumeration of $\G$. 

\paragraph*{The classes $\H = \Avperm(1234,1243,1423,4123)$, $\J = \Avperm(1324,1342,1432,4132)$ and $\K = \Avperm(2134,2143,2413,4213)$.}

These three classes have been studied in~\cite[Section 4.2]{GuibertThese}, 
where it is proved that they are enumerated by the central binomial coefficients. 
The proof first gives a generating tree for these classes, 
and then the enumeration is derived analytically from the corresponding rewriting system. 
In particular, this proof does not provide a description of the permutations in $\H$, $\J$ and $\K$ 
which could be used to give a \emph{combinatorial} proof of their enumeration. 
Excluded submatrices can be used for that purpose.

As before, because the $p$-basis of $\H$, $\J$ and $\K$ are exactly the permutations which cover the matrices $M_H$, $M_J$ and $M_K$ given below, we have 

\[
\begin{array}{lllll}
\H & = & \Avperm(M_H)  \text{ where } M_H & = & \left[ \begin{array}{ccc}
0 & 0 & 0 \\ 
0 & 0 & 1 \\ 
0 & 1 & 0 \\ 
1 & 0 & 0 \end{array} \right]\text{, }\\
 & & & &\\
\J & = & \Avperm(M_J)  \text{ where } 
M_J & = & \left[ \begin{array}{ccc}
0 & 0 & 0 \\ 
0 & 1 & 0 \\ 
0 & 0 & 1 \\ 
1 & 0 & 0 \end{array} \right]\text{,}\\
 & & & &\\
\K & = & \Avperm(M_K)  \text{ where } M_K & = & \left[ \begin{array}{ccc}
0 & 0 & 0  \\ 
0 & 0 & 1  \\ 
1 & 0 & 0  \\ 
0 & 1 & 0 \end{array} \right]\text{.}
\end{array}
\]

Similarly to the case of $\G$, any occurrence of a pattern $123$ (resp. $132$, resp. $213$)
in a permutation $\sigma$ can be extended to an occurrence of $M_H$ (resp. $M_J$, resp. $M_K$), 
as long as it does not involve the maximal element of $\sigma$. 
Conversely, if a permutation $\sigma$ contains $M_H$ (resp. $M_J$, resp. $M_K$), 
then there is an occurrence of $123$ (resp. $132$, resp. $213$) in $\sigma$ that does not involves its maximum. 
Consequently, the permutations of $\H$ (resp. $\J$, resp. $\K$) are exactly those of avoiding $123$ (resp. $132$, resp. $213$) 
to which a maximal element has been added. 
This provides a very simple description of the permutations of $\H$, $\J$ and $\K$. 
Moreover, recalling that for any permutation $\pi \in \sym_3$ $\Avperm(\pi)$ is enumerated by the Catalan numbers, 
it implies that the number of permutations of size $n$ in $\H$ (resp. $\J$, resp. $\K$) in $n\times C_{n-1} = {{2n-2}\choose {n-1}}$. 

\subsection[Propagating enumeration results]{Propagating enumeration results and \\ Wilf-equivalences with submatrices}

The similarities that we observed between the cases of the classes $\G$, $\H$, $\J$ and $\K$ are not a coincidence. 
Indeed, they can all be encapsulated in following proposition, which simply pushes the same idea forward to a general setting. 

\begin{proposition}
Let $\tau$ be a permutation. 
Let $M_{\tau,top}$ (resp. $M_{\tau,bottom}$) be the quasi-permutation matrix obtained by adding 
a row of $0$ entries above (resp. below) the permutation matrix of $\tau$. 
Similarly, let $M_{\tau,right}$ (resp. $M_{\tau,left}$) be the quasi-permutation matrix obtained by adding 
a column of $0$ entries on the right (resp. left) of the permutation matrix of $\tau$. 
The permutations of $\Avperm(M_{\tau,top})$ (resp. $\Avperm(M_{\tau,bottom})$, resp. $\Avperm(M_{\tau,right})$, resp. $\Avperm(M_{\tau,left})$) 
are exactly the permutations avoiding $\tau$ to which a maximal (resp. minimal, resp. last, resp. first) element has been added. 
\label{prop:perm+0_avoidance}
\end{proposition}

\begin{proof}
We prove the case of $M_{\tau,top}$ only, the other cases being identical up to symmetry. 

Consider a permutation $\sigma \in \Avperm(M_{\tau,top})$, and denote by $\sigma'$ the permutation obtained deleting the maximum of $\sigma$. 
Assuming that $\sigma'$ contains $\tau$, then $\sigma$ would contain $M_{\tau,top}$, which contradicts $\sigma \in \Avperm(M_{\tau,top})$; hence $\sigma' \in \Avperm(\tau)$. 

Conversely, consider a permutation $\sigma' \in \Avperm(\tau)$, and a permutation $\sigma$ obtained by adding a maximal element to $\sigma'$. 
Assume that $\sigma$ contains $M_{\tau,top}$, and consider an occurrence of $M_{\tau,top}$ in $\sigma$. 
Regardless of whether or not this occurrence involves the maximum of $\sigma$, 
it yields an occurrence of $\tau$ in $\sigma'$, and hence a contradition. 
Therefore, $\sigma \in \Avperm(M_{\tau,top})$. 
\end{proof}

Proposition~\ref{prop:perm+0_avoidance} has two very nice consequences in terms of enumeration. 
When the enumeration of the class is known, 
it allows to deduce the enumeration of four other permutation classes.
Similarly, for any pair of Wilf-equivalent classes (\emph{i.e.} of permutation classes having the same enumeration), 
it produces four other pairs of Wilf-equivalent classes.

\begin{corollary}
Let $\C$ be a permutation class whose $p$-basis is $\B$. 
Let $\M = \{M_{\tau,top} \mid \tau \in \B\}$. 
Denote by $c_n$ the number of permutation of size $n$ in $\C$. 
The permutation class $\Avperm(\M)$ is enumerated by the sequence $(n \cdot c_{n-1})_n$. 
The same holds replacing $M_{\tau,top}$ with $M_{\tau,bottom}$, $M_{\tau,right}$ or $M_{\tau,left}$. 
\end{corollary}

\begin{corollary}\label{cor:WE}
Let $\C_1$ and $\C_2$ be two Wilf-equivalent permutation classes whose $p$-basis are respectively $\B_1$ and $\B_2$.
Let $\M_1 = \{M_{\tau,top} \mid \tau \in \B_1\}$ and $\M_2 = \{M_{\tau,top} \mid \tau \in \B_2\}$. 
The permutation classes $\Avperm(\M_1)$ and $\Avperm(\M_2)$ are also Wilf-equivalent. 
The same holds replacing $M_{\tau,top}$ with $M_{\tau,bottom}$, $M_{\tau,right}$ or $M_{\tau,left}$. 
\end{corollary}

\section{Some classes of polyominoes defined by submatrix avoidance}
\label{sec:known_classes_poly}
In this section, we show that several families of polyominoes studied in the literature can be characterized in terms of submatrix avoidance. For more details on these families of polyominoes we address the reader to Section~\ref{sec:polyFamilies}. We also use submatrix avoidance to introduce new classes of polyominoes. 

\subsection[Characterization of known classes of polyminoes]{Characterizing known classes of polyminoes by submatrix avoidance}

There are plenty of examples of polyomino classes that can be described by the avoidance of submatrices. 
We provide a few of them here, however without giving detailed definitions of these classes.  
Figure~\ref{fig:polyominoes} shows examples of polyominoes belonging to the classes that we study. 


\begin{figure}[ht]
\begin{center}
\includegraphics[width=11cm]{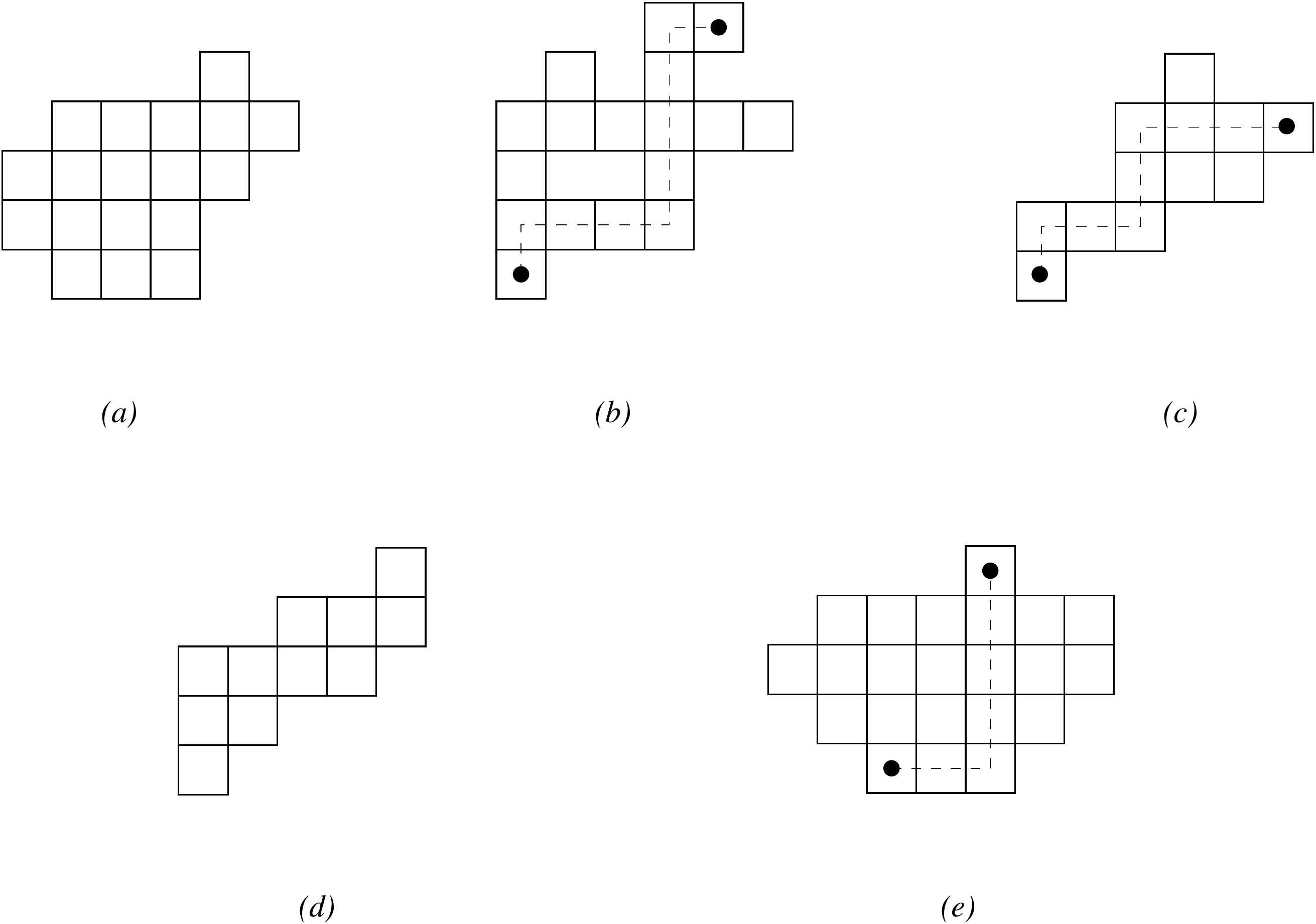}
\caption{$(a)$ A convex polyomino; $(b)$ A directed polyomino; $(b)$ A directed-convex polyomino; $(b)$ A parallelogram polyomino; $(b)$ An $L$-convex polyomino.}
\label{fig:polyominoes}
\end{center}
\end{figure}

\paragraph{Convex polyominoes.} These are defined by imposing one of the simplest geometrical constraints: the connectivity of rows/columns.

Figure~\ref{fig:polyominoes}$(a)$ shows a convex polyomino. 
The convexity constraint can be easily expressed in terms of excluded submatrices since its definition already relies on some specific configurations that the cells of each row and column of a polyomino have to avoid: 

\begin{proposition}
\label{prop:m-basis_convex}
Convex polyominoes can be represented by the avoidance of the two submatrices 
$H = {\footnotesize \left[ \begin{array}{ccc}
1 & 0 & 1 \end{array} \right]}$ and $V = {\footnotesize\left[ \begin{array}{c}
1 \\
0 \\
1 \end{array} \right]}$.
\end{proposition}

More precisely, the avoidance of the matrix $H$ (resp. $V$)
indicates the $h$-convexity (resp. $v$-convexity). 
Because removing columns (resp. rows) preserves the $h$-convexity (resp. $v$-convexity), 
$\{H,V\}$ is the canonical $m$-basis, and hence by Proposition~\ref{prop:minimal_m-basis} the unique minimal $m$-basis, of the class of convex polyominoes. 
From Proposition~\ref{prop:description_p-basis}, it is also easy to determine the $p$-basis of this class: 
it is the set of four polyominoes $\{H_1,H_2,V_1,V_2\}$ depicted in Figure~\ref{fig:convex_p-basis}.

\begin{figure}[ht]
\begin{center}
\includegraphics[width=13cm]{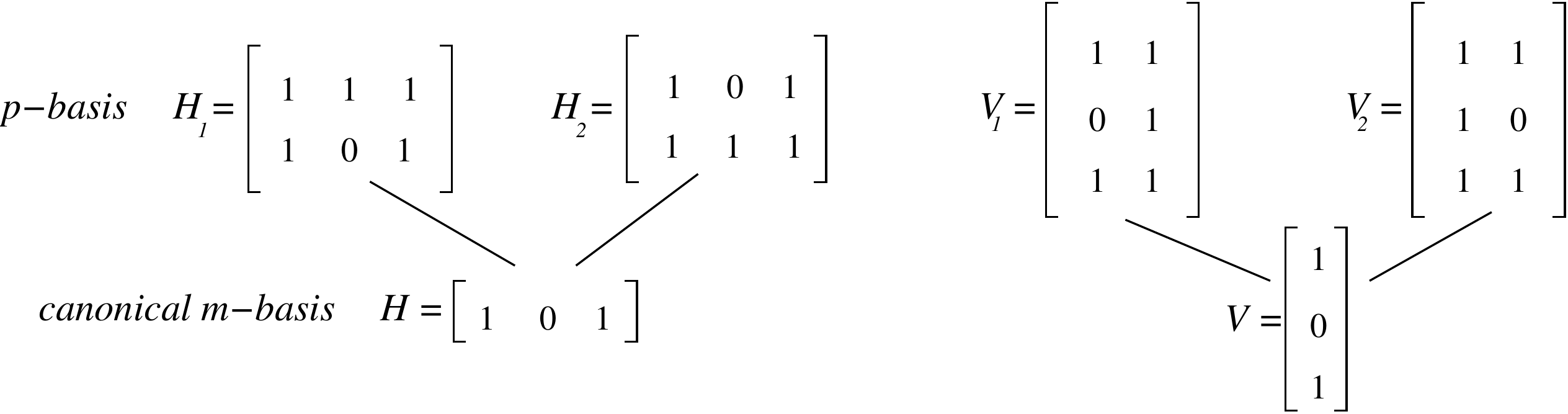}
\caption{The $p$-basis and canonical $m$-basis of the class of convex polyominoes.}
\label{fig:convex_p-basis}
\end{center}
\end{figure}

We advise the reader that, in the rest of the section, for each polyomino class, we will provide only a matrix description of the basis. Indeed, in all these examples the $p$-basis can easily be obtained from the given $m$-basis with Proposition~\ref{prop:description_p-basis}, like in the case of convex polyominoes. 

\paragraph{Directed-convex polyominoes.}\label{par:dirConvex}

Directed-convex polyominoes are defined using the notion of internal path to a polyomino. An {\em (internal) path} of a polyomino is a sequence of distinct cells $(c_1,\dots,c_n)$ of the polyomino such that every two consecutive cells in this sequence are edge-connected; according to the respective positions of the cells $c_i$ and $c_{i+1}$, we say that the pair $(c_i,c_{i+1})$ forms a {\em north, south, east or west step} in the path, respectively.


Figure~\ref{fig:polyominoes}$(b)$ depicts a directed polyomino. 
The reader can check 
that the set of directed polyominoes is not a polyomino class. 
However, the set of \emph{directed-convex} polyominoes 
(\ie of polyominoes that are both directed and convex -- see Figure~\ref{fig:polyominoes}$(c)$) 
is a class:

\begin{proposition}
The class ${\cal D}$ of directed-convex polyominoes is characterized by the avoidance of the submatrices
$H= {\footnotesize 
  \left[\begin{array}{ccc}
          1 & 0 & 1
         \end{array}
  \right]}$, 
$V={\footnotesize \left[\begin{array}{c}
          1\\
          0\\
          1
         \end{array}
  \right]}$
and $  D={\footnotesize \left[\begin{array}{cc}
          1 & 1\\
          0 & 1
         \end{array}
  \right]}$.
\label{prop:m-basis_directed_convex}
\end{proposition}

\begin{proof}
Let us prove that ${\cal D}=\Avp (H,V,D)$.

First, let $P\in \Avp (H,V,D)$. $P$ is a convex polyomino and avoids the submatrix $D$. Let $s$ be its source, necessarily the leftmost cell at the lower ordinate. Let us proceed by contradiction assuming that $P$ is not directed, \ie there exists a cell $c$ of $P$ such that all the paths from $s$ to $c$ contain either a south step or a west step. Consider one of these paths $p$ with minimal length (defined as the number of cells), say $l$, and having at least a south step (if $p$ has at least a west step, a similar reasoning holds). If $s$ lies at the same ordinate as $c$ or below it, then the presence of a south step in $p$ implies that there exist two cells $c_i$ and $c_j$ of $P$, with $1\leq i < j \leq l$, where $p$ crosses the ordinate of $c$ with a north and a south step, respectively. Since $p$ is minimal, then the row of cells containing $c_i$ and $c_j$ is not connected, contradicting that $P$ is $h$-convex. Otherwise, $s$ lies above $c$, and there exists a point $c_i \not= s$ with $1< i <l$, having the same ordinate as $s$, and such that it forms with the cell $c_{i+1}$ a south step. So, the four cells $s$, $c_i$, $c_{i+1}$ and the empty cell below $s$ form the pattern $D$, giving the desired contradiction.

Conversely, let $P$ be a directed convex polyomino. We will proceed by contradiction assuming that $P$ contains the pattern $D$. Let $c_1$ and $c_2$ be the cells of $P$ that correspond to the upper left and to the lower right cells of $D$, respectively. Two cases have to be considered: if the source cell $s$ of $P$ lies above $c_2$, then each path leading from $s$ to $c_2$ has to contain at least one south step, which contradicts the fact that $P$ is directed. On the other hand, if $s$ lies in the same row of $c_2$ or below it, then each path leading from $s$ to $c_1$ either runs entirely on the left of $c_1$, so that $P$ contains the pattern $H$, which is against the $h$-convexity of $P$, or it contains a west step, which again contradicts the fact that $P$ is directed.
\end{proof}

\paragraph{Parallelogram polyominoes.}
Another widely studied family of polyominoes --that can also be defined using a notion of path, this time of {\em boundary path}-- is that of {\em parallelogram polyominoes}.


Parallelogram polyominoes (see Figure~\ref{fig:polyominoes}$(d)$) are a polyomino class. Since the proof of this fact resembles that of Proposition~\ref{prop:m-basis_directed_convex}, it will be left to the reader.

\begin{proposition}\label{prop:parallelogram}
Parallelogram polyominoes are characterized by the avoidance of the submatrices:
$ {\footnotesize
   \left[\begin{array}{cc}
          1 & 0\\
          1 & 1
         \end{array}
  \right]}$ and ${\footnotesize
  \left[\begin{array}{cc}
          1 & 1\\
          0 & 1
         \end{array}
  \right]}$. 
This is also the $p$-basis of the class of parallelogram polyominoes. 
\label{prop:m-basis_parall}
\end{proposition}

\paragraph{$L$-convex polyominoes.}
Parallelogram polyominoes and directed-convex polyominoes form subclasses of the class of convex polyominoes 
that are both defined in terms of paths. 
The relationship between these two notions is closer than it may appear. 

Let us consider the $1$-convex polyominoes are more commonly called $L$-convex polyominoes, (see Figure~\ref{fig:polyominoes}$(e)$).

Here, we study how the constraint of being $k$-convex can be represented in terms of submatrix avoidance.
In order to reach this goal, let us present some basic definitions and properties from the field of {\em discrete tomography} \cite{ryser}. Given a binary matrix, the vector of its {\em horizontal} (resp. {\em vertical}) {\em projections} is the vector of the row (resp. column) sums of its elements. 
In 1963 Ryser~\cite{ryser} established a fundamental result which, using our notation, can be reformulated as follows:

\begin{theorem}\label{th:ryser}
A binary matrix is uniquely determined by its horizontal and vertical projections if and only if it does not contain $S_1= {\footnotesize \left[\begin{array}{cc}
          1 & 0\\
          0 & 1
         \end{array}
  \right]}$ and $S_2 = {\footnotesize \left[\begin{array}{cc}
          0 & 1\\
          1 & 0
         \end{array}
  \right]}$ as submatrices.
\end{theorem}

Now we show that the set $\cal L$ of $L$-convex polyominoes is a polyomino class. 

\begin{proposition}\label{prop:lconvex}
$L$-convex polyominoes are characterized by the avoidance of the submatrices $H,V,S_1$ and $S_2$. 
In other words, ${\cal L}=\Avp(H,V,S_1,S_2)$.
\label{prop:m-basis_lconvex}
\end{proposition}

\begin{proof}
First, let $P$ be a convex polyomino that avoids the submatrices $S_1$ and $S_2$. Let us proceed by contradiction assuming that $P$ is not $L$-convex. It means that there exists a pair of cells $c_1$ and $c_2$ of $P$ such that all the paths from $c_1$ to $c_2$ are not $L$-paths, \ie they have at least two changes of directions. Suppose that $c_2$ lies below $c_1$, on its right (if not, similar reasonings hold). 
Consider the path from $c_1$ that goes always right (resp. down) until it encounters a cell $c'$ (resp. $c''$) which does not belong to $P$. This happens before reaching the abscissa (resp. ordinate) of $c_2$, otherwise by convexity there would be a path from $c_1$ to $c_2$ with one change of direction. By convexity, all the cells on the right of $c'$ (resp. below $c''$) do not belong to $P$, so the four cells having the same abscissas and ordinates of $c_1$ and $c_2$ form the pattern $S_1$, giving the desired contradiction.

Conversely, let $P$ be an $L$-convex polyomino. 
By convexity, $P$ does not contain the submatrices $H$ and $V$. Moreover, in \cite{CFRR} it is proved that an $L$-convex polyomino is uniquely determined by its horizontal and vertical projections, so by Theorem~\ref{th:ryser} it cannot contain $S_1$ or $S_2$.  
\end{proof}

\subsection{Some families of polyominoes that are not classes}
Unfortunately not all the families of polyominoes are a polyomino class. Unlike $L$-convex polyominoes, 
$2$-convex polyominoes do not form a polyomino class. Indeed, 
 the $2$-convex polyomino in Figure~\ref{fig:2pc3}$(a)$ contains the
$3$-but-not-$2$-convex polyomino $(b)$ as a submatrix. 
Similarly, the set of $k$-convex polyominoes is not a polyomino class, for $k \geq 2$.

\begin{figure}[htp]
\begin{center}
\includegraphics[width=7cm]{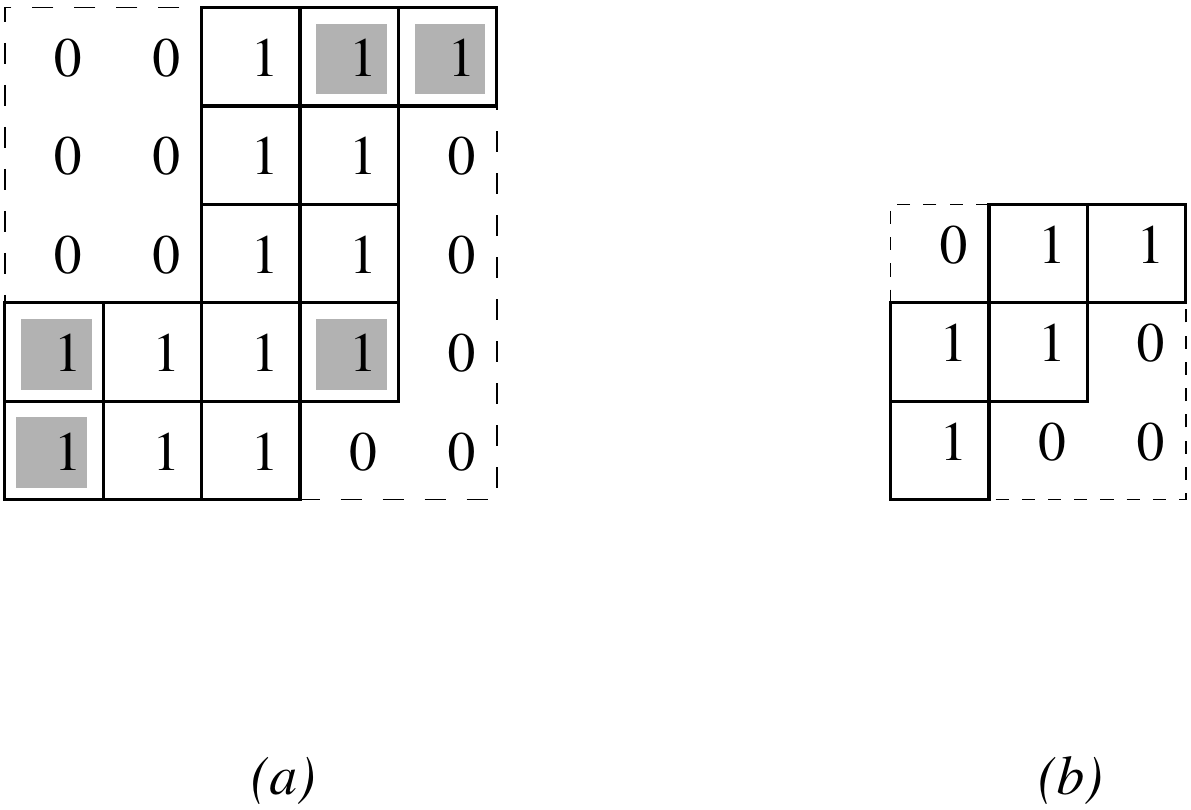}
\caption{$(a)$ a $2$-convex polyomino $P$; $(b)$ a submatrix of $P$
that is not a $2$-convex polyomino.} \label{fig:2pc3}
\end{center}
\end{figure}

In practice, this means that $2$-convex polyominoes cannot be described in terms of pattern avoidance. 
In order to be able to represent $2$-convex polyominoes we extend the notion of pattern avoidance, introducing the {\em generalized pattern avoidance}. Our extension consists in imposing the adjacency of two columns or rows by introducing special symbols, i.e. vertical/horizontal lines: with $A$ being a pattern, a vertical line between two columns of $A$, $c_i$ and $c_{i+1}$ (a horizontal line between two rows $r_i$
and $r_{i+1}$), will read that $c_i$ and $c_{i+1}$ (respectively $r_i$ and $r_{i+1}$) must be adjacent. When the vertical (resp. horizontal) line is external, it means that the adjacent column (resp. row) of the pattern must touch the minimal bounding rectangle of the polyomino. Moreover, we will use the $*$ symbol to denote $0$ or $1$ indifferently.

\begin{proposition}\label{2c}
The class of $2$-convex polyominoes can be described by the avoidance of  the following generalized patterns:

\begin{center}
\includegraphics[width=12cm]{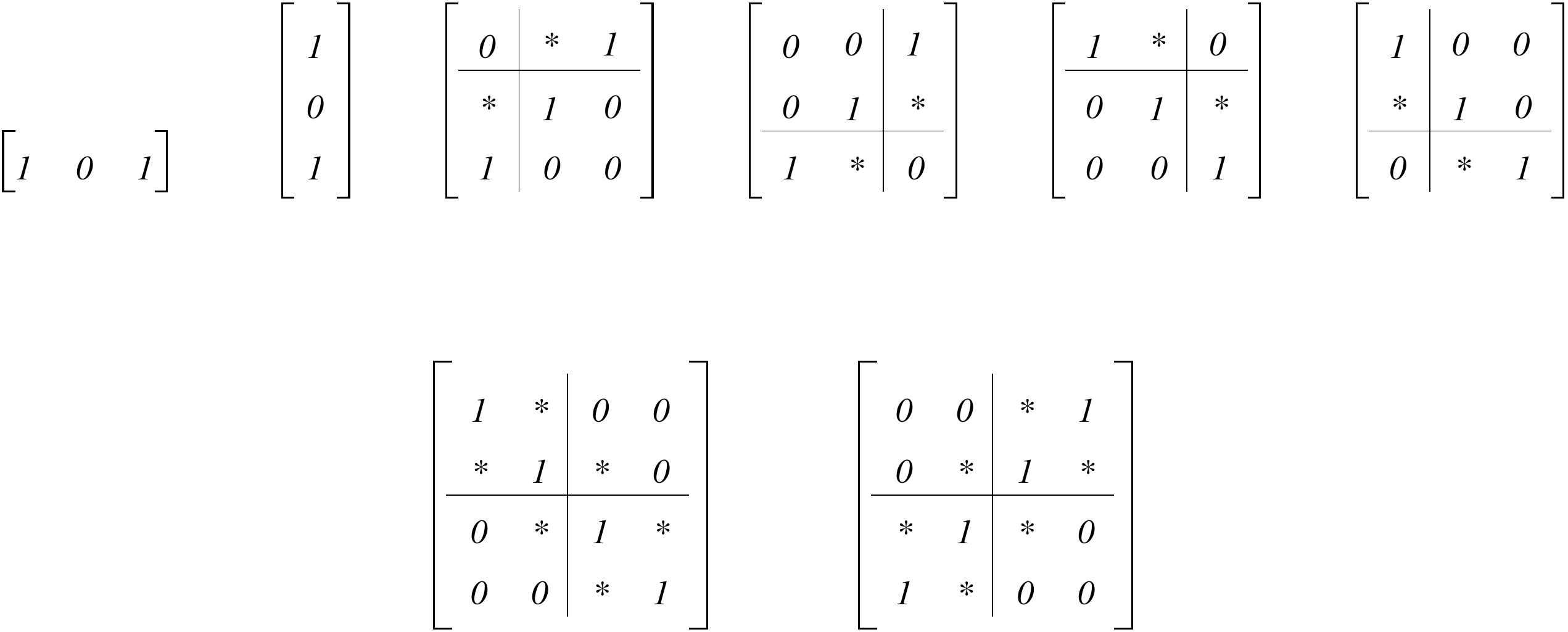}
\end{center}

\end{proposition}

Before to providing the proof of Proposition~\ref{2c} we need to recall some facts which will be useful.
\begin{remark}\label{rem:2c}
 In a $2$-convex polyomino, due to the convexity constraints,we have that for each two cells, there is a monotone path connecting them, which uses only two types of steps among $n, s, e, w$, see Section~\ref{subsec:kconv}. More precisely, after the first change of direction the two types of steps are determined.
 
 Another important fact is that, given two cells of a polyomino $c_1$ and $c_2$, the minimal number of changes of direction to go from $c_1$ to $c_2$ can be obtained studying two paths, the ones starting  with a vertical/horizontal step, in which every side has maximal length.
 
\end{remark}

\begin{proof}
Let $\cal M$ be the set of generalized patterns of Proposition~\ref{2c} and let $P$ be a polyomino. 

\noindent ($\Rightarrow$) If $P$ is a $2$-convex polyomino then $P$ avoids $\cal M$.

Let us assume that this is not true, then $P$ is a $2$-convex polyomino but it contains one of the patterns of $\cal M$. Clearly $P$ avoids the two patterns $H$ and $V$ otherwise it would not be a convex polyomino. For simplicity sake, we can consider only two patterns of $\cal M$, for instance 
$$Z_1=\left[\begin{tabular}[]{c|cc}
0 & * & 1\\
\hline
* & 1 & 0\\
1 & 0 & 0
\end{tabular}\right]
\,\,\,\,\mbox{and}\,\,\,\,
Z_2=\left[\begin{tabular}[]{cc|cc}
0 & 0 & * & 1\\
0 & * & 1 & *\\
\hline
* & 1 & * & 0\\
1 & * & 0 & 0
\end{tabular}\right]
\,\,,$$
since the remaining patterns are just the rotations of the previous ones.

If $P$ contains $Z_1$ then it has to contain a submatrix $P'$ of this type

$$\begin{array}{ccccccc}
   0 & * & . & . & . & . & 1 \\
   * & 1 & . & . & . & . & 0 \\
   . & . & . & . & . & . & . \\
   . & . & . & . & . & . & . \\
   . & . & . & . & . & . & . \\
   1 & 0 & . & . & . & . & 0 
  \end{array}
\,\,\,,$$
where the $0, 1, *$ are the elements of $Z_1$ and the dots can be replaced by $0, 1$ indifferently, clearly in agreement with the convexity and polyomino constraints.

Among all the polyominoes which can be obtained from $P'$, the one having the minimal convexity degree is the one, called $\overline{P'}$, where we have replaced any dot with a $1$ entry. So, given 
$$\overline{P'}=\begin{array}{ccccccc}
   0 & 1 & 1 & 1 & 1 & 1 & 1 \\
   1 & 1 & 1 & 1 & 1 & 1 & 0 \\
   1 & 1 & 1 & 1 & 1 & 1 & 0 \\
   1 & 1 & 1 & 1 & 1 & 1 & 0 \\
   1 & 1 & 1 & 1 & 1 & 1 & 0 \\
   1 & 0 & 0 & 0 & 0 & 0 & 0 
  \end{array}
\,\,\,,$$
it is easy to verify that the minimal number of changes of direction requested to run from the leftmost lower cell of $\overline{P'}$ to the rightmost bottom cell of $\overline{P'}$ is three, so $\overline{P'}$ is a $3$-convex polyomino, then we reach our goal.

Similarly, If $P$ contains $Z_2$ then it has to contain a submatrix $P'$ of this type

$$\begin{array}{cccccccccc}
   0 & . & . & . & \underline{0}| & * & . & . & . & 1 \\
   . & . & . & . & . & . & . & . & . & . \\
   . & . & . & . & . & . & . & . & . & . \\
   \underline{0}| & . & . & . & * & 1 & . & . & * \\
   * & . & . & . & 1 & * & . & . & . & \underline{0}| \\
   . & . & . & . & . & . & . & . & . & .  \\
   . & . & . & . & . & . & . & . & . & . \\
   1 & . & . & . & * & \underline{0}| & . & . & . & 0 \\
  \end{array}
\,\,\,,$$
where the vertical/horizontal lines have been drawn to mean that in this position there is a change of direction.
Also in this case we can consider the polyomino $\overline{P'}$, in which we have replaced any dot with a $1$ entry. It is easy to verify that the minimal number of changes of direction requested to run from the leftmost lower cell of $\overline{P'}$ to the rightmost bottom cell of $\overline{P'}$ is three, so $\overline{P'}$ and as consequence $P$ are a $3$-convex polyominoes against the hypothesis.

\smallskip

\noindent ($\Leftarrow$) If $P$ avoids $\cal M$ then $P$ is a $2$-convex polyomino.

Let us assume that, on the contrary, $P$ avoids $\cal M$ but it is a $3$-convex polyomino, \ie there exist two cells of $P$, $c_1$ and $c_2$, such that any paths from $c_1$ to $c_2$ requires at least three changes of direction.

Let us take into consideration the two paths, defined in Remark~\ref{rem:2c}, which use only steps of type $n$ and $e$ to prove that $P$ contains at least one of the patterns of $\cal M$. We have to analyze the following situations:
 \begin{itemize}
  \item [-] the two paths running from $c_1$ to $c_2$ are distinct, see Figure~\ref{fig:possiblePath}~$(a)$;
  \item [-] one of the paths running from $c_1$ to $c_2$ does not exist, see Figure~\ref{fig:possiblePath}~$(b)$;
  \item [-] the two paths running from $c_1$ to $c_2$ coincide after the first change of direction, see Figure~\ref{fig:possiblePath}~$(c)$;
  \item [-] the two paths running from $c_1$ to $c_2$ coincide after the second change of direction, see Figure~\ref{fig:possiblePath}~$(d)$.
 \end{itemize}

\begin{figure}[htp]
\begin{center}
\includegraphics[width=14cm]{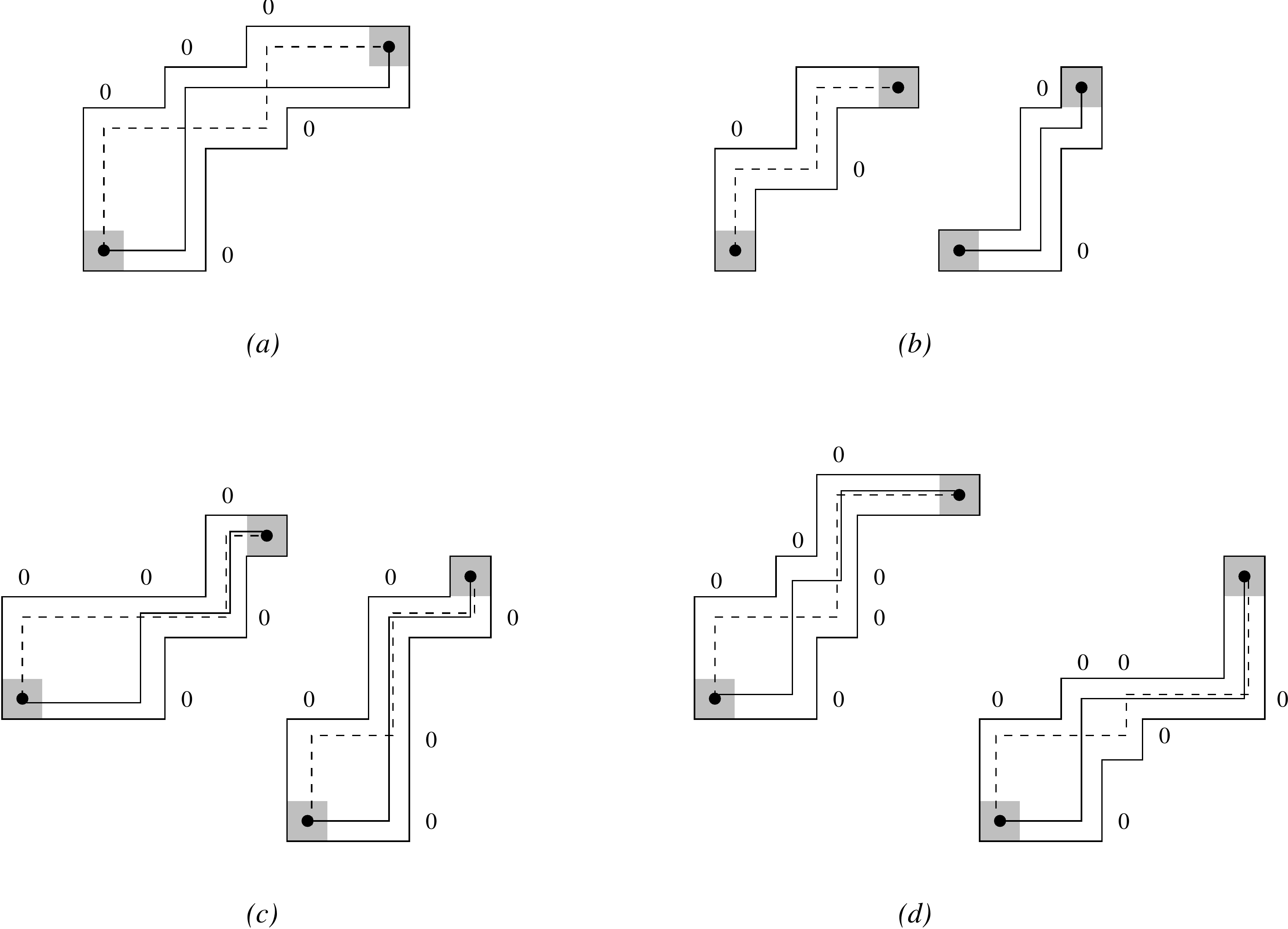}
\caption{The possible monotone paths connecting the cells $c_1$ and $c_2$, those cells are the ones greyed. $(a)$ the two paths are distinct; $(b)$  one of the paths does not exist; $(c)$ the two paths running from $c_1$ to $c_2$ coincide after the first change of direction;$(d)$ the two paths running from $c_1$ to $c_2$ coincide after the second change of direction.} \label{fig:possiblePath}
\end{center}
\end{figure}
Here, we will consider only the first situation (which is the most general), because in all the other cases we can use an analogous reasoning.

So, we have that the two cells $c_1$ and $c_2$ are connected by two distinct paths, see Figure~\ref{fig:possiblePath}~$(a)$, then $P$ have to contain a submatrix $P'$ of the following type

$$\begin{array}{ccccc|ccccc}
   0 & . & . & . & 0 & 1 & . & . & . & 1 \\
   . & . & . & . & . & . & . & . & . & . \\
   . & . & . & . & . & . & . & . & . & . \\
   0 & . & . & . & 1 & 1 & . & . & . & 1 \\
   \hline
   1 & . & . & . & 1 & 1 & . & . & . & 0 \\
   . & . & . & . & . & . & . & . & . & .  \\
   . & . & . & . & . & . & . & . & . & . \\
   1 & . & . & . & 1 & 0 & . & . & . & 0 \\
  \end{array}
\,\,\,,$$
where the horizontal and vertical lines, which indicate the adjacency constraints, have been placed to impose the changes of direction.

It is easy to see that the submatrix obtained from $P'$ deleting all the rows and columns containing dots, is one among the various that we can obtain replacing appropriately the symbol $*$ in the pattern $Z_2$. So, $P$ contains $Z_2$ against the hypothesis.

\end{proof}
We can point out that the pattern $Z_1$, and its rotation, can be obtained from the pattern $Z_2$ (or its rotation) replacing appropriately the $*$ entries, but we need to consider them in order to exclude the polyominoes, of dimension $n\times m$, with one among $n$ and $m$ less than $4$, that are $3$-convex.

Let us just observe, referring to Fig.~\ref{fig:2pcpp_d}, that the pattern~$(c)$ is not contained in the $2$-convex polyomino~$(a)$, but it is contained in the $3$-convex polyomino~$(d)$. 

\begin{figure}[htp]
\begin{center}
\includegraphics[width=12cm]{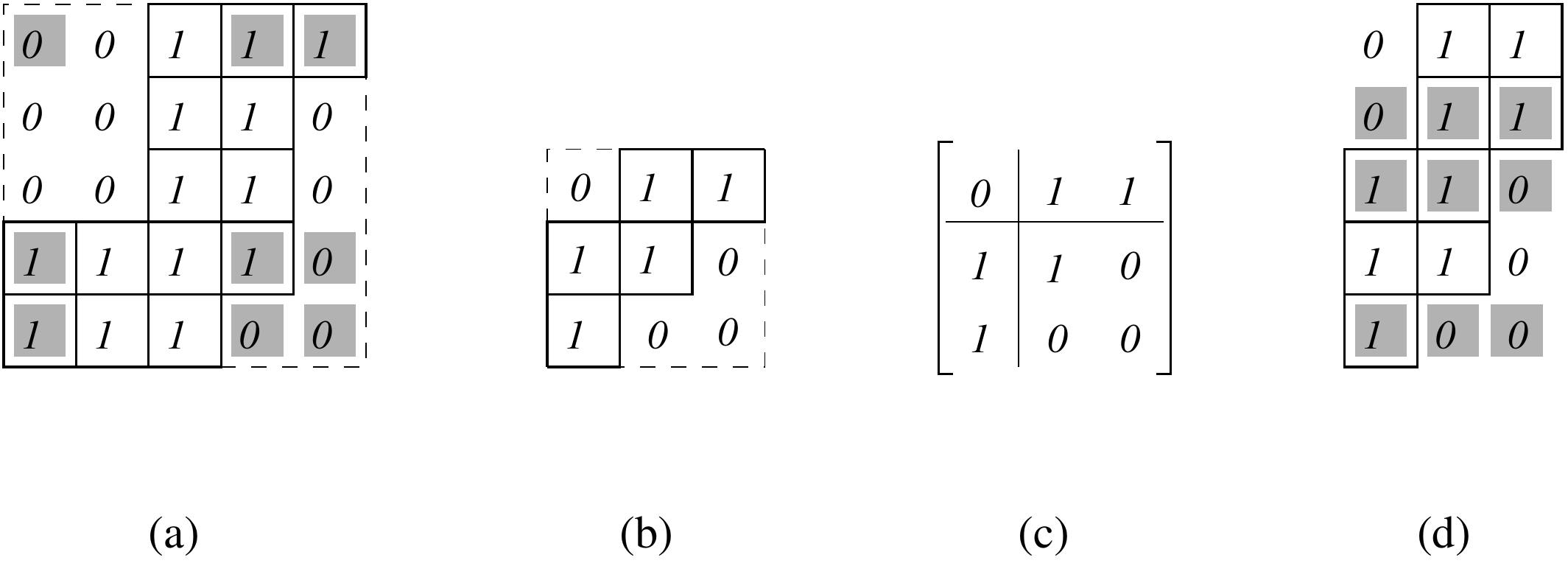}
\caption{$(a)$ a $2$-convex polyomino $P$; $(b)$ a pattern of $P$ that is not a $2$-convex polyomino; $(c)$ a generalized pattern, which is not contained in $(a)$, but is contained in the $3$-convex polyomino (non $2$-convex) $(d)$.}
\label{fig:2pcpp_d}
\end{center}
\end{figure}

It is possible generalize the previous result and give a characterization of the class of $k$-convex polyominoes, with $k> 2$, using generalized patterns.

As we observed in~\ref{par:dirConvex} is not possible to describe the set of directed polyominoes in terms of submatrix avoidance, but also in this case the introduction of generalized patterns will be useful.

\begin{proposition}\label{2c}
The class of directed polyominoes can be represented as the class of polyominoes avoiding the following patterns

\begin{center}
\includegraphics[width=4cm]{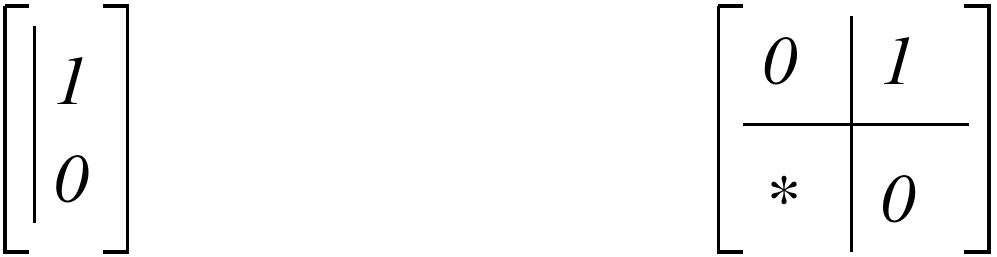}
\end{center}

\end{proposition}

\begin{proof}
 Basing on the definition of directed polyomino and recalling the proof of Proposition~\ref{prop:m-basis_directed_convex}, we can reach our goal using reasonings analogous to Proposition~\ref{2c}.
\end{proof}

We would like to point out that there are families of polyominoes which cannot be described, even using generalized pattern avoidance. For instance, the reader can easily check that one of these families is that of polyominoes having a square shape.

\subsection{Defining new polyomino classes by submatrix avoidance}
In addition to characterizing known classes, the approach of submatrix avoidance may be used to define new classes of polyominoes, the main question being then to give a combinatorial/geometrical characterization of these classes.  We present some examples of such classes, with simple characterizations and interesting combinatorial properties. These examples illustrate that the submatrix avoidance approach in the study of families of polyominoes is promising.

\paragraph{$L$-polyominoes.} Proposition~\ref{prop:lconvex}
states that $L$-convex polyominoes can be characterized
by the avoidance of four matrices: $H$ and $V$, which impose the 
convexity constraint; and $S_1$ and $S_2$, which account for the
$L$-property, or equivalently (by Theorem~\ref{th:ryser}) indicate 
the uniqueness of the polyomino w.r.t its horizontal and vertical
projections. So, it is quite natural to study the
class $\Avp(S_1,S_2)$, which we call the class of {\em $L$-polyominoes}. 
From Theorem~\ref{th:ryser}, it follows that: 

\begin{proposition}
Every $L$-polyomino is uniquely determined by its horizontal and vertical projections.
\end{proposition}

From a geometrical point of view, the $L$-polyominoes can be
characterized using the concept of (geometrical) inclusion between rows (resp.
columns) of a polyomino. 
For any polyomino $P$ with $n$ columns, and any rows $r_1=(r_{1;1} \ldots r_{1;n})$, $r_2=(r_{2;1} \ldots r_{2;n})$ of the matrix representing $P$, we say that $r_1$ is {\em geometrically included} in $r_2$ (denoted $r_1 \leqslant r_2$) if, for all $1 \leq i \leq n$ we have that $r_{1;i}=1$ implies $r_{2;i}=1$. Geometric inclusion of columns is defined analogously. 
Two rows (resp. columns) $r_1, r_2$ (resp. $c_1, c_2$) of a polyomino $P$ are said to be {\em comparable} if $r_1\leqslant r_2$ or $r_2\leqslant r_1$ (resp. $c_1\leqslant c_2$ or $c_2\leqslant c_1$). These definitions are illustrated in Figure~\ref{bicentered}. 

The avoidance of $S_1$ and $S_2$ has an immediate interpretation in geometric terms, proving that: 

\begin{proposition}
The class of $L$-polyominoes coincides with the set of the
polyominoes where every pair of rows (resp. columns) are comparable.
\end{proposition}

\begin{figure}[htd]
\begin{center}
\includegraphics[width=9cm]{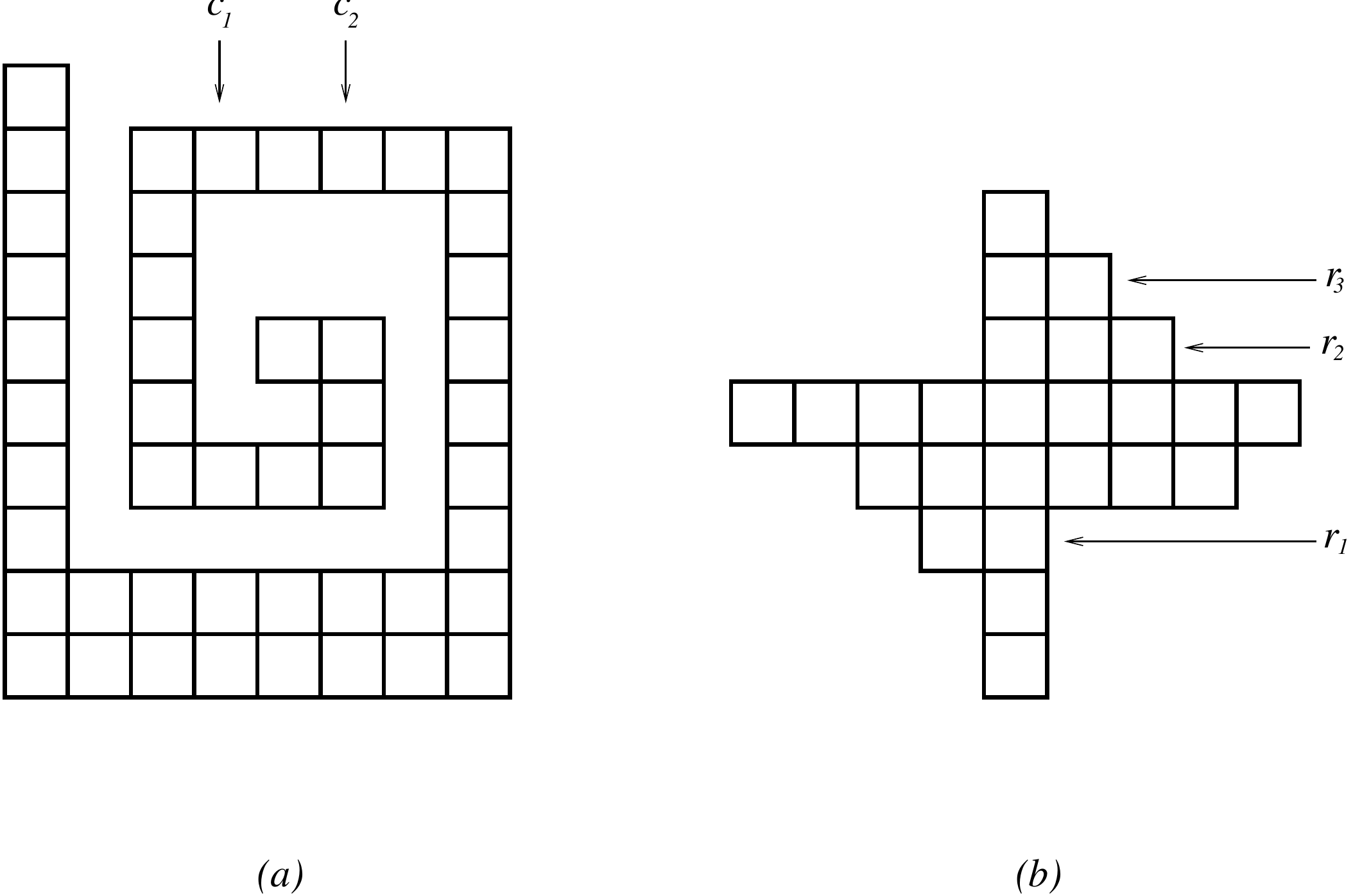}
\caption{$(a)$ a $L$-polyomino, where the reader can check that
every pair of  rows and columns are comparable; for instance,
$c_1\leqslant c_2$; $(b)$ an example of a polyomino which is not an
$L$-polyomino, where the row $r_1$ is not comparable both with  rows
$r_2$ and $r_3$.} \label{bicentered}
\end{center}
\end{figure}

We leave open the problem of studying further the class of
$L$-polyominoes, in particular from an enumerative point of view 
(enumeration w.r.t. the area or the semi-perimeter).

\paragraph*{ The class $\Avp(H', V') $ with $H' = {\footnotesize \left[\begin{array}{c}
          0 \\
          1 \\
          0 
         \end{array}
  \right] }$ and $V' = {\footnotesize\, \left[\begin{array}{ccc}
          0  & 1 & 0 
         \end{array}
  \right] }$.} ~ \\
  
By analogy\footnote{which essentially consists in exchanging $0$ and $1$ in the excluded submatrices} with the class of convex polyominoes (characterized by the avoidance of $H$ and $V$), we may consider the class 
${\cal C}'$ of polyominoes avoiding the two submatrices $H'$ and $V'$ defined above. 
In some sense these objects can be viewed as a dual class to convex polyominoes. Figure \ref{fig:010}$(a)$ shows a polyomino in ${\cal C}'$.

\begin{figure}[htd]
\begin{center}
\includegraphics[width=10cm]{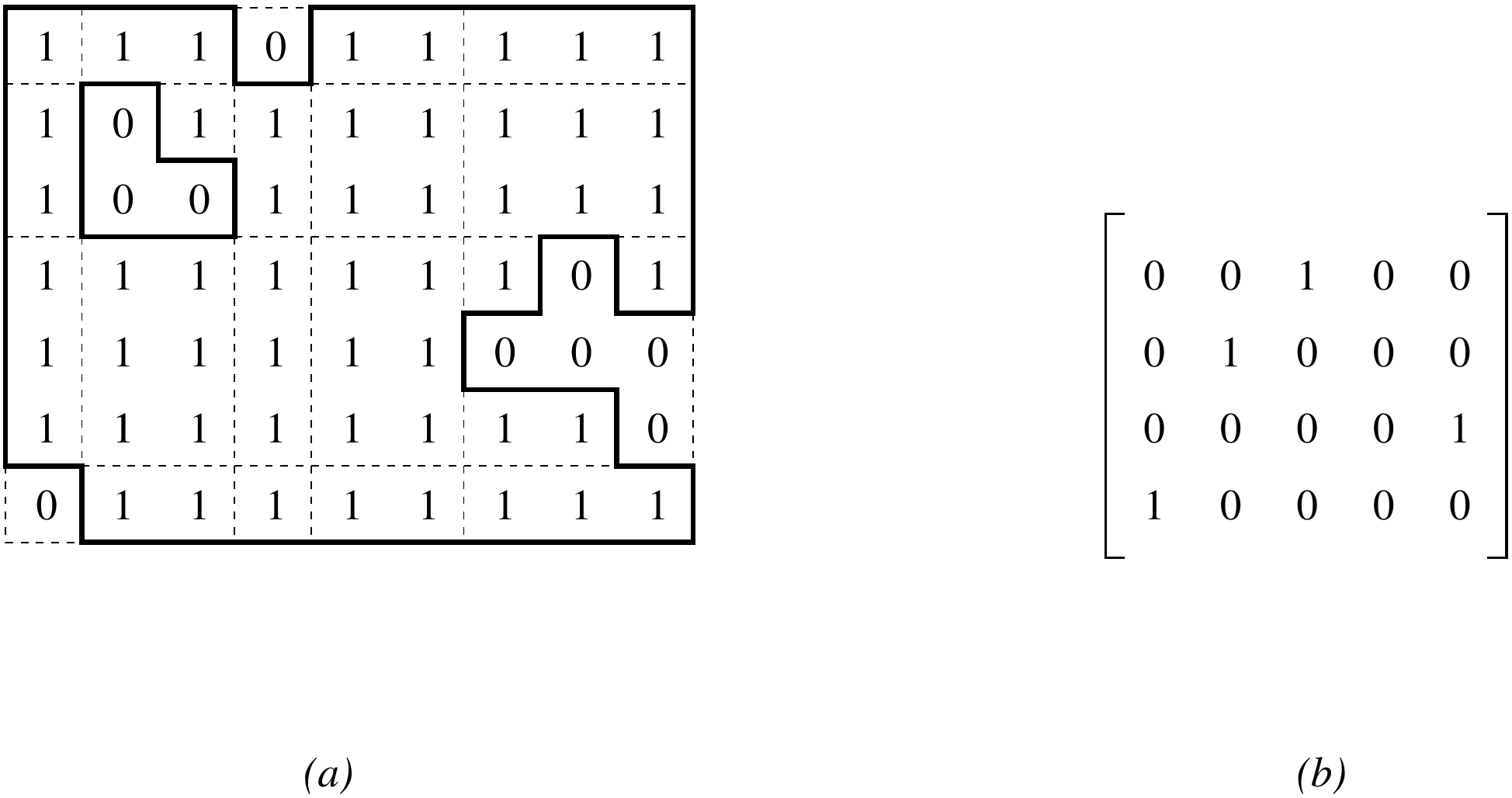}
\caption{$(a)$ a polyomino in ${\cal C}'$ and its decomposition; $(b)$ the corresponding quasi-permutation matrix.}
\label{fig:010}
\end{center}
\end{figure}

The avoidance of $H'$ and $V'$ has a straightforward geometric interpretation, giving immediately that: 

\begin{proposition}
A polyomino $P$ belongs to ${\cal C}'$ if and only if every connected set of cells of maximal length in a row (resp. column) has a contact with the minimal bounding rectangle of $P$. 
\label{prop:contact}
\end{proposition}

The avoidance of $H'$ and $V'$ also ensures that in a polyomino of ${\cal C}'$ every connected set of $0$s has the shape of a convex polyomino, which we call -- by abuse of notation -- a \emph{convex $0$-polyomino} (contained) in $P$. Each of these convex  $0$-polyominoes has a minimal bounding rectangle, which  individuates an horizontal (resp. vertical) strip of cells in $P$, where no other convex $0$-polyomino of $P$ can be found. 
Therefore every polyomino $P$ of ${\cal C}'$ can be uniquely decomposed  in regions of two types: rectangles all made of $1$s (of type $A$) or rectangles bounding a convex $0$-polyomino (of type $B$). Then, we can map $P$ onto a quasi-permutation matrix as follows: each rectangle of type $A$ is mapped onto a $0$, and each rectangle of type $B$ is mapped onto a $1$. See an example in Figure~\ref{fig:010}$(a)$. 

Although this representation is clearly non unique, we believe it may be used for the enumeration of ${\cal C}'$. 
For a start, it provides a simple lower bound on the number of polyominoes in ${\cal C}'$ whose bounding rectangle is a square. 

\begin{proposition}\label{prop:stanley_wilf} 
Let $c'_n$ be the number of polyominoes in ${\cal C}'$ whose bounding rectangle is an $n\times n$ square. 
For $n\geq 2$, $c'_n \geq \lfloor \frac{n}{2} \rfloor \, ! \, .$
\end{proposition}

\begin{proof}
The statement directly follows from a mapping from permutations of size  $m \geq 1$ to polyominoes in ${\cal C}'$ whose bounding rectangle is an $2m\times 2m$ square, and defined as follows.  
From a permutation $\pi$, we replace every entry of its permutation matrix by a $2\times 2$ matrix according to the following rules. \\
Every $0$ entry is mapped onto a $2 \times 2$ matrix of type $A$. \\
The $1$ entry in the leftmost column is mapped onto ${\footnotesize \left[\begin{array}{cc}
          0 & 1\\
          0 & 0 
         \end{array}\right]}$. \\
The $1$ entry in the topmost row (if different) is mapped onto ${\footnotesize \left[\begin{array}{cc}
          0 & 0\\
          1 & 0
         \end{array}\right]}$. \\
Every other $1$ entry is mapped onto ${\footnotesize \left[\begin{array}{cc}
          1 & 0\\
          0 & 0 
         \end{array}\right]}$. 
         
This mapping (illustrated in Figure~\ref{fig:inv}) guarantees that the set of cells obtained is connected (hence is a polyomino), and avoids the submatrices $H'$ and $V'$, concluding the proof.
\end{proof}

\begin{figure}[htd]
\begin{center}
\includegraphics[width=9cm]{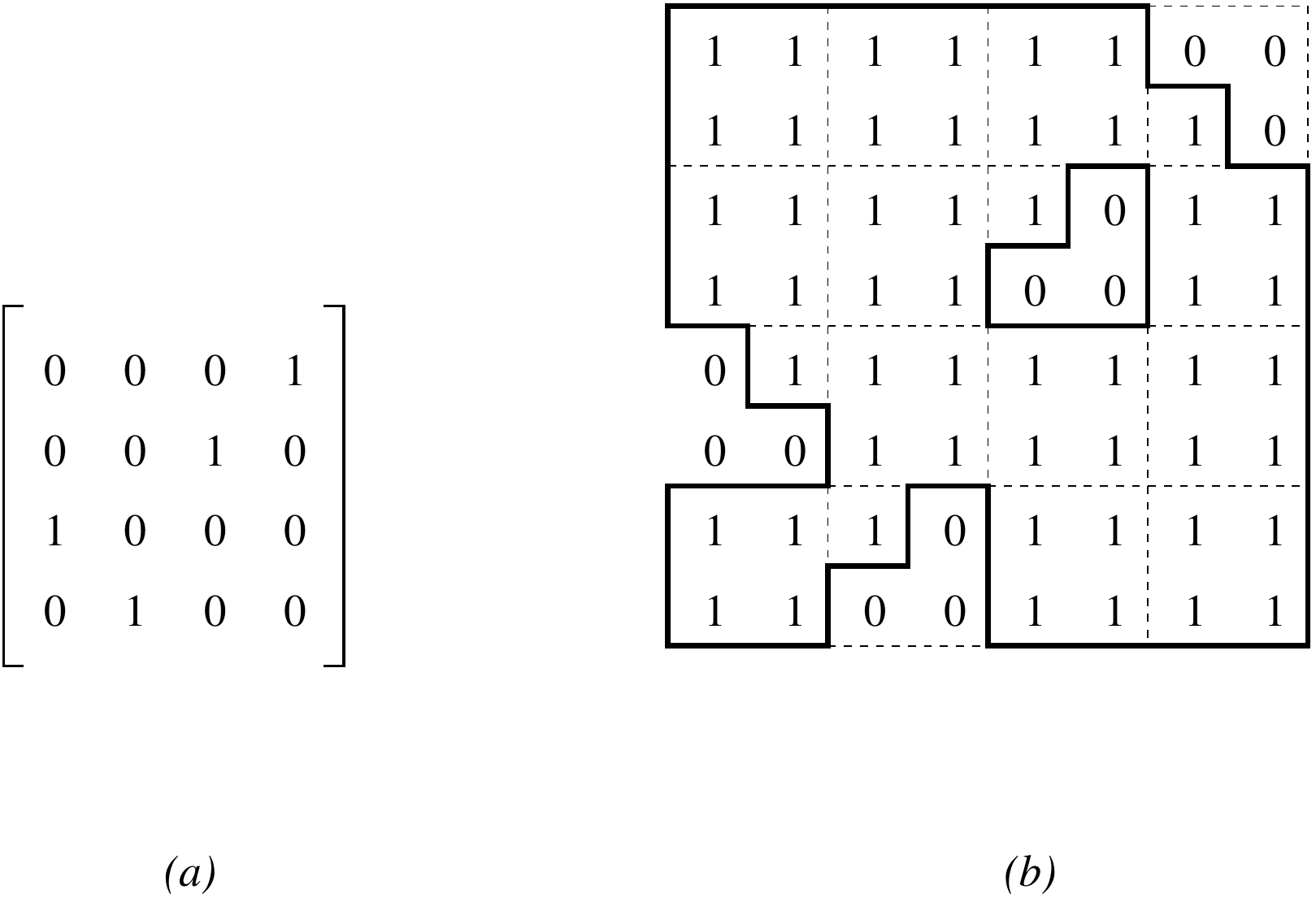}
\caption{$(a)$ a permutation matrix of dimension $4$; $(b)$ the corresponding polyomino of dimension $8$ in ${\cal C}'$.}
\label{fig:inv}
\end{center}
\end{figure}

\paragraph{Polyominoes avoiding rectangles.} Let $O_{m,n}$ be set of rectangles -- binary pictures  with all the entries equal to $1$ -- of dimension $m \times n$ (see Figure \ref{serp}~$(a)$). With $n=m=2$ these objects (also called {\em snake-like polyominoes}) have a simple geometrical characterization.

\begin{figure}[htd]
\begin{center}
\includegraphics[width=8cm]{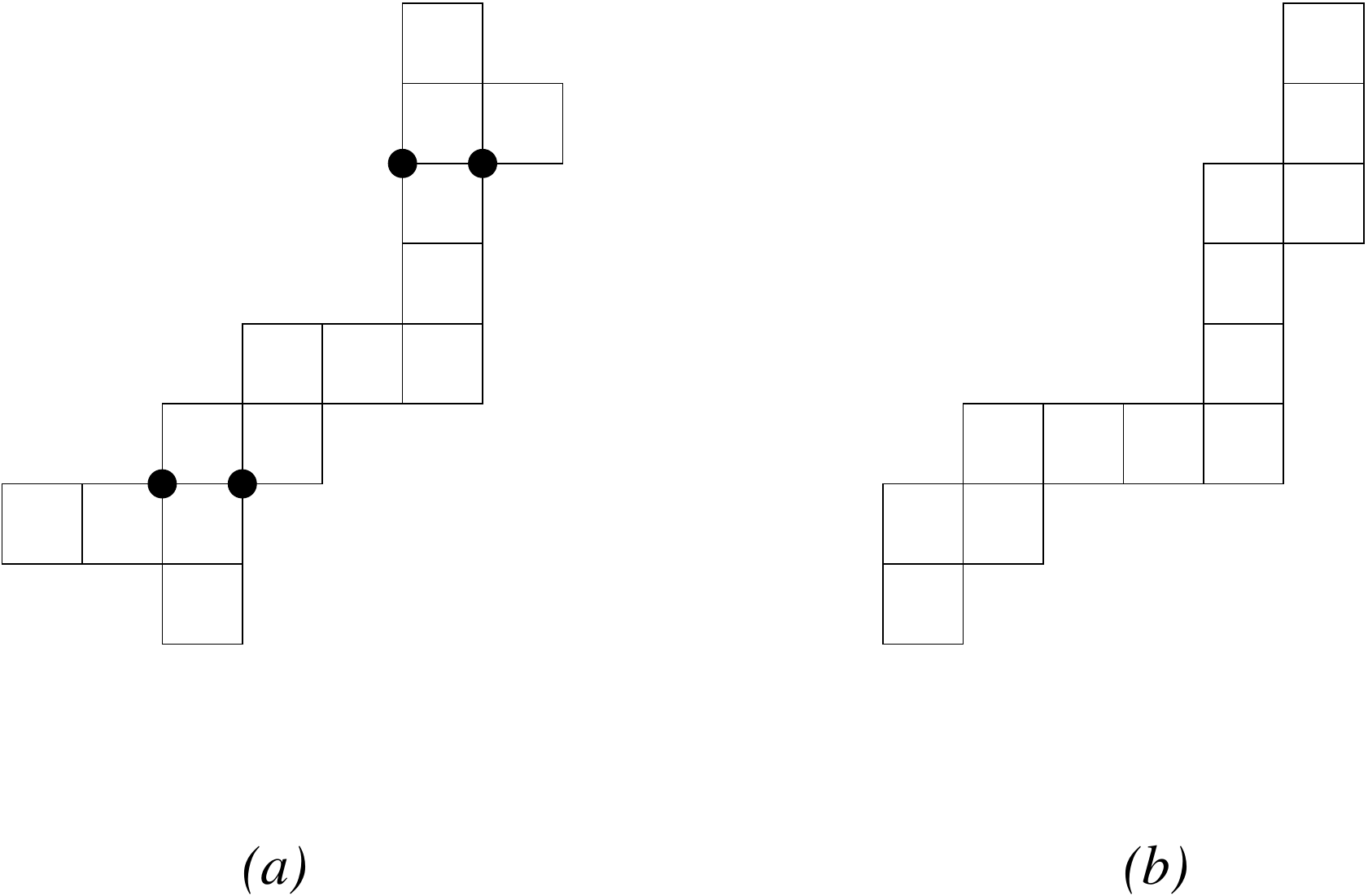}
\caption{$(a)$ a snake-like polyomino; $(b)$ a snake.}
\label{serp}
\end{center}
\end{figure}

\begin{proposition}
Every snake-like polyomino can be uniquely decomposed into three parts: a
unimodal staircase polyomino oriented with respect to two axis-parallel directions $d_1$ and $d_2$ and
two (possibly empty) $L$-shaped polyominoes placed at the extremities of the staircase.
These two $L$-shaped polyominoes have to be oriented with respect
to $d_1, d_2$. 
\end{proposition} 

We have studied the classes $\Avp(O_{m,n})$, for other values of $m,n$, obtaining similar characterizations which here are omitted for brevity.

\paragraph{Snakes.}
Let us consider the family of {\em snake-shaped} polyominoes (briefly, {\em snakes}) -- as that shown in Fig.~\ref{serp}~$(b)$:

\begin{proposition}
The family of snakes is a polyomino class, which can be described by the avoidance of the following polyomino patterns:

\begin{center}
\includegraphics[width=12cm]{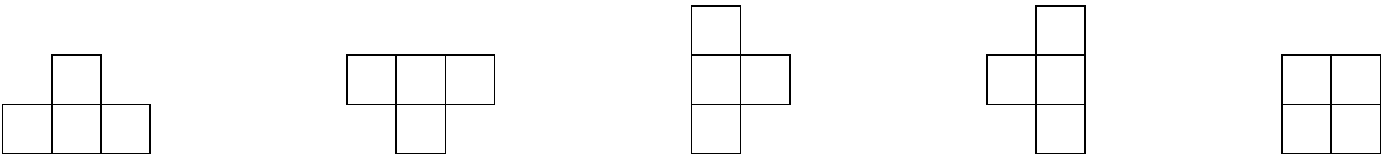}
\end{center}
\end{proposition}

\paragraph{Hollow stacks.} Let us recall that a stack polyomino is a convex polyomino containing two adjacent corners of its minimal bounding rectangle (see Fig.~\ref{hollow}~$(a)$). Stack polyominoes clearly form a polyomino class, described by the avoidance of the patterns:
\begin{center}
\includegraphics[width=4.5cm]{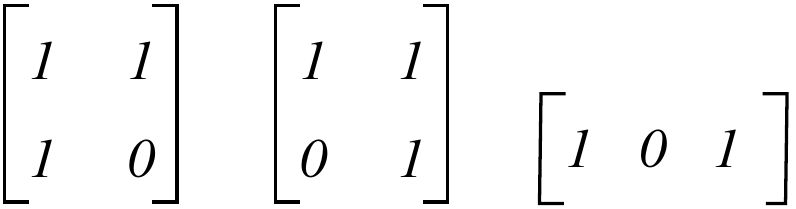}
\end{center}
A {\em hollow stack} (polyomino) is a polyomino obtained from a stack polyomino $P$ by removing from $P$ a stack polyomino $P'$ which is geometrically contained in $P$ and whose basis lie on the basis of the minimal bounding rectangle of $P$. Figure~\ref{hollow}~$(b)$, $(c)$  depict two hollow stacks. 

\begin{figure}[htd]
\begin{center}
\includegraphics[width=11cm]{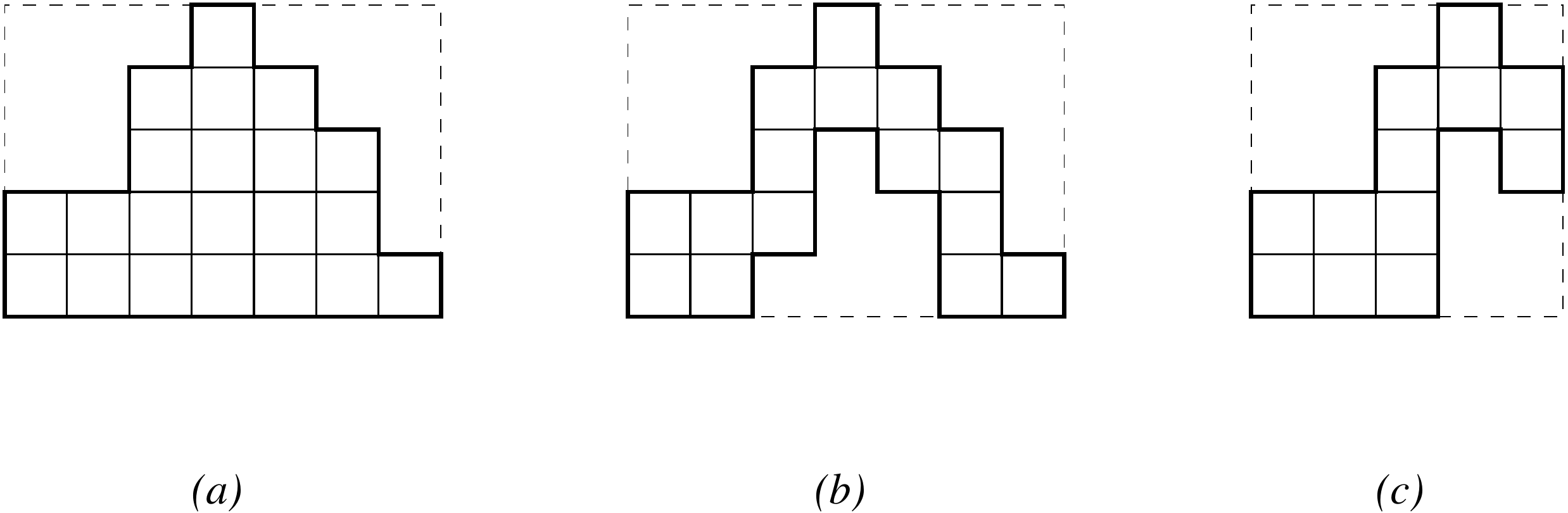}
\caption{$(a)$ a stack polyomino; $(b)$, $(c)$: hollow stacks.}
\label{hollow}
\end{center}
\end{figure}

\begin{proposition}
The family $\cal H$ of hollow stack polyominoes forms a polyomino class with $p$-basis given by:
\begin{center}
\includegraphics[width=4.5cm]{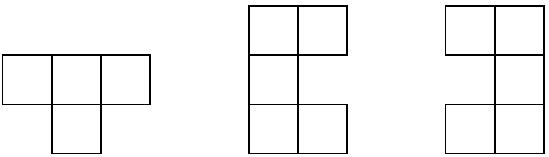}
\end{center}

\end{proposition}
%

\paragraph{Rectangles with rectangular holes.} Let $\cal R$ be the class of polyominoes obtained from a rectangle by removing sets of cells which have themselves a rectangular shape, and such that there is no more than one connected set of $0$'s for each row and column. The family $\cal R$ can easily be proved to be a polyomino class, and moreover:

\begin{figure}[htd]
\begin{center}
\includegraphics[width=4.5cm]{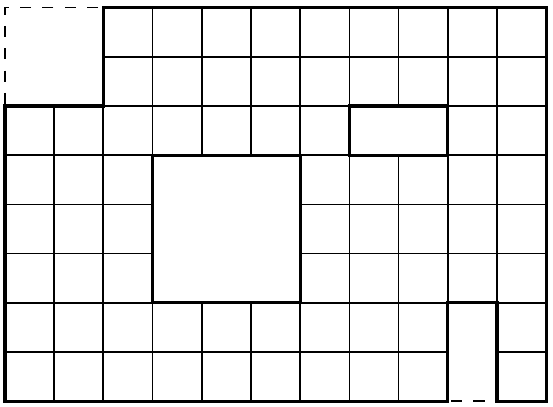}
\caption{A rectangle with rectangular holes.}
\label{hollow}
\end{center}
\end{figure}

\begin{proposition}
The class $\cal R$ can be described by the avoidance of the patterns:
$$\left[ \begin{array}{ccc}
0 & 1 & 0 \end{array} \right], \, 
\left[ \begin{array}{c}
0 \\
1 \\
0 \end{array} \right] \, 
\left[ \begin{array}{cc}
1 & 0 \\
0 & 0 \end{array} \right] \, 
\left[ \begin{array}{cc}
0 & 1 \\
0 & 0 \end{array} \right] \, 
\left[ \begin{array}{cc}
0 & 0 \\
1 & 0 \end{array} \right] \, 
\left[ \begin{array}{cc}
0 & 0 \\
0 & 1 \end{array} \right] \, .$$
\end{proposition}

\section{Some directions for future research}
\label{sec:further_research}

Our work opens numerous and various directions for future research. We introduced a new approach of submatrix avoidance in the study of permutation and polyomino classes. 

\smallskip

In both cases permutation and polyomino classes, 
we have described several notions of bases for these classes: $p$-basis, $m$-bases, canonical $m$-basis, minimal $m$-bases. 
Section~\ref{sec:from_one_basis_to_another} explains how to describe the $p$-basis from any $m$-basis. 
Conversely, we may ask how to transform the $p$-basis into an ``efficient'' $m$-basis. Of course, the $p$-basis is itself an $m$-basis, but we may wish to describe the canonical one, or a minimal one. 

Many questions may also be asked about the canonical and minimal $m$-bases themselves. 
For instance: 
When does a class have a unique minimal $m$-basis? 
Which elements of the canonical $m$-basis may belong to a minimal $m$-basis? 
May we describe (or compute) the minimal $m$-bases from the canonical $m$-basis? 

Finally, we can study the classes for which the $p$-basis is itself a minimal $m$-basis of the class (see the examples of the polyomino classes of vertical bars, or of parallelogram polyominoes). 

\smallskip

Submatrix avoidance in 	permutation classes has allowed us to derive a statement (Corollary~\ref{cor:WE}) from which infinitely many Wilf-equivalences follow. 
Such general results on Wilf-equivalences are rare in the permutation patterns literature, and it would be interesting to explore how much further we can go in the study of Wilf-equivalences with the submatrix avoidance approach.

\smallskip

The most original concept of this work is certainly the introduction of the polyomino classes, which opens many directions for future research. 

One is a systematic study of polyomino classes defined by pattern avoidance. 
Because enumeration is the biggest open question about polyominoes, we should study the enumeration of such classes, and see whether some interesting bounds can be provided. Notice that the Stanley-Wilf-Marcus-Tardos theorem~\cite{marcTard} on permutation classes implies that the permutations in any given class represent a negligible proportion of all permutations. We don't know if a similar statement holds for polyomino classes. 

As we have reported in Section~\ref{sec:papc}, the poset $({\poly}, \polypattern )$ of polyominoes was introduced in \cite{CR}, where the authors proved that it is a ranked poset, and contains infinite antichains. There are however some combinatorial and algebraic properties of this poset which are still to explore, in particular w.r.t. characterizing some simple intervals in this poset.  

\smallskip

Finally, we have used binary matrices to import some questions on permutation classes to the context of polyominoes. But a similar approach could be applied to any other family of combinatorial objects which are represented by binary matrices.

\pagestyle{plain}

\phantomsection
\addcontentsline{toc}{chapter}{Bibliography}


\end{document}